\documentclass[a4paper,10pt]{article}
\usepackage[utf8]{inputenc}

\usepackage{geometry}
\newtheorem{lemma}{Lemma}
\newtheorem{proof}{Proof}
\newtheorem{theorem}{Theorem}

\usepackage{lipsum}
\usepackage{amsfonts}
\usepackage{graphicx}
\usepackage{epstopdf}
\usepackage{algorithmic}

\usepackage{amsopn}

\usepackage[english]{babel}
\usepackage{moreverb}

\usepackage{yfonts}

\usepackage{lineno}
\usepackage{amsmath, bm}

\usepackage{numcompress}
\usepackage{caption}
\usepackage{booktabs}

\usepackage{hyperref}
\usepackage{color}
\usepackage{multirow}

\usepackage{amssymb}
\usepackage{latexsym}
\usepackage{fancyhdr}
\usepackage{graphicx}
\usepackage{newlfont}
\usepackage{textcomp}
\usepackage{float}
\usepackage{multirow}

\usepackage{enumitem}

\usepackage{color}

\def\BINS{\mathcal B}
\def\BIN{\mathcal B}
\def\lll{\ell}

\title{\bf Piecewise polynomial approximation of probability density functions with application to uncertainty quantification for stochastic PDEs$^{\mbox{\footnotesize{0}}}$}
\author{%
  Giacomo Capodaglio\footnote{Department of Scientific Computing, Florida State University, Tallahassee, FL 32306-4120, USA}%
  \and Max Gunzburger$^*$
  }
\date{}

\begin{document}

\maketitle
\footnotetext{{\bf{Funding:}} This work was supported by US Air Force Office of Scientific Research grant FA9550-15-1-0001 and by the Sandia National Laboratories contract 1985151.}

\begin{abstract}
The probability density function (PDF) associated with a given set of samples is approximated by a piecewise-linear polynomial constructed with respect to a binning of the sample space. 
The kernel functions are a compactly supported basis for the space of such polynomials, {i.e. finite element hat functions, that} are centered at the bin nodes rather than at the samples, as is the case for the standard kernel density estimation approach. This feature {naturally} provides an approximation that is scalable with respect to the sample size.  On the other hand, unlike other strategies that {use a finite element approach}, the {proposed} approximation does not require the solution of a linear system. In addition, a simple rule that relates the bin size to the sample size eliminates the need for bandwidth selection procedures. 
The {proposed} density estimator has unitary integral, does not require a constraint to enforce positivity, and is consistent. The {proposed} approach is validated through numerical examples in which samples are drawn from  known PDFs. The approach is also used to determine approximations of (unknown) PDFs associated with outputs of interest that depend on the solution of a stochastic partial differential equation.
\end{abstract}

\section{Introduction}
The problem of estimating a probability density function (PDF) associated with a given set of samples is of major relevance in a variety of mathematical and statistical applications; see, e.g.,
\cite{andronova2001objective, botev2010kernel, botev2011generalized, capodaglio2018approximation, criminisi2012decision, gerber2014predicting, izenman1991review, silverman2018density, zivkovic2006efficient}. Histograms are perhaps the most popular means used in practice for this purpose. A histogram is a piecewise-constant approximation of an unknown PDF that is based on the subdivision of the sample space into subdomains that are commonly referred as bins. For simplicity, consider the case in which all bins have equal volume. The value of the histogram on any bin is given by the number of samples that lie within that bin along with a global scaling applied to ensure that the histogram has unitary integral. As an example, assume that a bounded one-dimensional sample domain $\Gamma$ has been discretized into a set $\BINS^\delta=\{\BIN_\lll\}_{\lll=1}^{N_{bins}}$ of $N_{bins}$ non-overlapping, covering, bins (i.e., intervals) of length $\delta$. Assume also that one has in hand $M$ samples $\{Y_m\}_{m=1}^M$ of an unknown PDF $f(Y)$. For every $\BIN_\lll\in{\BINS^\delta}$, the value of the histogram function ${f}^{histo}_{\delta,M}(Y)$ on $\BIN_\lll$ is given by 
\begin{align}\label{histo_def}
 {f}^{histo}_{\delta,M}(Y)\big|_{\BIN_\lll} = \dfrac{1}{\delta M} \sum_{m=1}^{M} {\mathcal X}_{\BIN_\lll}(Y_m),
 \quad \lll=1,\ldots,N_{bins},
\end{align}
where $ {\mathcal X}_{\BIN_\lll}(Y)$ denotes the indicator function for $\BIN_\lll$. The above formula easily generalizes to higher-dimensional sample domains. Although the histogram is probably the easiest, with regards to implementation, means for estimating a PDF, it suffers from some limitations. For instance, being only a piecewise-constant approximation, it is discontinuous across bin boundaries and also the use of piecewise-constant approximations severely limits the discretization accuracy one can achieve.

A {popular} alternative method capable of overcoming the differentiability issue is {\em kernel density estimation} (KDE) for which the PDF is approximated by a sum of {\em kernel} functions ${\mathcal K}(\cdot)$ centered at the samples so that a desired smoothness of the approximation can be obtained \cite{izenman1991review}. Considering again a one-dimensional sample domain $\Gamma$, the KDE approximation is defined as
\begin{align}\label{KDE}
 {f}^{kde}(Y) =  \dfrac{1}{b M} \sum_{m=1}^{M} {\mathcal K}\Big(\dfrac{Y - Y_m}{b}\Big),
\end{align}
where $b$, which is referred to as the bandwidth, usually governs the decay rate of the kernel as $|Y-Y_m|$ increases. KDE approximations have been shown to be effective in a variety of applications; see, e.g.,  \cite{gerber2014predicting, lopez2015multi, xie2008kernel,  xu2015estimating}. {A limitation of KDE is that the choice of the bandwidth strongly affects the accuracy \cite{heidenreich2013bandwidth,turlach1993bandwidth} of the approximation}. More important, the {naive KDE method of \eqref{KDE}  } does not scale well with the dimension $M$ of the sample data set, i.e., for a given $Y$, the evaluation of the KDE approximation \eqref{KDE} requires $M$ kernel evaluations so that clearly evaluating \eqref{KDE} becomes more expensive as the value of $M$ grows.
{A way to overcome this issue is performing an appropriate binning of the sample set, as for the histogram, and appropriately transferring the information from the samples to the grid points. The kernel functions are then centered at the grid points and scalability with respect to the sample size can be achieved \cite{fan1994fast,hall1996accuracy}.} 
{Another method related to KDE is the one presented in \cite{hegland2009finite}, which is based on spline smoothing and on a finite element discretization of the estimator. This method {is related to our approach because the PDF is also approximated by a finite element function}. Our method  presents several advantages compared to that of \cite{hegland2009finite}. First, to determine the coefficients, the solution of a linear system is not required. Second, our method only involves the binning size as a smoothing parameter, whereas the one in \cite{hegland2009finite} also requires the treatment of  an additional smoothing parameter $\lambda$. Third, no special constraints have to be introduced to ensure positivity of the PDF approximation. In turn, the method in \cite{hegland2009finite} has features that our approach does not provide such as the ability to match the sample moments up to a certain degree and the possibility of allowing aggregate data.
We note that in \cite{peherstorfer2014density}, sparse-grid basis functions are substituted for the standard finite element basis functions in the method of \cite{hegland2009finite}, allowing for consideration of larger $N_\Gamma$.} 

In our approach, the unknown PDF is approximated by a piecewise-polynomial function, specifically a piecewise-linear polynomial, that is defined, as are histograms, with respect to a subdivision of the sample domain $\Gamma$ into bins. 
{The piecewise-linear approximation we propose is a finite element function obtained as a linear combination of hat functions. The procedure to determine the coefficients of the linear combination is purely algebraic and can be efficiently carried out.}
The approach is consistent in the sense that the approximate PDF converges to the exact PDF with respect to the $L^2(\Gamma)$ norm as the number of samples tends to infinity and the volume of the bins tend to zero. {Moreover, the only smoothing parameter involved is the bin size that can be related heuristically to the sample size by a simple rule.}

The paper is structured as follows. In Section \ref{proposed_method}, the mathematical foundation of our approach is laid out; there, it is shown analytically that the proposed approximation satisfies several requirements needed for it to be considered as a PDF estimator. A numerical investigation of the accuracy and computational costs incurred by our approach is then provided. First, in Section \ref{kpdf}, our method is tested and validated using sample sets associated with different types of known PDFs so that exact errors can be determined. Then, in Section \ref{updf}, the method is applied to the estimation of {\em unknown} PDFs of outputs of interest associated with the solution of a stochastic partial differential equation. Finally, concluding remarks and future work are discussed in Section \ref{conclusions}.

\section{Piecewise-linear polynomial approximations of PDFs}\label{proposed_method}

Let ${\bm Y}$ denote a multivariate random variable belonging to a closed, bounded parameter domain $\Gamma \subset \mathbb{R}^{N_\Gamma}$ which, for simplicity, we assume is a polytope in $\mathbb{R}^{N_\Gamma}$. The probability density function (PDF) $f({\bm Y})$ corresponding to ${\bm Y}$ is not known. However, we assume that we have in hand a data set of $M$ samples ${\bm Y}_m\in\Gamma$, $m=1,\ldots,M$, of ${\bm Y}\in\Gamma$. The goal is to construct an estimator for $f({\bm Y})$ using the given data set $\{{\bm Y}_m\}_{m=1}^M$ of samples. To this end, we use an approximating space that is popular in the finite element community \cite{brenner2007mathematical,ciarlet}.

Let ${\BINS}^\delta=\{\BIN_\lll\}_{\lll=1}^{N_{bins}}$ denote a covering, non-overlapping subdivision of the sample domain $\Gamma$ into ${N_{bins}}$ bins.\footnote{In the partial differential equation (PDE) setting, what we refer to as bins are often referred to as grid cells or finite elements or finite volumes. We instead refer to the subdomains $\{\BIN_\lll\}_{\lll=1}^{N_{bins}}$ as {\em bins} because that is the notation in common use for histograms which we use to compare to our approach. Furthermore, in Section \ref{updf}, we also use finite element grids for spatial discretization of partial differential equations, so that using the notation ``bins'' for parameter domain subdivisions helps us differentiate between subdivisions of parameter and spatial domains. For the same reason, we use $\delta$ instead of $h$ to parametrize parameter bin sizes because $h$ is in common use to parametrize spatial grid sizes.} Here, $\delta$ parametrizes the subdivision and may be taken as, e.g, the largest diameter of any of the bins $\{\BIN_\ell\}$. The bins are chosen to be hyper-quadrilaterals; for example, if $N_\Gamma=2$, they would be quadrilaterals. It is also assumed that the faces of the bins are either also complete faces of abutting bins or part of the boundary of $\Gamma$. From a practical point of view, our considerations are limited to relatively small ${N_\Gamma}$ because ${N_{bins}}={\mathcal O}(1/\delta^{N_\Gamma})$. Detailed discussion about the subdivisions we use can be found in, e.g., \cite{brenner2007mathematical,ciarlet}.

Let $\{\widehat{{\bm Y}}_j\}_{j=1}^{{N_{nodes}}}$ denote the set of nodes, i.e., vertices, of ${\BINS}^\delta$ with ${N_{nodes}}$ denoting the number of nodes of ${\BINS}^\delta$.  Note that we also have that ${N_{nodes}}={\mathcal O}(1/\delta^{N_\Gamma})$. Based on the subdivision ${\BINS}^\delta$, we define the {\em space of continuous piecewise polynomials}
$$
    {\mathcal V}_\delta =\big\{  v\in C(\Gamma) \,\,:\,\,  v|_{\BIN_\ell}\in {\mathcal P}_1(\BIN_\ell)\quad\mbox{for $\ell=1,\ldots,N_{bins}$} \big\},
$$
where, 
for hyper-quadrilateral elements, ${\mathcal P}_1(\cdot)$ denotes the space of {\em $N_\Gamma$-linear} polynomials, e.g., bilinear and trilinear polynomials in two and three dimensions, respectively. 

A basis $\{\phi_j({\bm Y})\}_{j=1}^{N_{nodes}}$ for ${\mathcal V}_\delta$ is given by, for $j=1,\ldots,N_{nodes}$,   
$$
    \phi_j({\bm Y}) = 
    \big\{ \phi_j({\bm Y})\in {\mathcal V}_\delta\, :\, \phi_j({\widehat{\bm Y}}_{j'}) = \delta_{jj'} \quad\mbox{for $j'=1,\ldots,N_{nodes}$} \big\},
$$
where $\delta_{jj'}$ denotes the Kronecker delta function. In detail, we have that $\{\phi_j({\bm Y})\}_{j=1}^{N_{nodes}}$ denotes the continuous piecewise-linear or piecewise ${N_\Gamma}$-linear Lagrangian FEM basis corresponding to ${\BINS}^\delta$, i.e, we have that, for $j=1,\ldots,{N_{nodes}}$,
\begin{itemize}
\item[--]for hyper-quadrilateral bins, $\phi_j({\bm Y})$ is an ${N_\Gamma}$-linear function on each bin $\BIN_\ell$, $\ell=1,\ldots,{N_{bins}}$,  e.g., for ${N_\Gamma}=\{1,2,3\}$, a linear, bilinear, or trilinear function, respectively;
\item[--]$\phi_j({\bm Y})$ is continuous on $\Gamma$;
\item[--]$\phi_j(\widehat{{\bm Y}}_{j})=1$ at the $j$-th node $\widehat{{\bm Y}}_{j}$ of the subdivision ${\BINS}^\delta$; and
\item[--]if $j'\ne j$, $\phi_j(\widehat{{\bm Y}}_{j'})=0$ at the $j'$-th node $\widehat{{\bm Y}}_{j'}$ of the subdivision ${\BINS}^\delta$.
\end{itemize}
For $j=1,\ldots,{N_{nodes}}$, let $S_j({\bm Y})=\mbox{\em support}\,\{\phi_j({\bm Y})\}\subset\Gamma$ and let $V_j=\mbox{\em volume}\,\{S_j({\bm Y})\}$; note that $S_j({\bm Y})$ consists of the union of the bins $\BIN_\ell\in{\BINS}^\delta$ having the node $\widehat{{\bm Y}}_j$ as one of its vertices. Thus, the basis functions have compact support with respect to $\Gamma$.  An illustration of the basis functions in one dimension is given in Figure \ref{fig:hatfunctions}. We further let $N_{bins,j}$, for $j=1,\ldots,N_{nodes}$, denote the number of bins in $S_j({\bm Y})$, i.e., the number of bins that share the vertex $\widehat{{\bm Y}}_{j}$.

\begin{figure}[h!]
\centerline{ \includegraphics[width=3.0in]{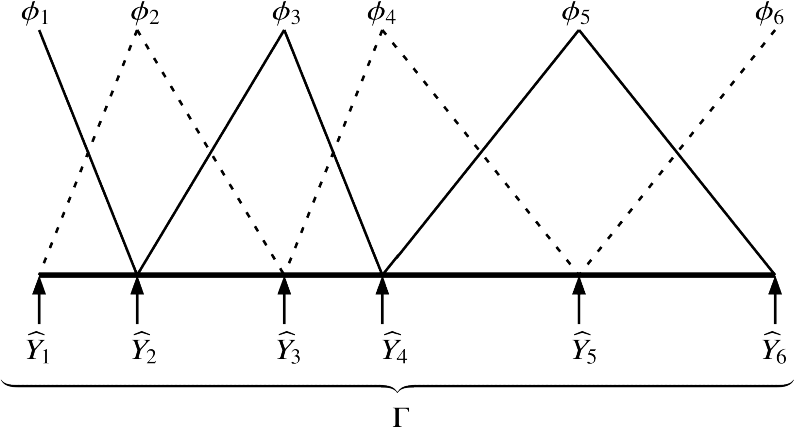}}
\caption{\em The set of basis function $\{\phi_j({\bm Y})\}_{j=1}^{6}$ for the case of $N_{nodes}=6$ in one dimension. Note that the support of the basis functions is limited to the two intervals that contain the corresponding node. Here, the number of bins is $N_{bins}=5$ and the number of bins $N_{bins,j}$ in the support of the basis functions $\phi_j(Y)$ is one for $j=1,6$ and two for $j=2,3,4,5$.}
\label{fig:hatfunctions}
\end{figure}

Note that the approximating space ${\mathcal V}_\delta$ and the basis $\{\phi_j({\bm Y})\}_{j=1}^{N_{nodes}}$ are in common use for the finite element discretization or partial differential equations. Details about the geometric subdivision ${\BINS}^\delta$, the approximation space ${\mathcal V}_\delta$, and the basis functions $\{\phi_j({\bm Y})\}$ and their properties may be found in, e.g., \cite{brenner2007mathematical,ciarlet}.

Below we make use of two properties of the basis $\{\phi_j(\mathbf{Y})\}_{j=1}^{N_{nodes}}$. First, we have the well-known relation
\begin{equation}\label{unity1}
\sum_{j=1}^{N_{nodes}}\,\phi_j(\mathbf{Y})=1  \quad\forall \, \mathbf{Y}\in\Gamma.
\end{equation}
We also have that
\begin{equation}\label{unity2}
C_j = \int_\Gamma \phi_j({\bm Y}) d{\bm Y}=
\int_{S_j({\bm Y})} \phi_j({\bm Y}) d{\bm Y}
=\sum_{{\BIN}_\ell\in S_j({\bm Y})} \int_{{\BIN}_\ell} \phi_j({\bm Y}) d{\bm Y}.
\end{equation}
Note that, in general, $C_j$ is proportional to $V_j$.

\subsection{The piecewise-linear approximation of a PDF}

Given the $M$ samples values $\{{\bm Y}_m\}_{m=1}^M$ in $\Gamma$, we define the approximation $f_{\delta,M}({\bm Y})$ $\in{\mathcal V}_\delta$ of the unknown PDF $f({\bm Y})$ given by   
\begin{equation}\label{proposed_f}
\fbox{$\displaystyle
\begin{aligned}
& 
f({\bm Y}) \approx f_{\delta,M}({\bm Y}) = \sum_{j=1}^{N_{nodes}} F_j \phi_j({\bm Y}) \in{\mathcal V}_\delta,
\\[1ex]&
\qquad\qquad\mbox{where}\quad
 F_j =  \dfrac{1}{MC_j} \,\sum_{{\bm Y}_m\in S_j({\bm Y}) } \phi_j({\bm Y}_m)
\\[1ex]
&
\qquad\qquad\qquad \mbox{with}
 \quad
 C_j = \int_{S_j({\bm Y})} \phi_j({\bm Y}) d{\bm Y},\,\,j=1,\ldots,N_{nodes}.
\end{aligned}
$}
\end{equation}
Note that only the samples ${\bm Y}_m\in S_j({\bm Y})$, i.e., only the samples in the support $S_j({\bm Y})$ of the basis function $\phi_j({\bm Y})$, are used to determine $F_j$. 
{We observe that the proposed estimator can be regarded as a kernel density estimator with linear binning \cite{hall1996accuracy}, where the binning kernel and the kernel associated with a given grid point are equal to the same hat function. With this choice, the smoothing parameter of the kernel becomes the binning parameter $\delta$, so no additional tuning of the bandwidth is necessary.}

Of course, the approximate PDF \eqref{proposed_f} should be a PDF in its own right. That it is indeed a PDF is shown in the following lemma.

\begin{lemma}
$f_{\delta,M}(\mathbf{Y})\ge0$ for all $\mathbf{Y}\in\Gamma$ and $\int_{\Gamma} f_{\delta,M}(\mathbf{Y}) d\mathbf{Y}=1$.

\begin{proof}
Clearly $f_{\delta,M}({\bm Y})$ is non-negative because it is a linear combination of non-negative functions with non-negative coefficients.
\begin{equation}
\begin{aligned}
 \int_{\Gamma} f_{\delta,M}(\mathbf{Y}) d\mathbf{Y} &
 = \sum_{j=1}^{N_{nodes}} F_j  \int_{\Gamma}  \phi_j(\mathbf{Y}) d\mathbf{Y}
\\& = \sum_{j=1}^{N_{nodes}} \dfrac{1}{M \, C_j } \,\sum_{m=1}^M \phi_j(\mathbf{Y}_m) \int_{\Gamma}  \phi_j(\mathbf{Y}) d\mathbf{Y} \\
 &=\dfrac{1}{M} \,\sum_{i=1}^{N_{nodes}}\, \sum_{m=1}^M \phi_j(\mathbf{Y}_m) = \dfrac{1}{M} \, \sum_{m=1}^M 1 = 1.
\end{aligned}
\end{equation}
The third and fourth equalities hold because of \eqref{unity2} and \eqref{unity1}, respectively.
\end{proof}
\end{lemma}

The next lemma is useful to prove the convergence of our approximation.
\begin{lemma}\label{bias}
 Let $f \in C^2(\Gamma)$ with $f|_{\partial \Gamma} = 0$ and let $\mathbb{E}[ F_j]$ denote the expectation of  $F_j$ with respect to $f$. Then
 \begin{align}
   \Big|f(\widehat{{\bm Y}}_{j}) - \mathbb{E}[ F_j] \Big| \leq C \delta^{\alpha},
 \end{align}
 where the constant $C$ does not depend on either $\delta$ or $M$, and $\alpha$ is a positive integer.
 If $[-1,1]^{N_{\Gamma}} \cap \Gamma = [-c_1,c_1]^{{N_{\Gamma}}}$ for some positive constant $c_1$, then $\alpha = 2$. Otherwise $\alpha=1$.
\end{lemma}
\begin{proof}
 Let $\chi_{[-1,1]}$ be the characteristic function of $[-1,1]$ and define $\phi(Y) := (1 - |Y|)\chi_{[-1,1]}$.
 Let $\phi(\mathbf{Y}):= \prod_{n=1}^{N_{\Gamma}} \phi(Y_n) $ be defined in the usual tensor product fashion.
 Assuming $f|_{\partial \Gamma} = 0$, we have $C_j = \delta^{N_{\Gamma}}$ for all $j$.
 Then 
 \begin{align}
  \phi_j(\mathbf{Y}_m) = \phi\Big(\dfrac{\widehat{{\bm Y}}_{j} - \mathbf{Y}_m}{\delta}\Big), \qquad F_j = \dfrac{1}{\delta^{N_{\Gamma}} \, M} \sum_{m=1}^M \phi\Big(\dfrac{\widehat{{\bm Y}}_{j} - \mathbf{Y}_m}{\delta}\Big).
 \end{align}
 $F_j$ is the value of a naive kernel density estimator of $f$ evaluated at $\widehat{{\bm Y}}_{j}$, with the function $\phi$ as a kernel. 
 Using a standard argument for the bias of kernel density estimators we have that
 \begin{equation}
 \begin{aligned}
  \Big|f(\widehat{{\bm Y}}_{j}) - \mathbb{E}[F_j] \Big| \leq &\delta  \Big|\dfrac{\partial f(\widehat{{\bm Y}}_{j})}{\partial \mathbf{Y}}\int_{[-1,1]^{N_{\Gamma}}\cap\Gamma} \phi(\mathbf{Y}^{'}) \mathbf{Y}^{'}d\mathbf{Y}^{'} \Big|\\
  & \mbox{+} \dfrac{\delta^2}{2}  \Big|\int_{[-1,1]^{N_{\Gamma}}\cap\Gamma} {\mathbf{Y}^{'}}^T \dfrac{\partial^2 f(\widehat{{\bm Y}}_{j})}{\partial \mathbf{Y}^2} \mathbf{Y}^{'} d\mathbf{Y}^{'} \Big| + \mathcal{O}(\delta^2).
 \end{aligned}
 \end{equation}
 The above inequality proves the result for $\alpha=1$.
 Thanks to the symmetry of $\phi$, if $[-1,1]^{N_{\Gamma}} \cap \Gamma = [-c_1,c_1]^{N_{\Gamma}}$ for some positive constant $c_1$, then $\int_{[-1,1]^{N_{\Gamma}}\cap\Gamma}  \phi(\mathbf{Y}^{'})\mathbf{Y}^{'}d\mathbf{Y}^{'} = 0$, hence the result follows with $\alpha=2$.
\end{proof}
The next theorem shows that the approximate PDF obtained with our method converges to the exact PDF with respect to the $L^2(\Gamma)$ norm.
\begin{theorem}\label{consistency}
Let $\Gamma$ be a polytope in $\mathbb{R}^{N_\Gamma}$ and $f \in C^2(\Gamma)$ with $f|_{\partial \Gamma} = 0.$ 
If $f_{\delta,M}$ is the approximation of $f$ given in  \eqref{proposed_f}, then:
 $$ \lim_{\delta \rightarrow 0} \lim_{M \rightarrow \infty}  \|f - f_{\delta,M} \|_{L^2(\Gamma)} = 0.$$
\end{theorem}
Moreover, if $[-1,1]^{N_{\Gamma}} \cap \Gamma = [-c_1,c_1]^{N_{\Gamma}}$ for some positive constant $c_1$, then 
 $$ \lim_{M \rightarrow \infty}  \|f - f_{\delta,M} \|_{L^2(\Gamma)} \leq C \delta^2,$$
 where $C$ is a constant that does not depend on $\delta$ or $M$.
\begin{proof}
Let $\mathcal{I}_{\delta} f = \sum\limits_{j=1}^{N_{nodes}} f(\widehat{{\bm Y}}_{j}) \phi_j$ be the finite element nodal interpolant of $f$, then
\begin{equation}
\begin{aligned}
    \|f - f_{\delta,M} \|_{L^2(\Gamma)} &\leq \|f - \mathcal{I}_{\delta} f \|_{L^2(\Gamma)} \mbox{+}  \|\mathcal{I}_{\delta} f -  f_{\delta,M} \|_{L^2(\Gamma)}\\
    &\leq C_1\delta^2  \mbox{+}  \|\mathcal{I}_{\delta} f -  f_{\delta,M} \|_{L^2(\Gamma)},
\end{aligned}
\end{equation}
where $C_1$ is a constant that does not depend on $\delta$ \cite{brenner2007mathematical}.
Considering the second term in the above inequality, we have
\begin{equation}
\begin{aligned}
 \|\mathcal{I}_{\delta} f - f_{\delta,M} \|_{L^2(\Gamma)} & \leq \sqrt{\int_{\Gamma} \Big(\sum\limits_{j=1}^{N_{nodes}} | f(\widehat{{\bm Y}}_{j}) - F_j|\phi_j\Big)^2} \\
 & \leq \sqrt{\int_{\Gamma} \Big[\Big(\sum\limits_{j=1}^{N_{nodes}} | f(\widehat{{\bm Y}}_{j}) - \mathbb{E}[F_j]|\phi_j\Big) \mbox{+}  \Big(\sum\limits_{j=1}^{N_{nodes}} | \mathbb{E}[ F_j] - F_j|\phi_j\Big)\Big]^2} \\
 & \leq \sqrt{\int_{\Gamma} \Big[ C_2\delta^{\alpha} \mbox{+}  \Big(\sum\limits_{j=1}^{N_{nodes}} | \mathbb{E}[ F_j] - F_j|\phi_j\Big)\Big]^2}.
\end{aligned}
\end{equation}
The last inequality is obtained using Lemma \ref{bias} and  \eqref{unity1}. 
Considering that $ \mathbb{E}[ F_j] = \mathbb{E}\Big[\dfrac{\phi_j}{\delta^{N_{\Gamma}}}\Big] = \dfrac{1}{\delta^{N_{\Gamma}}} \mathbb{E}\Big[\phi_j\Big]$, we have 
\begin{equation}
\begin{aligned}
  | \mathbb{E}[ F_j] - F_j| &=   \Big| \mathbb{E}\Big[\dfrac{\phi_j}{\delta^{N_{\Gamma}}}\Big] - \dfrac{1}{M}\sum\limits_{m=1}^M \dfrac{\phi_j(\mathbf{Y}_m)}{\delta^{N_{\Gamma}}} \Big| 
\\&  =  \dfrac{1}{\delta^{N_{\Gamma}}}\Big| \mathbb{E}\Big[\phi_j\Big] - \dfrac{1}{M}\sum\limits_{m=1}^M \phi_j(\mathbf{Y}_m) \Big| \leq \dfrac{\sigma(\phi_j)}{\delta^{N_{\Gamma}} \, \sqrt{M}},
\end{aligned}
\end{equation}
where {$\sigma(\phi_j) = \sqrt{\mathbb{E}[\phi_j^2] - \mathbb{E}[\phi_j]^2} \leq C_3\sqrt{\delta^{N_{\Gamma}}}$ for all $j$, with $C_3$ independent of both $\delta$ and $M$.}
Hence
\begin{align}\label{someEq}
   \|f - f_{\delta,M} \|_{L^2(\Gamma)} \leq C_1 \delta^2 \mbox{+}  C_2 \delta^{\alpha} \mbox{+}  \dfrac{{C_3}}{{\sqrt{\delta^{N_{\Gamma}}}} \, \sqrt{M}},
\end{align}
so the first result is obtained.
If $[-1,1]^{N_{\Gamma}} \cap \Gamma = [-c_1,c_1]^{N_{\Gamma}}$ for some positive constant $c_1$, then $\alpha=2$ in  \eqref{someEq}, so the second result also follows taking the limit as $M\rightarrow \infty$. 
 \end{proof}

We note that the numerical examples considered below show that convergence can be obtained even for cases where the PDF is not in $C^2(\Gamma)$, even when the PDF is not differentiable or even continuous.

\subsection{Numerical illustrations}\label{kkpdf}

In Section \ref{kpdf}, we validate our approach by approximating known joint PDFs $f({\bm Y})$. Of course, in comparing approximations to an exact known PDF, we pretend that we have no or very little knowledge about the latter except that we have available $M$ samples of the PDF $\{f({\bm Y}_m)\}_{m=1}^M$ at $M$ points ${\bm Y}_m\in\Gamma$, $m=1,\ldots,M$. 
{For the rest of the paper, whenever $M_1<M_2$, then $M_1 \subset M_2$, meaning that smaller sample sets are obtained as subsets of a larger sample set.}
Comparing with known PDFs allows us to precisely determine errors in the approximation of the PDF determined using our method. Then, in Section \ref{updf}, we use our method to approximate the PDFs of outputs of interest associated with the solution of a stochastic partial differential equation; in that case, the PDF is not known. All computations were performed on a Dell Inspiron 15, 5000 series laptop with the CPU \{Intel(R) Core(TM) i3-4030U CPU\@1.90GHz, 1895 MHz\} and 8 GB of RAM.

Note that in all the numerical examples, $\Gamma$ denotes a sampling domain, i.e., all samples $\{\bm Y_m\}_{m=1}^M$ lie within $\Gamma$. For most cases, $\Gamma$ is also the support domain for the PDF. However, we also consider the case in which the support of the PDF is not known beforehand so that the sampling domain $\Gamma$ is merely assumed to contain, but not be the same as, the support domain. 

For simplicity, the sample space is assumed to be a bounded ${N_\Gamma}$-dimensional box $\Gamma = [a,b]^{N_\Gamma}$ with $a<b$. We subdivide the parameter domain $\Gamma$ into a congruent set of bins ${\BINS}^\delta=\{\BIN_\ell\}_{k=1}^{N_{bins}}$ consisting of ${N_\Gamma}$-dimensional hypercubes of side $\delta=(b-a)/N_\delta$, where $N_\delta$ denotes the number of intervals in the subdivision ${\BINS}^\delta$ of $\Gamma$ along each of the ${N_\Gamma}$ coordinate directions. We then have that the number of bins is given by $N_{bins}=N_\delta^{N_\Gamma}$ and the number of nodes is given by ${N_{nodes}}=(N_\delta$+$1)^{N_\Gamma}$. 
For simplicity, we assume throughout that the components of the random variable ${\bm Y}$ are independently distributed so that the joint PDFs are given as the product of univariate PDFs; our method can also be {applied in a straightforward way} to cases in which the components of ${\bm Y}$ are correlated.

\section{Validation through comparisons with known PDFs}\label{kpdf}
 
In this section, we assume that we have available $M$ samples $\{{\bm Y}_m\}_{m=1}^M$ drawn from a {\em known} PDF $f({\bm Y})$. The error incurred by any approximation ${f}_{approx}({\bm Y})$ of the exact PDF $f({\bm Y})$ is measured by 
\begin{equation}\label{al2e}
 {\mathcal E}_{f^{approx}} = \Big(\dfrac{1}{M}\sum_{m=1}^M \big(f({\bm Y}_m) - {f}^{approx}({\bm Y}_m)\big)^2 \Big)^{1/2}.
\end{equation}
In particular, we use this error measure for our approximation ${f}_{\delta,M}({\bm Y})$ defined in \eqref{proposed_f}. 

The accuracy of approximations of a PDF, be they by histograms or by our method, depends on both $M$ (the number of samples available) and $\delta$ (the length of the bin edges). Thus, $M$ and $\delta$ should be related to each other in such a way that errors {in \eqref{al2e}} due to sampling and bin size are commensurate with each other. Thus, if the bin size errors  {in \eqref{al2e}} are of ${\mathcal O}(\delta^r)$ and the sampling error is of ${\mathcal O}(M^{-1/2})$, we set
\begin{equation}\label{ns}
M = {(b-a)^{2r}}\delta^{-2r} = {{N_\delta}^{2r}}.
\end{equation}
Thus, once $a$ and $b$ are specified, one can choose $N_\delta$ (or equivalently $\delta$) and the value of $M$ is determined by \eqref{ns} or vice versa. Clearly, $M$ increases as $\delta$ decreases. 

For most of the convergence rate illustrations given below, we 
\begin{equation}\label{ns1}
\mbox{choose \quad{$N_\delta=2^{(3-r)k}$},\, $k=1,2,\ldots$,\quad so that\quad ${\delta = \dfrac{(b-a)}{2^{(3-r)k}}}$ \quad and\quad
${M = 2^{2r(3-r)k}}$.}
\end{equation}
Note that neither \eqref{ns} or \eqref{ns1} depend on the dimension $N_\Gamma$ of the parameter domain but, of course, $N_{bins}=N_\delta^{N_\Gamma}$ and $N_{nodes}=(N_{\delta}$ + $1)^{N_\Gamma}$ do. If the variance of the PDF is large, one may want to increase the size of $M$ by multiplying the term $\delta^{-2r}$ in \eqref{ns} by the variance.

The computation of the coefficients $F_j$ defined in \eqref{proposed_f} may be costly if $M$ is large and consequently $\delta$ is small. To improve the computational efficiency, we evaluate a basis function $\phi_j({\bm Y})$ at a sample point ${\bm Y}_m$ only if the point is within the support of $\phi_j({\bm Y})$. However, the determination of the bin $\BIN_\ell\in{\BINS}^\delta$ such that ${\bm Y}_m\in\BIN_\ell$ may be expensive in case of large $M$ and $N_\Gamma>2$. For this task, we employ the efficient point locating algorithm described in \cite{capodaglio2017particle}.


\subsection{A smooth PDF with known support}\label{tgpdf}

For the first example we consider, we ignore the fact that we know the exact PDF we are trying to approximate, but assume we know the support of the PDF so the sampling domain $\Gamma$ is also the support domain. We also use this example to illustrate that the number of Monte Carlo samples needed is independent of the dimension $N_\Gamma$ of the parameter space.

We set $\Gamma=[-5.5,5.5]^{N_\Gamma}$ so that $(b-a)=11$ and assume that the components of the random vector ${\bm Y} = (Y_1, \ldots, Y_{N_\Gamma})$ are independently and identically distributed according to a truncated standard Gaussian PDF so that the joint PDF is given by 
\begin{equation}\label{stdGaussiantrunc}
\begin{aligned}
 &f({\bm Y}) = \prod\limits_{n=1}^{N_\Gamma} \dfrac{1}{\sqrt{2 \pi} C_G}\exp{\Big( -\dfrac{Y_n^2}{2} \Big)}\qquad\mbox{for ${\bm Y}\in\Gamma=[-5.5,5.5]^{N_\Gamma}$}
 \\[-.5ex]
 &\qquad\qquad\mbox{with}\quad C_G = \frac12\big(\mbox{erf}(5.5/\sqrt{2}) - \mbox{erf}(-5.5/\sqrt{2})\big).
\end{aligned}
 \end{equation}
The scaling factor $C_G$ is introduced to insure that we indeed have a PDF, i.e., that the  integral of $f({\bm Y})$ over $\Gamma$ is unity. Note that because the standard deviation of the underlying standard Gaussian PDF is unity, the values of the truncated Gaussian distribution \eqref {stdGaussiantrunc} near the faces of the box $\Gamma=[-5.5,5.5]^{N_\Gamma}$ are very small so that the results of this example are given to a precision such that they would not change if one considers instead the (non-truncated) standard Gaussian distribution. Also note that because the second moment of the standard Gaussian distribution is unity, the absolute error \eqref{al2e} is also very close to the error relative to the given PDF.

Before we use the formula \eqref{ns1} to relate $M$ and $\delta$, we first separately examine the convergence rates with respect to $\delta$ and $M$. To this end, to illustrate the convergence with respect to $\delta$, we set
\begin{align}\label{h_order}
M = 10^7 \quad \mbox{and} \quad \delta=11/2^k \quad \mbox{for} \qquad  k=3,4,5,6, 
\end{align}
so that the error due to sampling is relatively negligible. For illustrating the convergence with respect to $M$, we set
\begin{align}\label{M_order}
\delta = 11/2^8 \quad \mbox{and} \quad M=10^k \qquad \mbox{for} \qquad  k=3,4,5,6, 
\end{align}
so that the error due to the bin size is relatively negligible. The plots for $N_\Gamma=1$ in Figure \ref{order_analytic} illustrate a second-order convergence rate with respect to $\delta$ and a half-order convergence rate with respect to $1/M$.

We now turn to relating $M$ and $\delta$ using the formula \eqref{ns1}. We consider the multivariate truncated standard Gaussian PDF \eqref{stdGaussiantrunc} for $N_\Gamma=1,2,3$. Plots of the error vs. both $\delta$ and $M$ are given in Figure \ref{order_analytic_2} from which we observe, in all cases, the second-order convergence rate with respect $\delta$ and the half-order convergence rate with respect to $1/M$. We also observe that the errors and the number of samples used are largely independent of the value of $N_\Gamma$. A visual comparison of 
the exact truncated standard Gaussian distribution \eqref{stdGaussiantrunc} and its approximation \eqref{proposed_f} for the bivariate case is given in Figure \ref{2D_analytic}.

\begin{figure}[h!]
\centerline{
\includegraphics[scale=0.37]{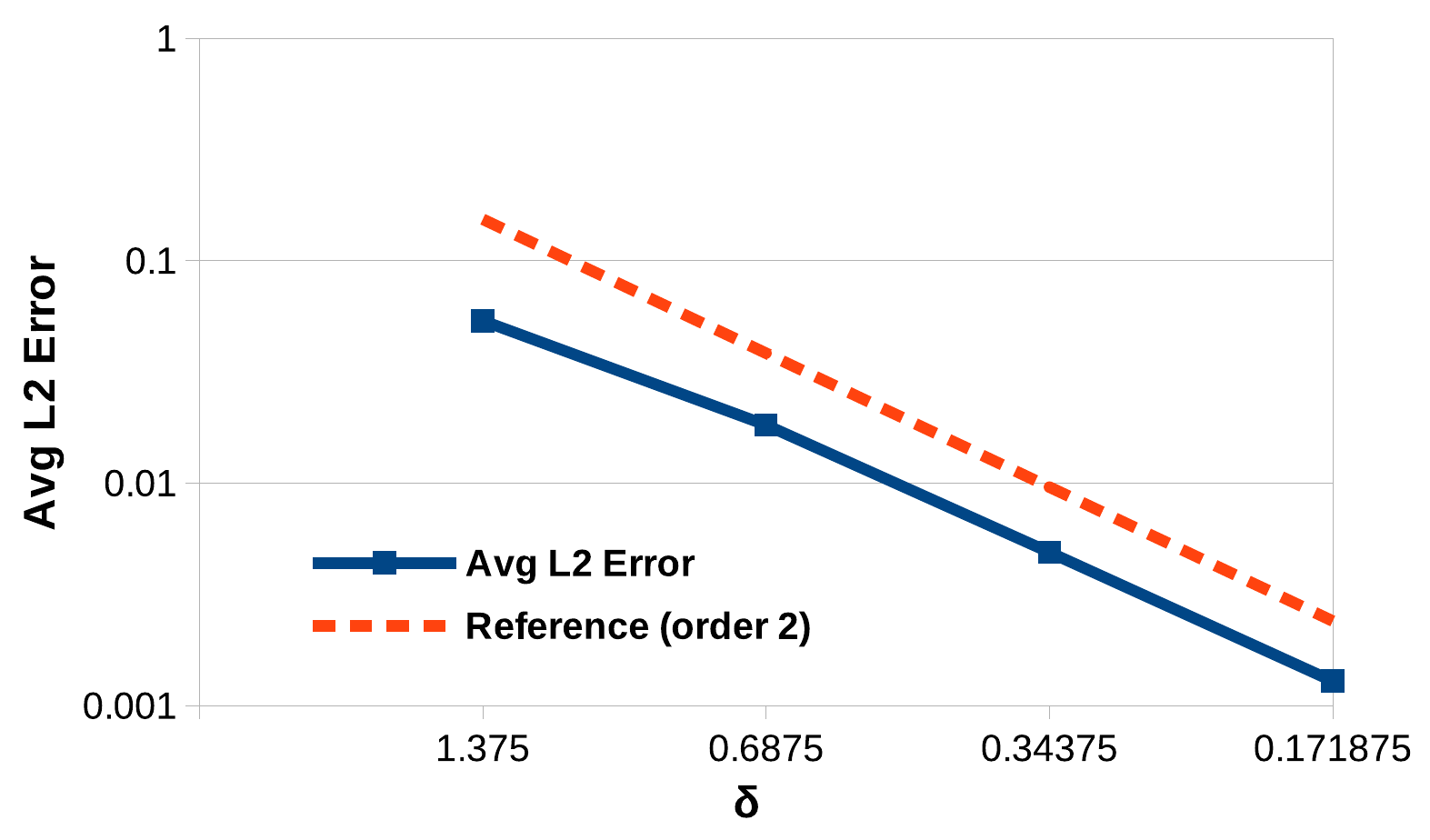}
\quad
\includegraphics[scale=0.37]{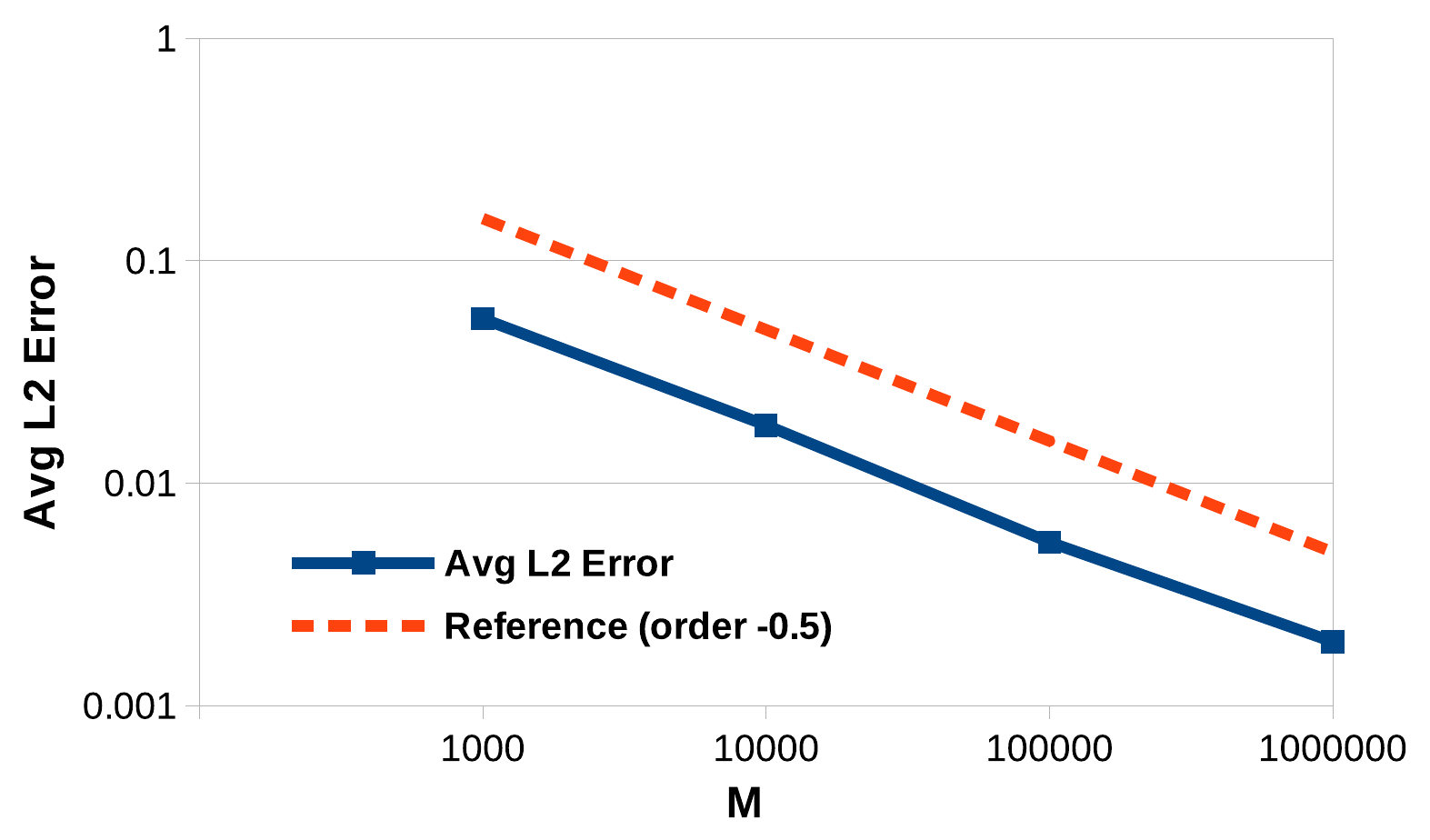}
}
\caption{\em Errors and convergence rates for the approximation \eqref{proposed_f} for the one-dimensional truncated standard Gaussian PDF \eqref{stdGaussiantrunc}. Left: second-order convergence rate with respect to $\delta$ with $M=10^7$. Right: half-order convergence rate with respect to $M$ with $\delta=11/2^8$.}
\label{order_analytic}
\end{figure}

\begin{figure}[h!]
\centerline{
\includegraphics[scale=0.35]{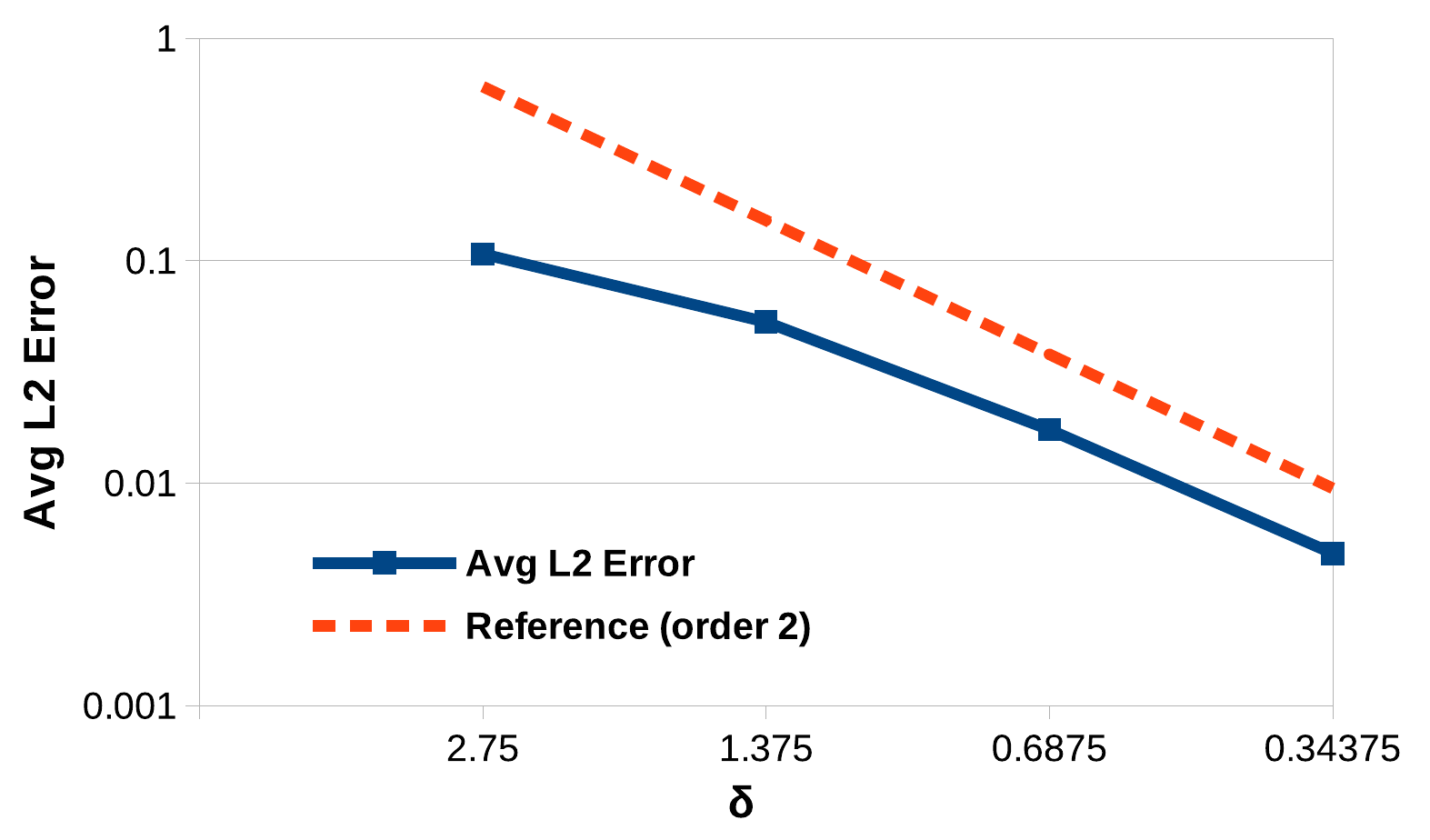} 
\quad
\includegraphics[scale=0.35]{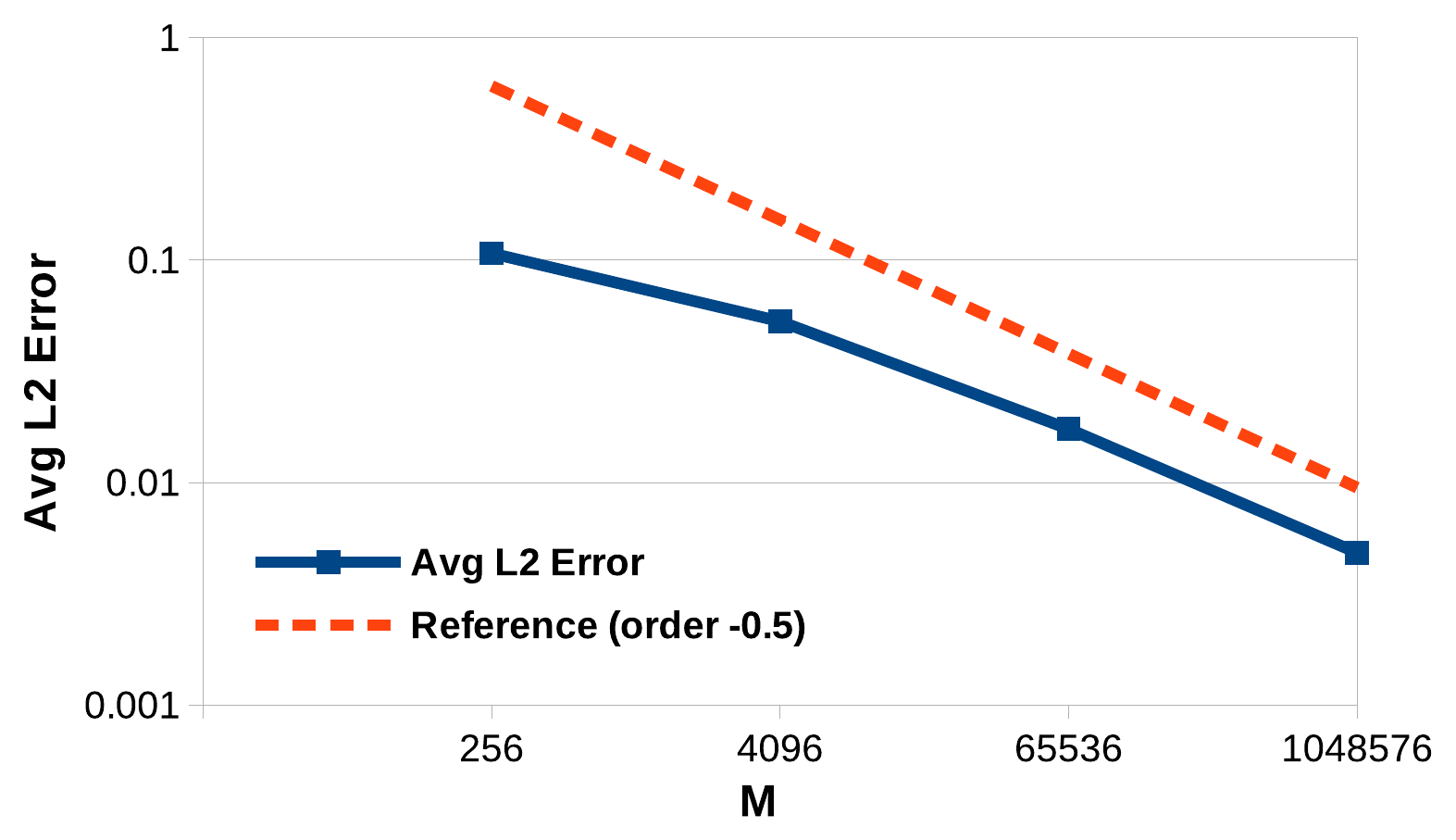}
}
\vskip5pt
\centerline{
\includegraphics[scale=0.35]{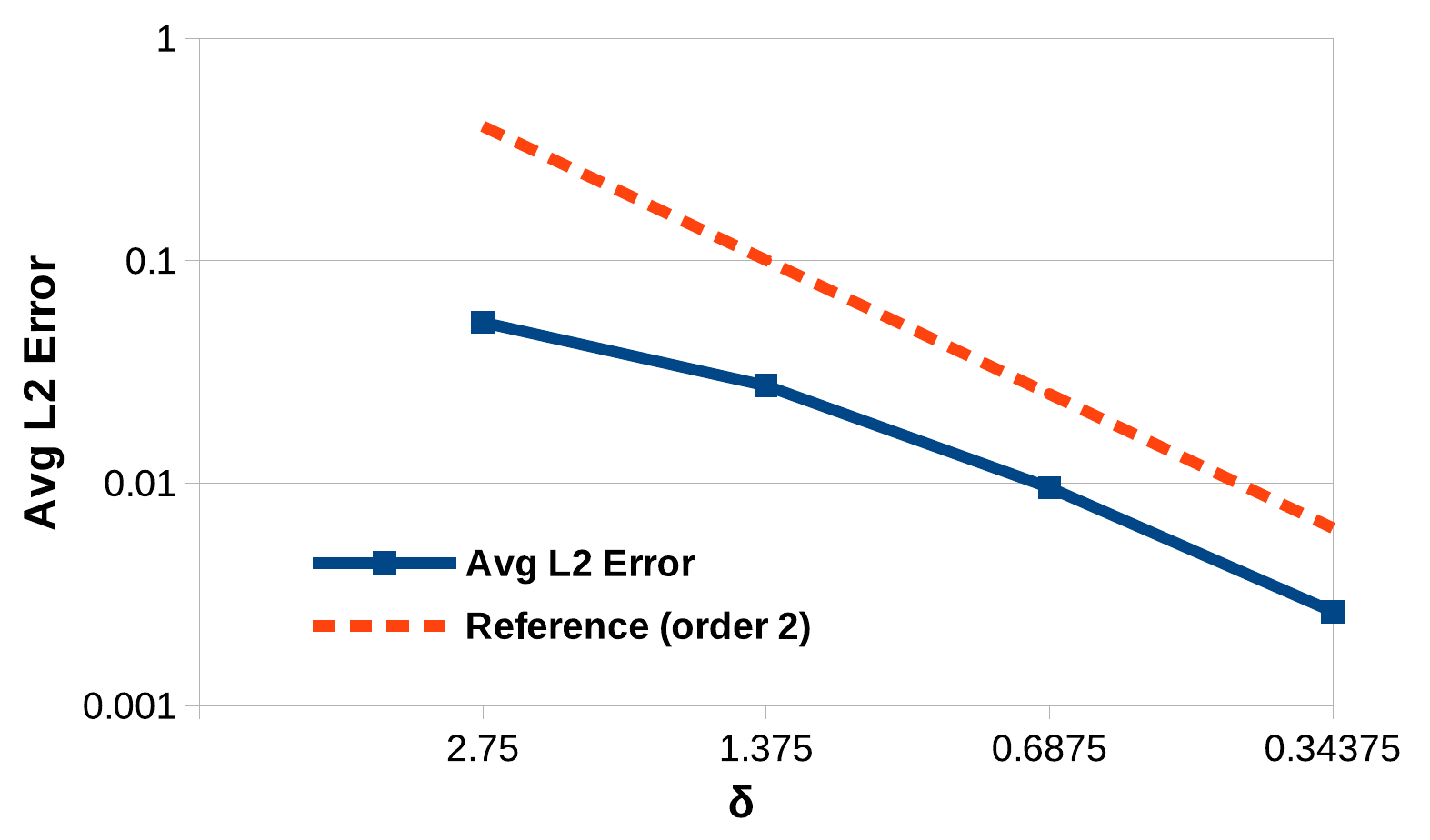} 
\quad
\includegraphics[scale=0.35]{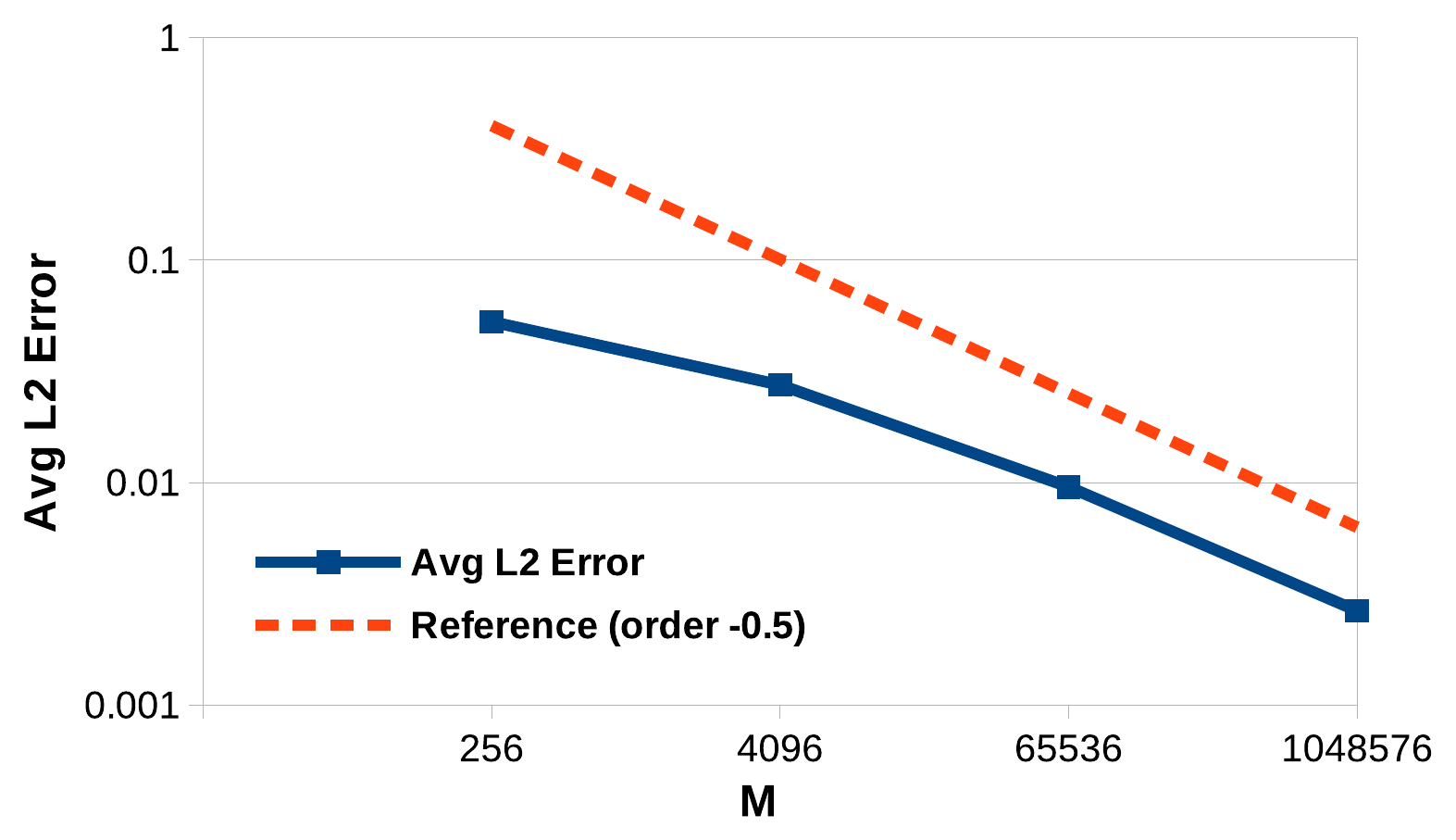}
}
\vskip5pt
\centerline{
\includegraphics[scale=0.35]{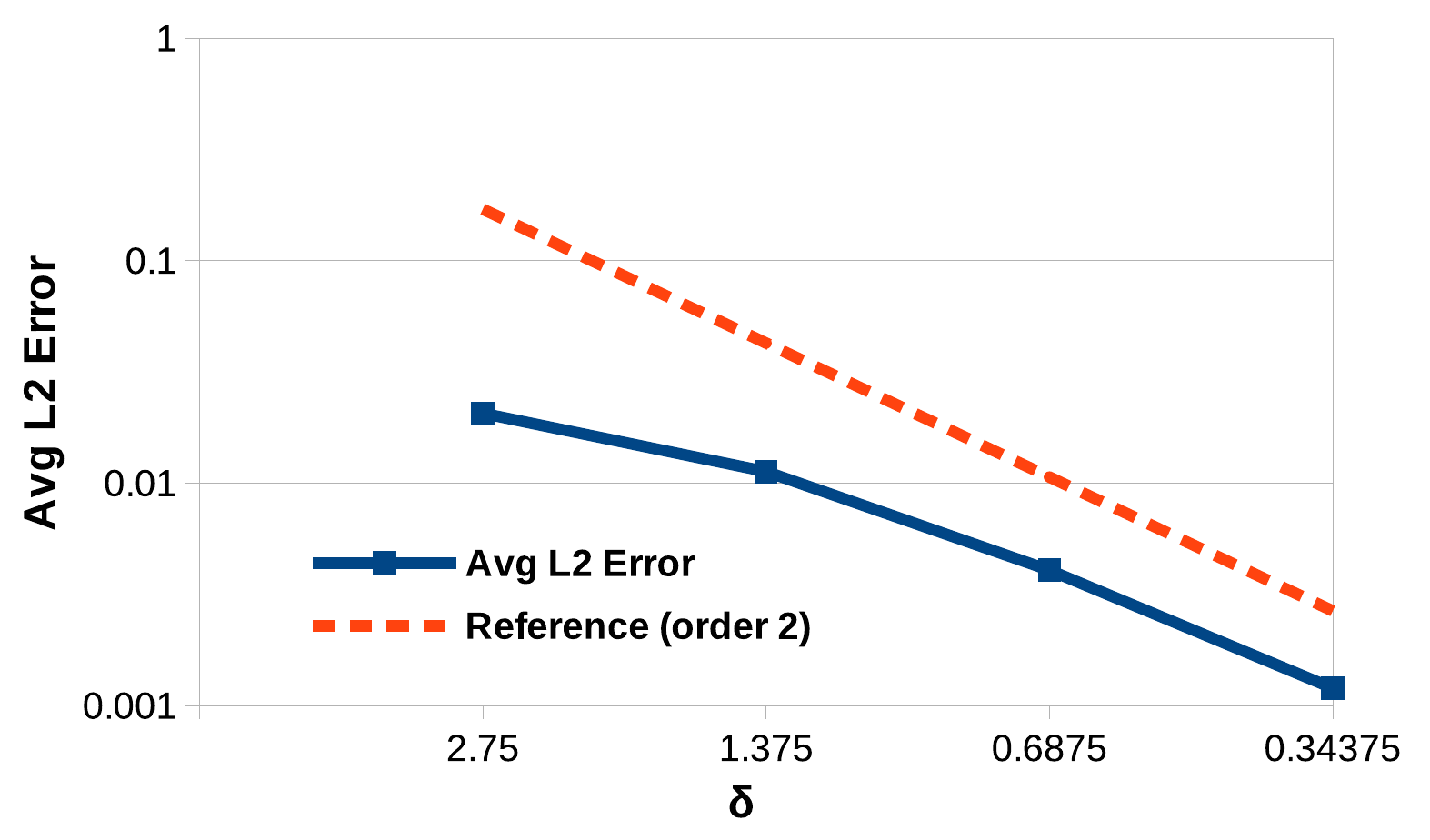} 
\quad
\includegraphics[scale=0.35]{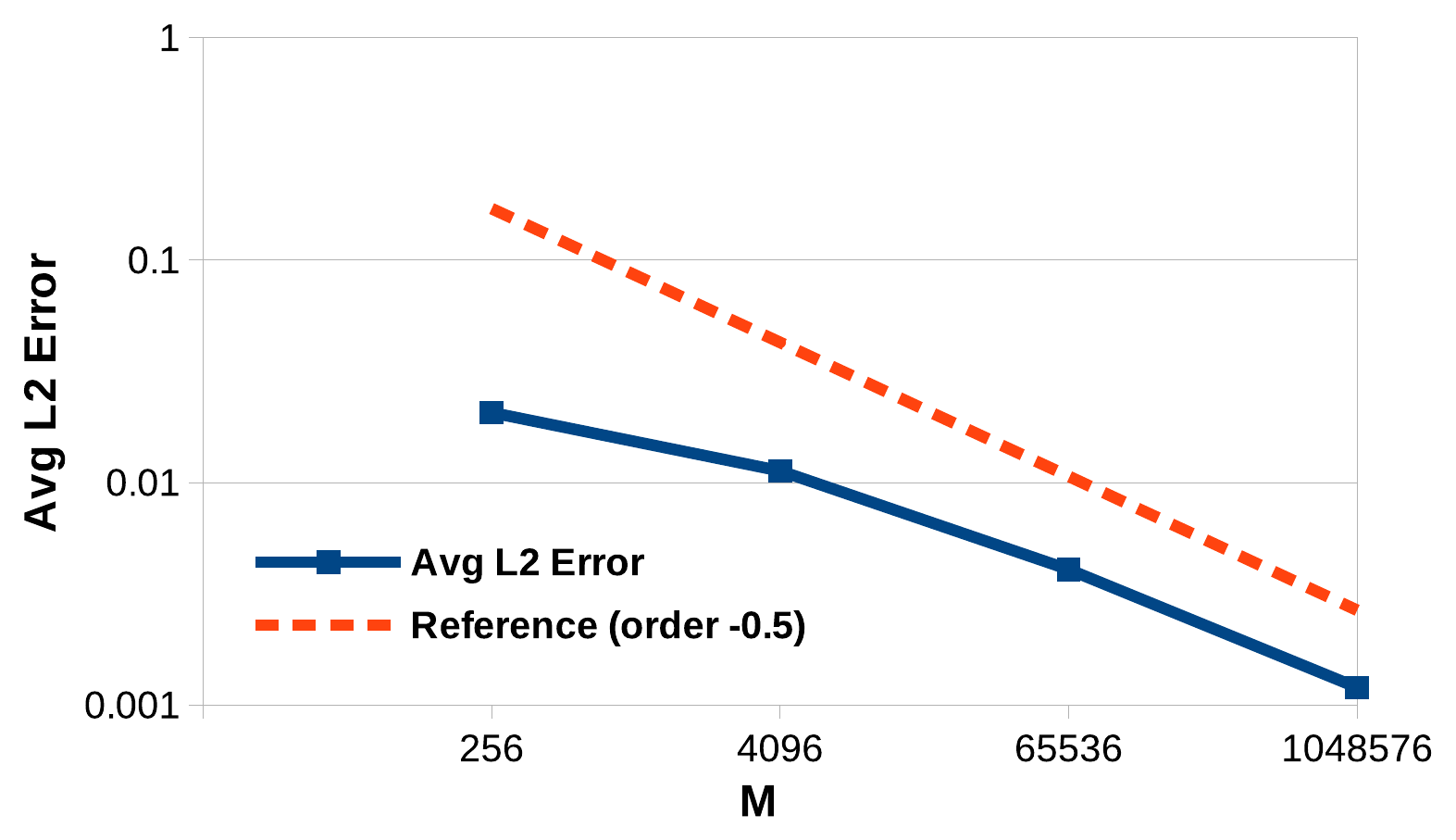}
}
\caption{\em Errors and convergence rates for the approximation \eqref{proposed_f} for the truncated standard Gaussian PDF \eqref{stdGaussiantrunc} with $M$ and $\delta$ related through \eqref{ns1}. Left: convergence rates with respect to $\delta$. Right: convergence rates with respect to $M$. Top to bottom: $N_\Gamma=1,2,3$.
}
\label{order_analytic_2}
\end{figure}

\begin{figure}[h!]
\centerline{
\includegraphics[height=1.5in,width=1.7in]{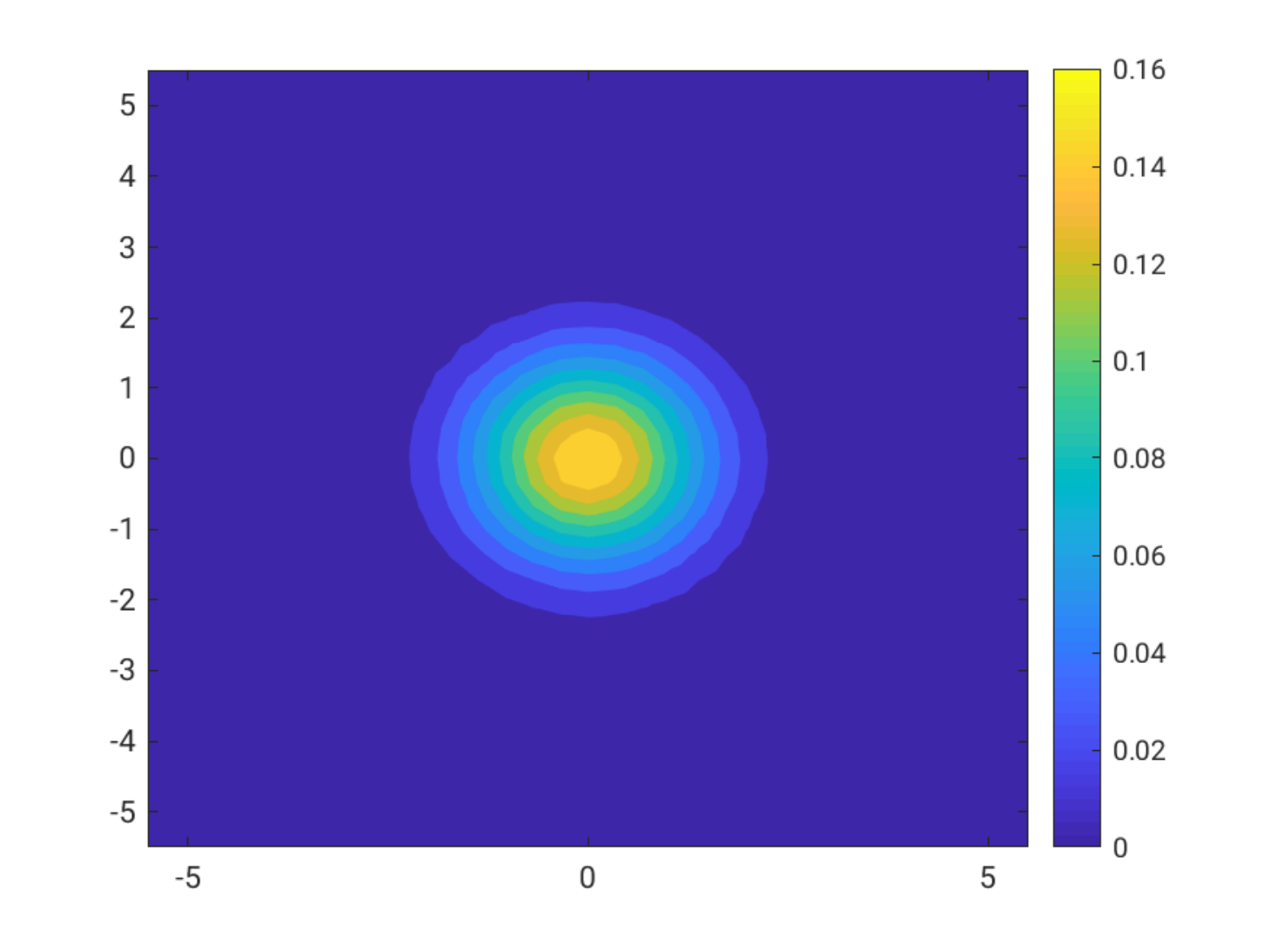}\qquad
\includegraphics[height=1.5in,width=1.33in]{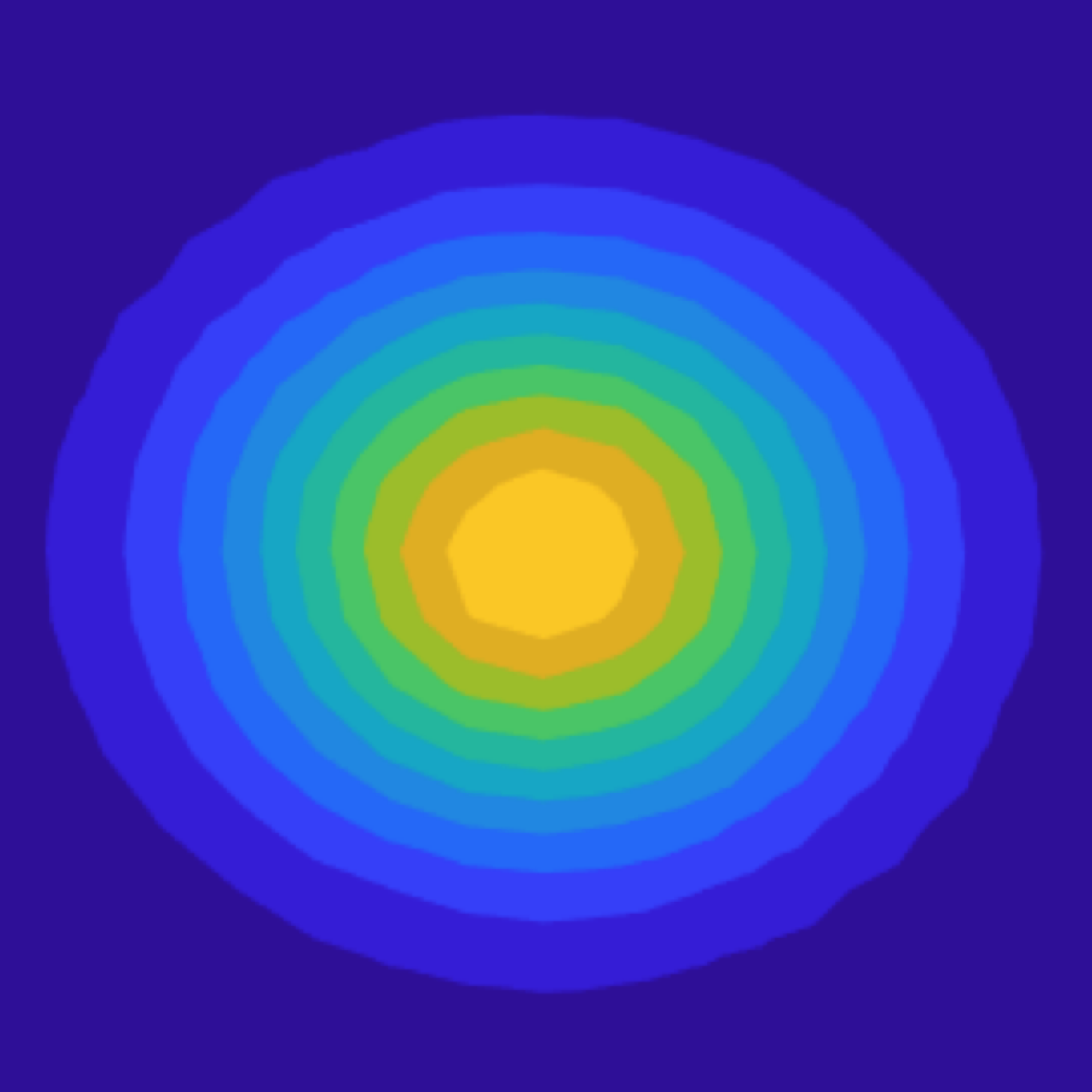}
\includegraphics[height=1.5in,width=1.33in]{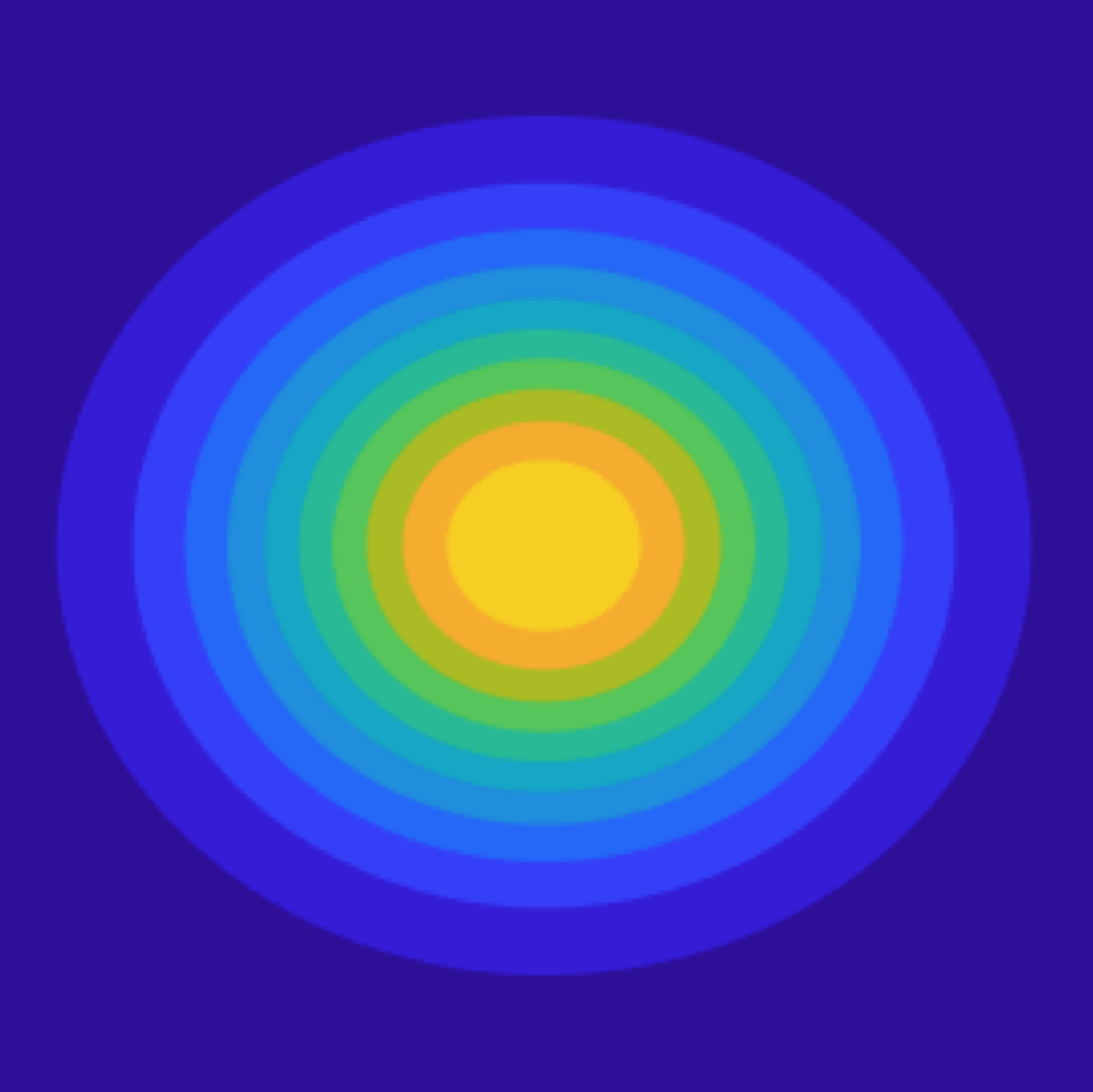}
}
\caption{\em Left: the approximation \eqref{proposed_f} of the bivariate truncated standard Gaussian PDF \eqref{stdGaussiantrunc}. Center: a zoom in of the approximate PDF. Right: a zoom in of the exact PDF. For these plots, $\delta=0.34375=11/2^5$ and $M=1048576=16^5$.}
\label{2D_analytic}
\end{figure}

Computational costs are reported in Table \ref{timeGauss} in which, for each $N_\Gamma=1,2,3$, we choose $k=2,3,4,5$ in \eqref{ns1} to determine $\delta$ and $M$. Reading vertically for each $N_\Gamma$, we see the increase in computational costs due to the decrease in $\delta$ and the related increase in $M$, {although the method scales linearly with respect to the sample size $M$}. Reading horizontally so that $\delta$ and $M$ are fixed, the increase in costs is due to the increasing number of bins and nodes as $N_\Gamma$ increases. We note that our method is amenable to highly scalable parallelization not only as $\delta$ decreases and $M$ increases, but also as $N_\Gamma$ increases so that, through parallelization, our method may prove to be useful in dimensions higher than those considered here. When developing a parallel implementation of our method, using a point locating algorithm such as that of \cite{capodaglio2017particle} to locate a sample on a finite element grid shared by several processors would be crucial to realizing the gains in efficiency due to parallelization.

\begin {table}[!h!]
\caption{\em Computational time (in seconds) for determining the approximation \eqref{proposed_f} of the truncated standard Gaussian PDF \eqref{stdGaussiantrunc}.}
\setlength\tabcolsep{3.25pt} 
\begin{tabular}{|c|c|c|c||c|r||c|r||c|r|} 
\cline{5-10}   
\multicolumn{4}{c|}{} & \multicolumn{6}{c|}{Computation time in seconds}
\\ \hline  
$k$ & $N_\delta$ & $\delta=\frac{11}{N_\delta}$ & $M=2^{4k}$  
    & $N_{bins}$ & ${N_\Gamma}=1$ 
    & $N_{bins}$ & ${N_\Gamma}=2$
    & $N_{bins}$ & ${N_\Gamma}=3$  
\\ 
\hline 
2 & 4  & 2.75    & 256     & 4 & $7.710\mbox{e}{-04}$ & 16 & $3.777\mbox{e}{-02}$ & 64 & $1.190\mbox{e}{-01}$ \\ 
\hline   
3 & 8  & 1.375   & 4096    & 8 & $1.356\mbox{e}{-02}$ & 64 &$5.114\mbox{e}{-01}$ & 512 & $1.955\mbox{e}{-00}$  \\ 
\hline
4 & 16 & 0.6875  & 65536   & 16 &$2.796\mbox{e}{-01}$ & 256 &$8.106\mbox{e}{-00}$ & 4096 & $3.545\mbox{e+01}$  \\ 
\hline
5 & 32 & 0.34375 & 1048576 & 32 &$5.870\mbox{e}{-00}$ & 1024 &$1.281\mbox{e+02}$   & 32768 & $1.013\mbox{e+03}$  \\ 
\hline
\end{tabular}
\label{timeGauss}
\end{table}

\subsection{A smooth PDF with unknown support}
\label{unpdf}
Still considering a known PDF, we now consider a case for which we not only pretend we do not know the PDF, but also we do not know its support. Specifically, we consider the uniform distribution $f(Y)=0.5$ on $[-1,1]$. A univariate distribution suffices for the discussions of this case; multivariate distributions can be handled by the obvious extensions of what is said here about the univariate case. We assume that we know that the support of the known PDF lies within a larger interval $\Gamma$. Of course, we may be mistaken about this so that once we examine the sample set $\{{\bm Y}_m\}_{m=1}^M$, we may observe that some of the samples fall outside of $\Gamma$. In this case we can enlarge the interval $\Gamma$ until we observe that the interval spanned by smallest to largest sample values is contained within the new $\Gamma$. 

We first simply assume that we have determined, either through external knowledge or by the process just described, that the support of the PDF we are considering lies somewhere within the interval $\Gamma=[-1.5,1.5]$. Not knowing the true support, we not only sample in the larger interval $\Gamma$ (so that here we have $(b-a)=3$ and $\delta=3/N_\delta$), but we also build the approximate PDF with respect to $\Gamma$.  We remark that a uniform distribution provides a stern test when the support of the distribution is not known because that distribution is as large at the boundaries of its support as it is in the interior. Distributions that are small near the boundaries of their support, e.g., the truncated Gaussian distribution of Section \ref{tgpdf}, would yield considerably smaller errors and better convergence rates compared to what are obtained for the uniform distribution. Choosing $k=2,3,4,5$ and $r=2$ in \eqref{ns1}, we obtain the errors plotted in Figure \ref{unif2_nogood}. Clearly, the convergence rates are nowhere near optimal. Of course, the reason for this is that by building the approximation with respect to $\Gamma$, we are not approximating the uniform distribution on $[-1,1]$, but instead we are approximating the {\em discontinuous} distribution
$$
      f_{[-1.5,1.5]}(Y) = 
      \left\{\begin{aligned}
      1 & \quad \mbox{for $Y\in[-1,1]$}
      \\
      0 & \quad \mbox{for $Y\in[-1.5,-1]$ and $Y\in[1,1.5]$.}
      \end{aligned}\right.
$$
For comparison purposes we provide, in Figure \ref{unif2}, results for the case where we use the support interval $[-1,1]$ for both sampling and for approximation construction. Because now the PDF is smooth, in fact constant, throughout the interval in which the approximation is constructed, we obtain optimal convergence rates.

\begin{figure}[h!]
\centerline{
\includegraphics[scale=0.37]{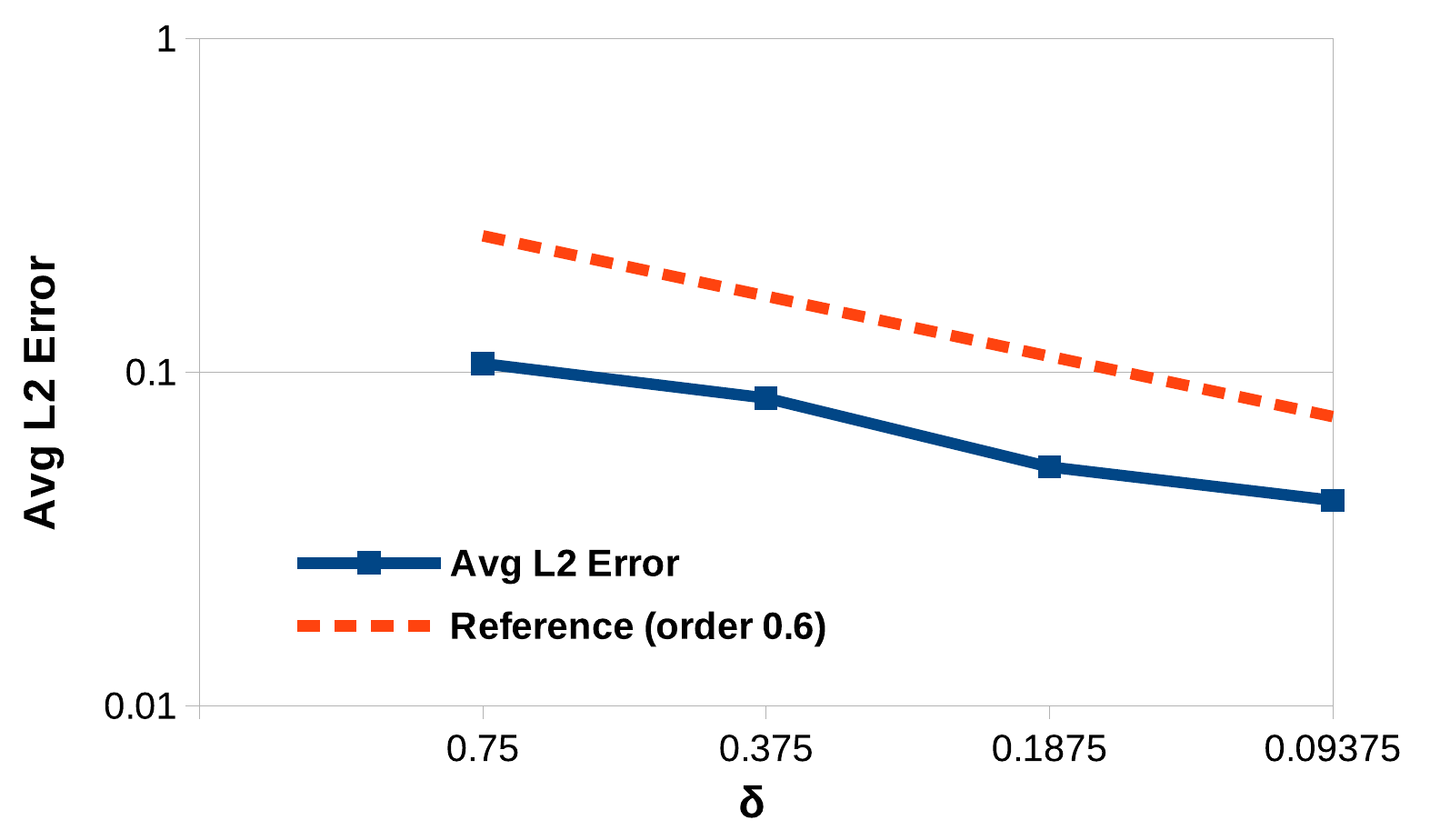}\quad 
\includegraphics[scale=0.37]{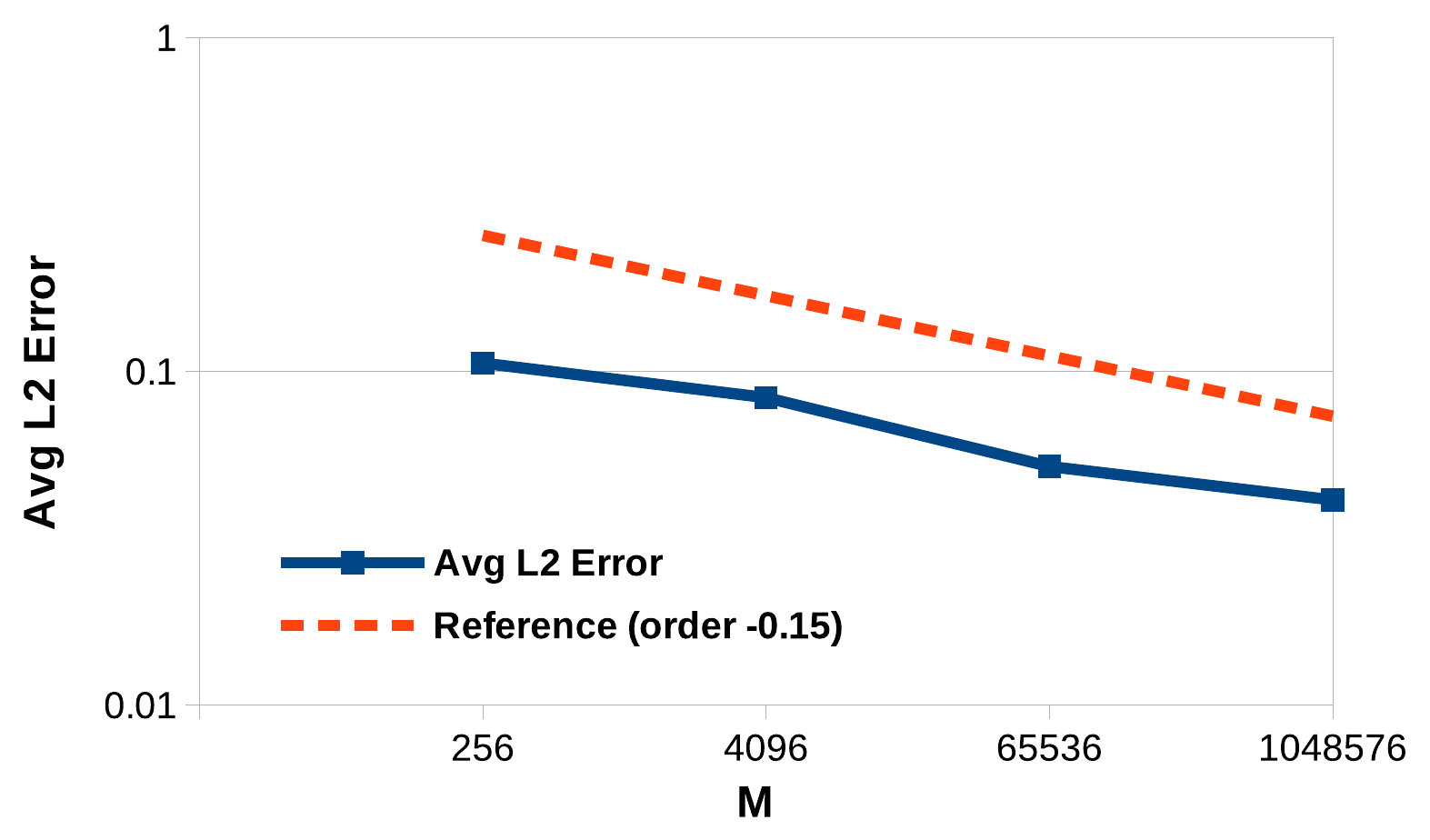}
}
\caption{\em For $M$ and $\delta$ related through \eqref{ns1} with $r=2$, convergence rates with respect to $\delta$ (left) and $M$ (right) for the uniform distribution on $[-1,1]$ but for an approximation built with respect to the larger interval $[-1.5,1.5]$.}
\label{unif2_nogood}
\end{figure}

\begin{figure}[h!]
\centerline{
\includegraphics[scale=0.37]{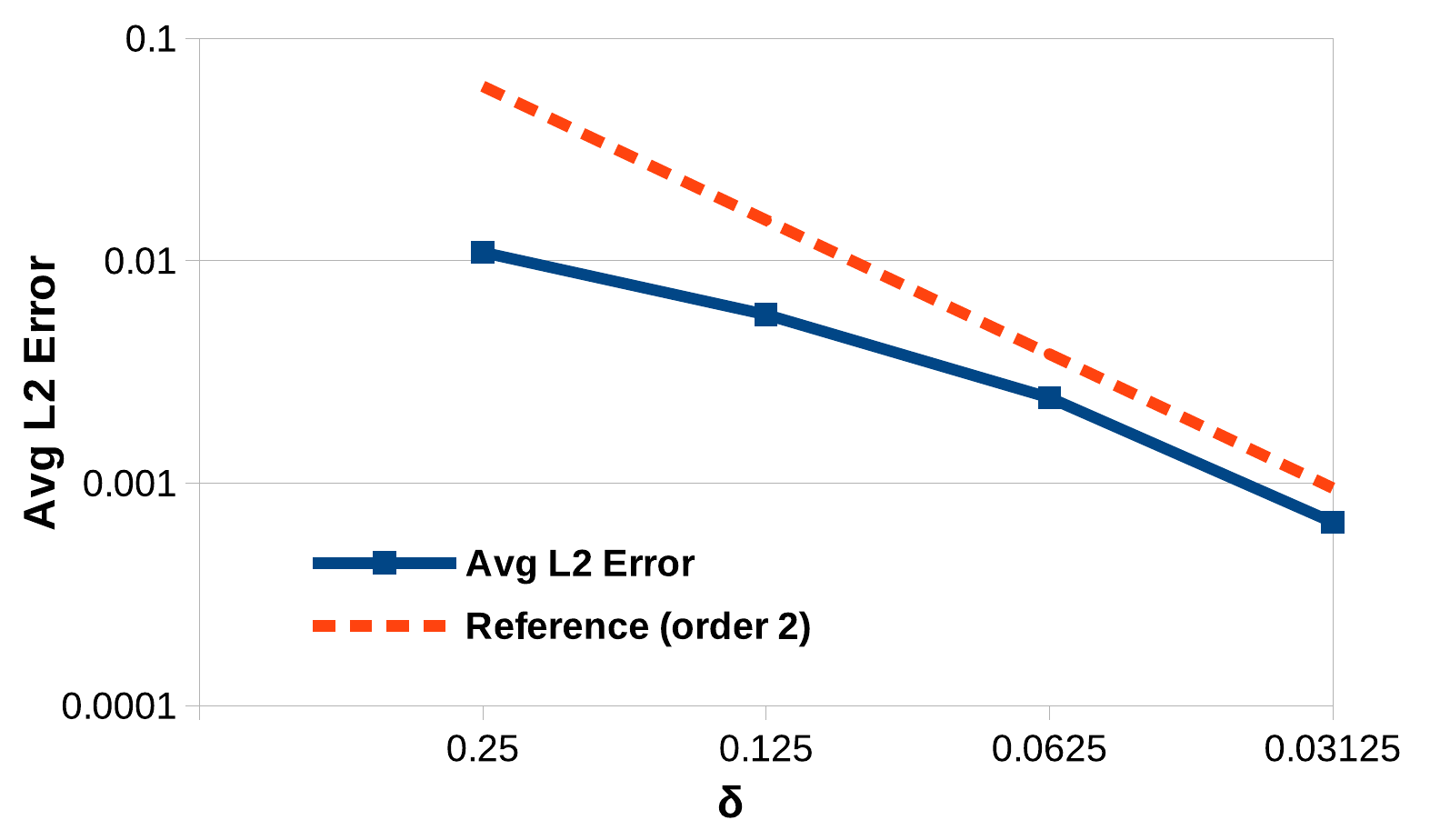}\quad 
\includegraphics[scale=0.37]{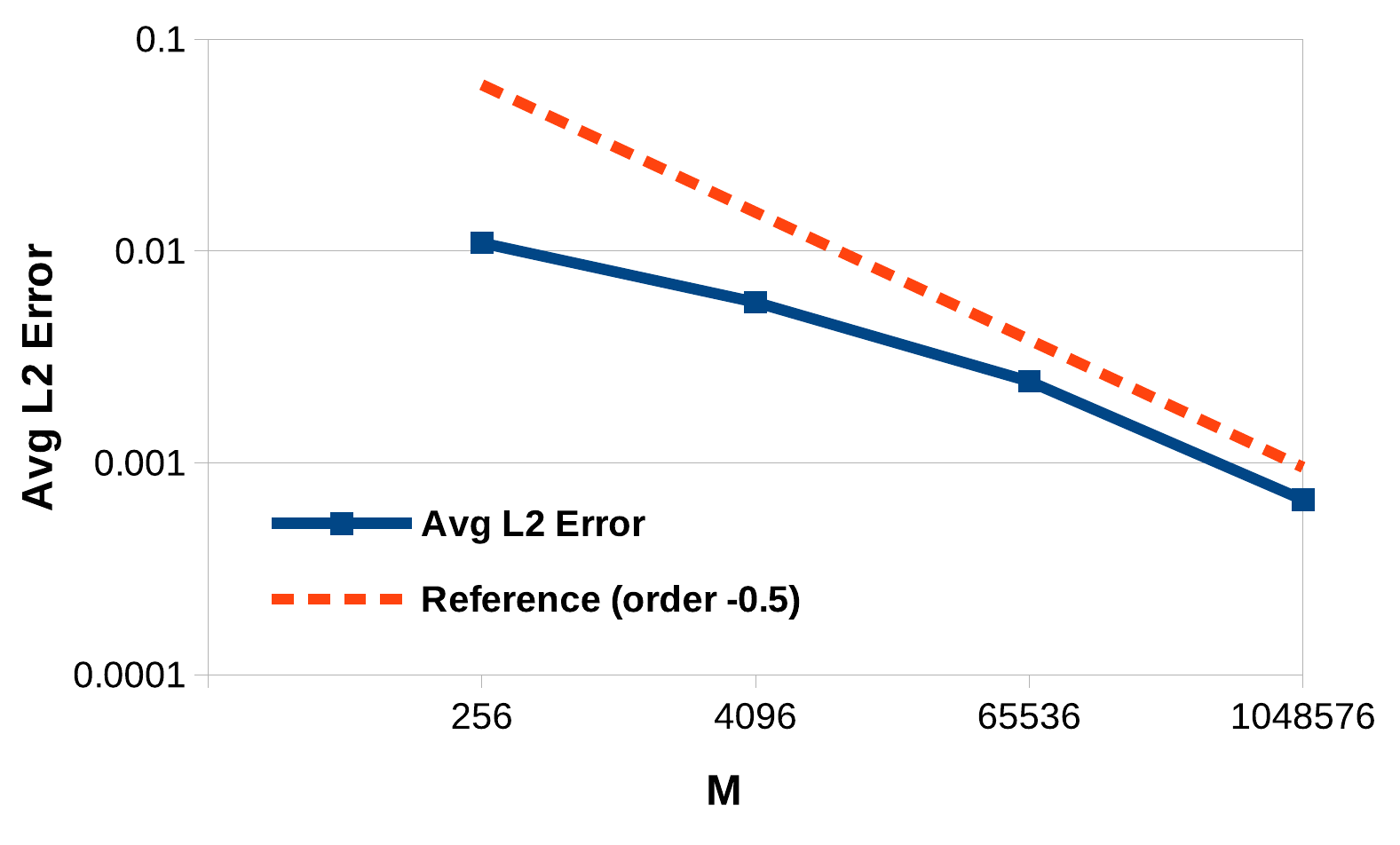}
}
\caption{\em For $M$ and $\delta$ related through \eqref{ns1} with $r=2$, convergence rates with respect to $\delta$ (left) and $M$ (right) for the uniform distribution on $[-1,1]$ for an approximation built with respect to the same interval.
}
\label{unif2}
\end{figure}

One can improve on the results of Figure \ref{unif2_nogood}, even if one does not know the support of the PDF one is trying to approximate, by taking advantage of the fact that the samples obtained necessarily have to belong to the support of the PDF and therefore provide an estimate for that support. For instance, for the example we are considering, one could proceed as follows. 
\begin{enumerate}
\item For a chosen $M$, sample $\{Y_m\}_{m=1}^M$ over $[-1.5,1.5]$.

\item Determine the minimum and maximum values $Y_{min}$ and $Y_{max}$, respectively, of the sample set $\{Y_m\}_{m=1}^M$.

\item Choose the number of bins $N_{bins}$ and set $\delta=(Y_{min}-Y_{max})/N_{bins}$.

\item Build the approximation over the interval $[Y_{min},Y_{max}]$ with a bin size $\delta$.
\end{enumerate}
It is reasonable to expect that as $M$ increases, the interval $[Y_{min},Y_{max}]$ becomes a better approximation to the true support interval $[-1,1]$. Figure \ref{yminymax} illustrates the convergence of $[Y_{min},Y_{max}]$ to $[-1,1]$.  Note that because $[Y_{min},Y_{max}]\subset[-1,1]$, the exact PDF is continuous within $[Y_{min},Y_{max}]$. Thus, it is also reasonable to expect that because the approximate PDF is built with respect an interval which is contained within the support of the exact PDF, that there will be an improvement in the accuracy of that approximation compared to that reported in Figure \ref{unif2_nogood} and, in particular, that as one increases $N_{bins}$ so that $\delta$ decreases and $M$ increases, better rates of convergence will be obtained. Figure \ref{inbetween} corresponds to the application of this procedure and shows the substantially smaller errors and substantially higher convergence rates compared that reported in Figure \ref{unif2_nogood}.

\begin{figure}[h!]
\centerline{
\includegraphics[scale=0.37]{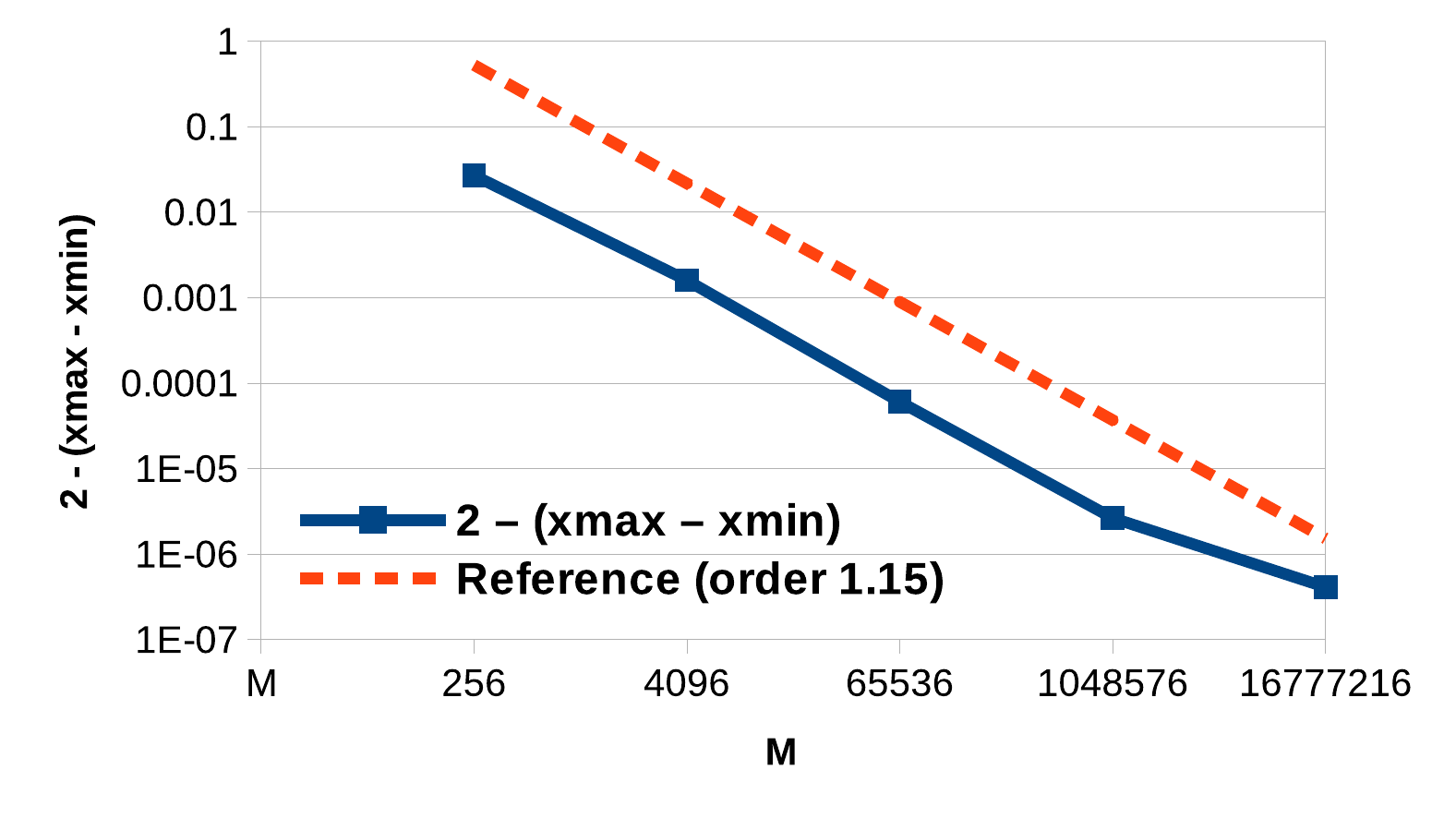} 
}
\caption{\em For the uniform distribution and for $M$ and $\delta$ related through \eqref{ns1} with $r=2$, convergence of the approximate support interval $[Y_{min},Y_{max}]$ to the exact support interval $[-1,1]$.
}
\label{yminymax}
\end{figure}

\begin{figure}[h!]
\centerline{
\includegraphics[scale=0.37]{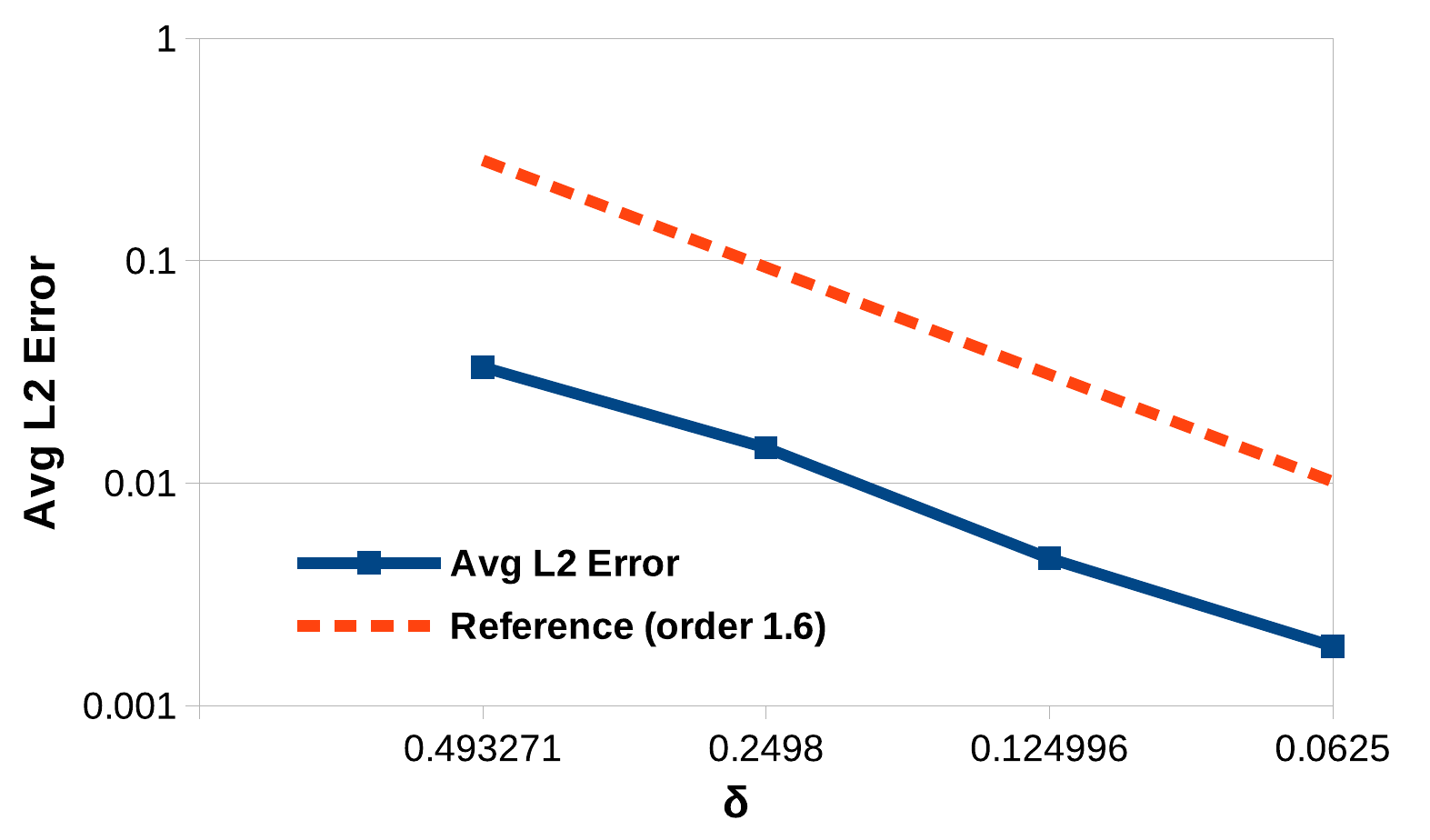}\quad 
\includegraphics[scale=0.37]{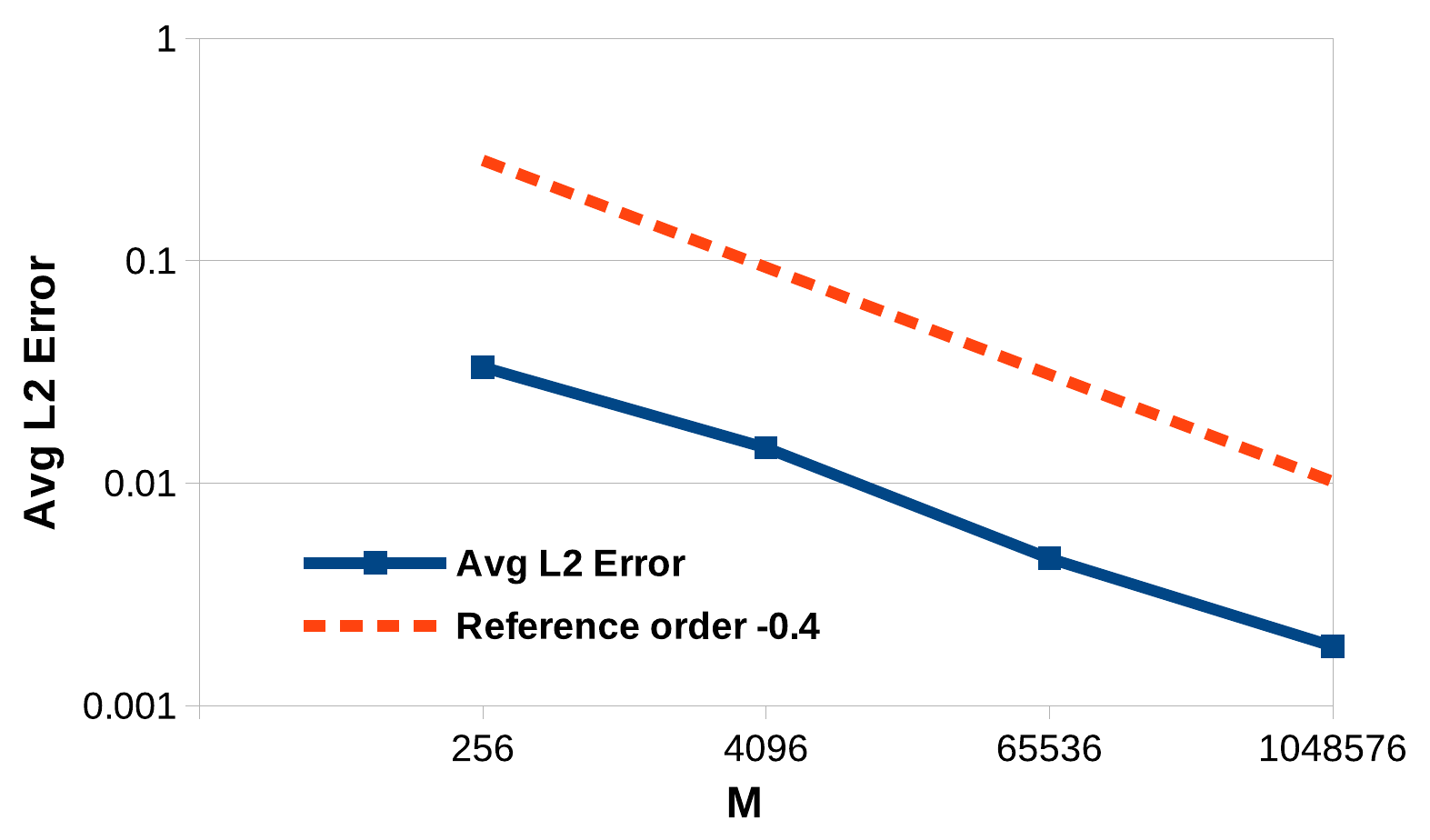}
}
\caption{\em For $M$ and $\delta$ related through \eqref{ns1} with $r=2$, convergence rates with respect to $\delta$ (left) and $M$ (right) for the uniform distribution on $[-1,1]$ with approximations built with respect to the approximate support interval $[Y_{min},Y_{max}]$.
}
\label{inbetween}
\end{figure}

A visual comparisons of the approximations obtained using the smallest $\delta$/largest $M$ pairing corresponding to Figures \ref{unif2_nogood}, \ref{unif2}, and \ref{inbetween} are given in Figure \ref{3ways}. The defects resulting from the use of the interval $[-1.5,1.5]$ for constructing the approximation of a uniform PDF that has support on the interval $[-1,1]$ are clearly evident. On the other had, using the support interval approximation process outlined above results in a visually identical approximation as that obtained using the correct support interval $[-1,1]$. Note that for the smallest value of $\delta$, we have that $Y_{min}$ approximates $-1$ and $Y_{max}$ approximates $1$ to seven decimal places.

\begin{figure}[h!]
\centerline{
\includegraphics[width=1.6in]{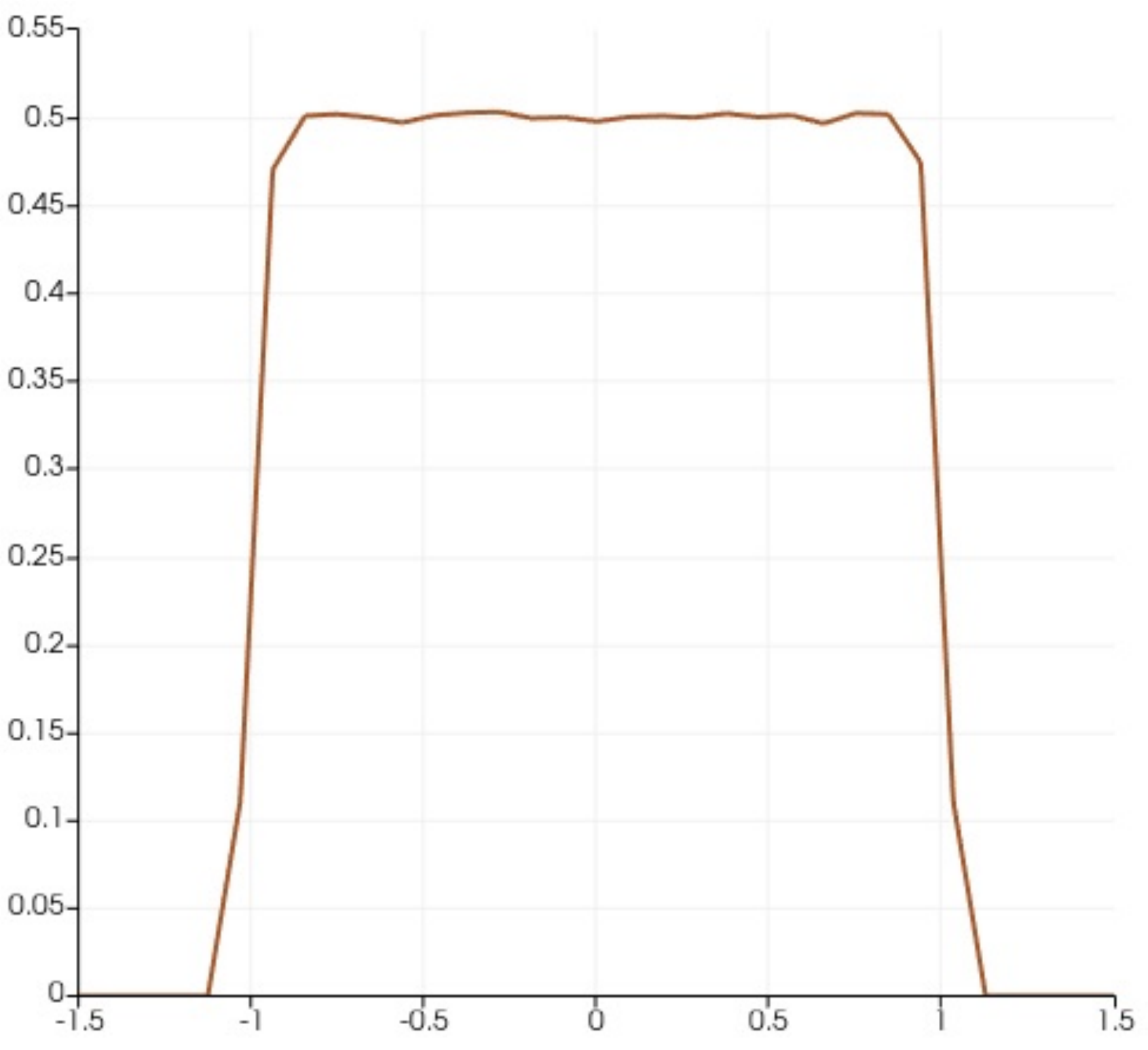}\quad
\includegraphics[width=1.6in]{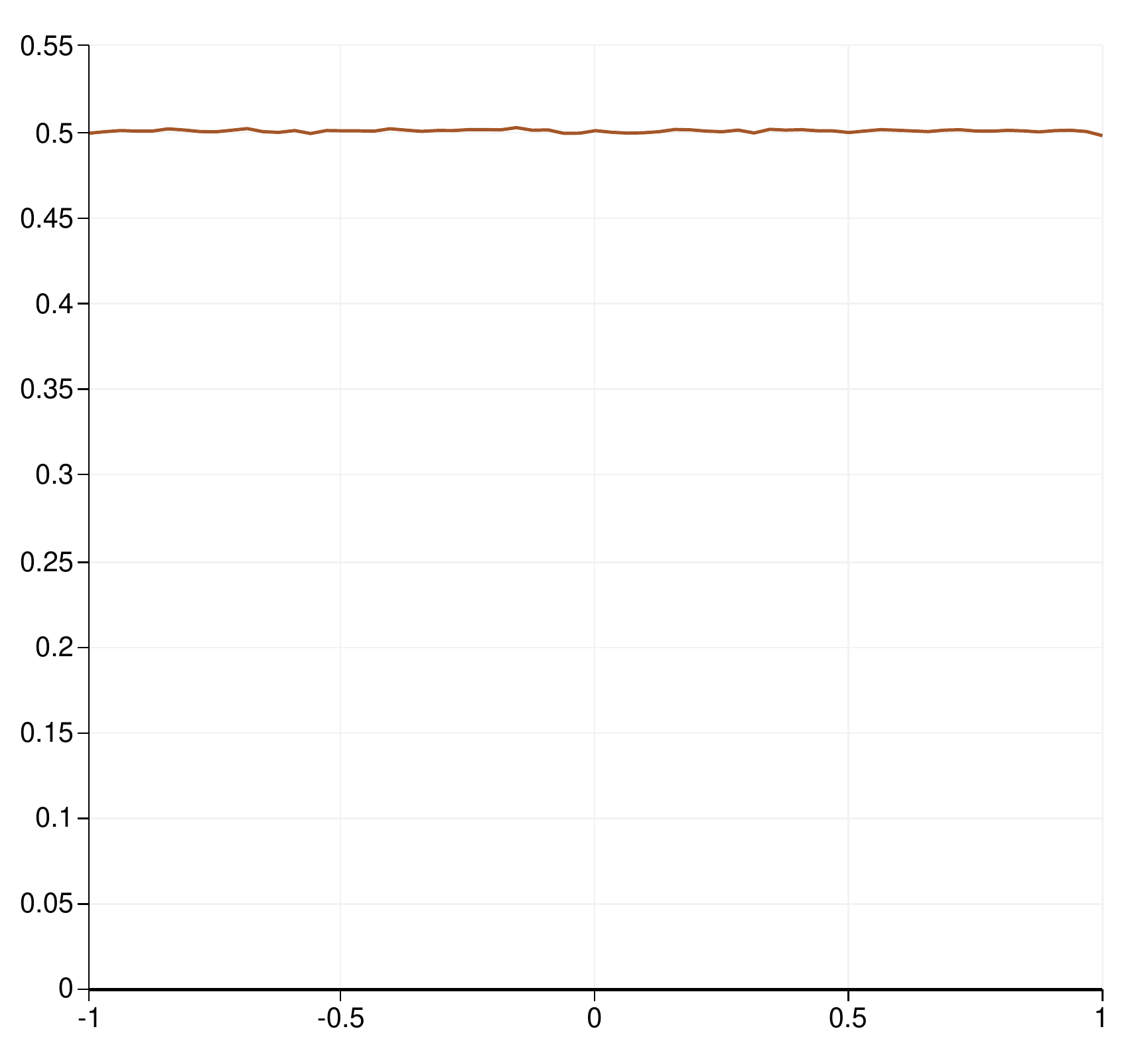}\quad
\includegraphics[width=1.6in]{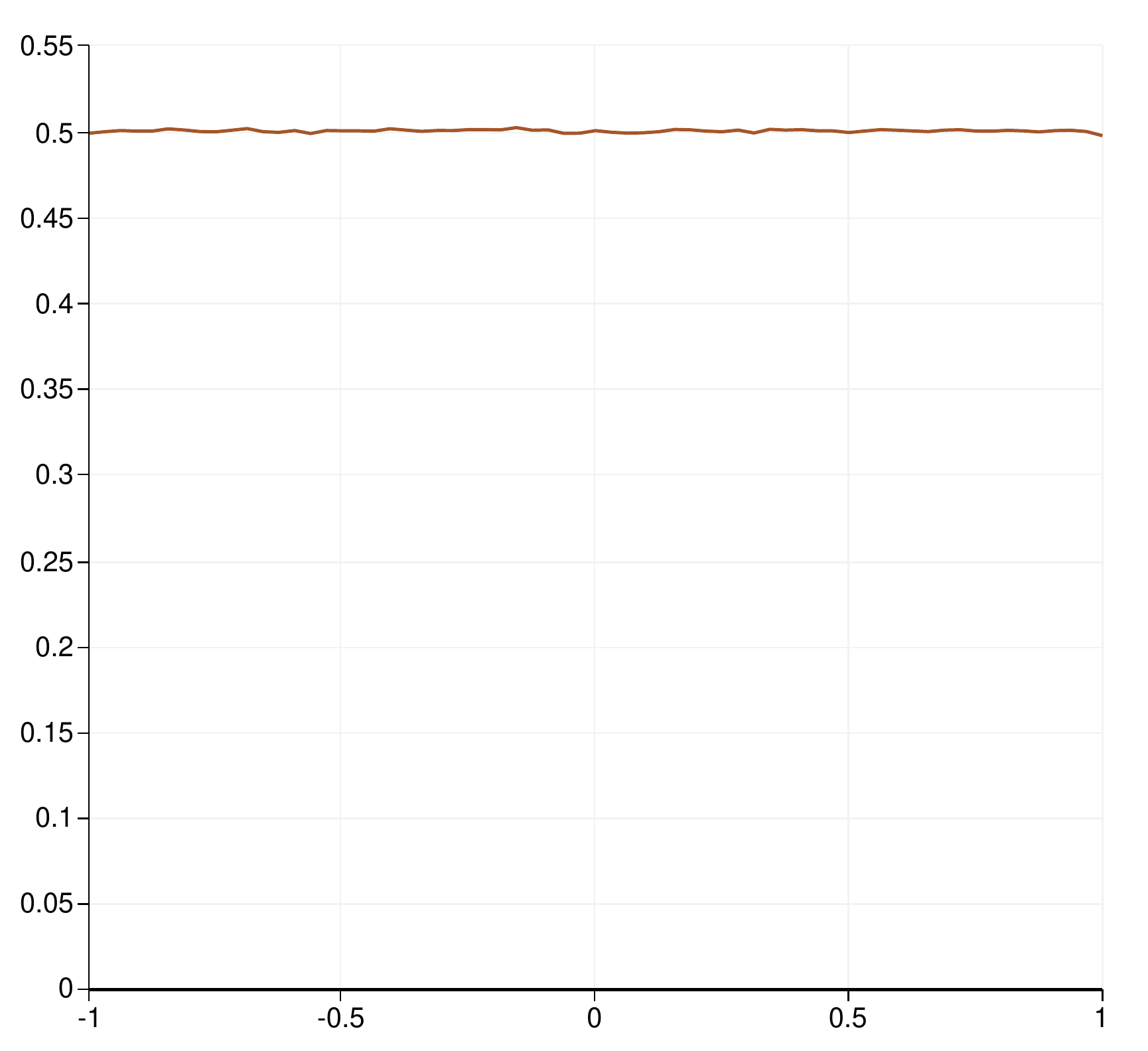}
}
\caption{\em For the uniform distribution on $[-1,1]$, the approximate PDF based on sampling in $[-1.5,1.5]$ (left), $[-1,1]$ (center), and $[Y_{min},Y_{max}]$ (right).}\label{3ways}
\end{figure}

\subsection{A non-smooth PDF}\label{nosmooth}

We next consider the approximation of a non-smooth PDF. Specifically, we consider the centered truncated Laplace distribution
\begin{equation}\label{LaplacePDF}
 f(Y) = \dfrac{1}{3 C_L} \exp{\Big(\dfrac{-|Y|}{1.5}\Big)}
\end{equation}
over $\Gamma=[-5.5,5.5]$, where $C_L = 1 - \exp(-5.5/1.5)$ is a scaling factor that ensures a unitary integral of the PDF over $\Gamma$. Here, the support domain and sampling domain are the same. This distribution is merely continuous. i.e., its derivative is discontinuous at $Y=0$, so one cannot expect optimally accurate approximations. However, as illustrated in Figure \ref{Laplace}, it seems the approximation does converge, but at a lower rate with respect to $\delta$ and at the optimal rate with respect to $M$. The latter is not surprising because Monte Carlo sampling is largely impervious to the smoothness or lack thereof of the function being approximated.

\begin{figure}[h!]
\centerline{
\includegraphics[scale=0.37]{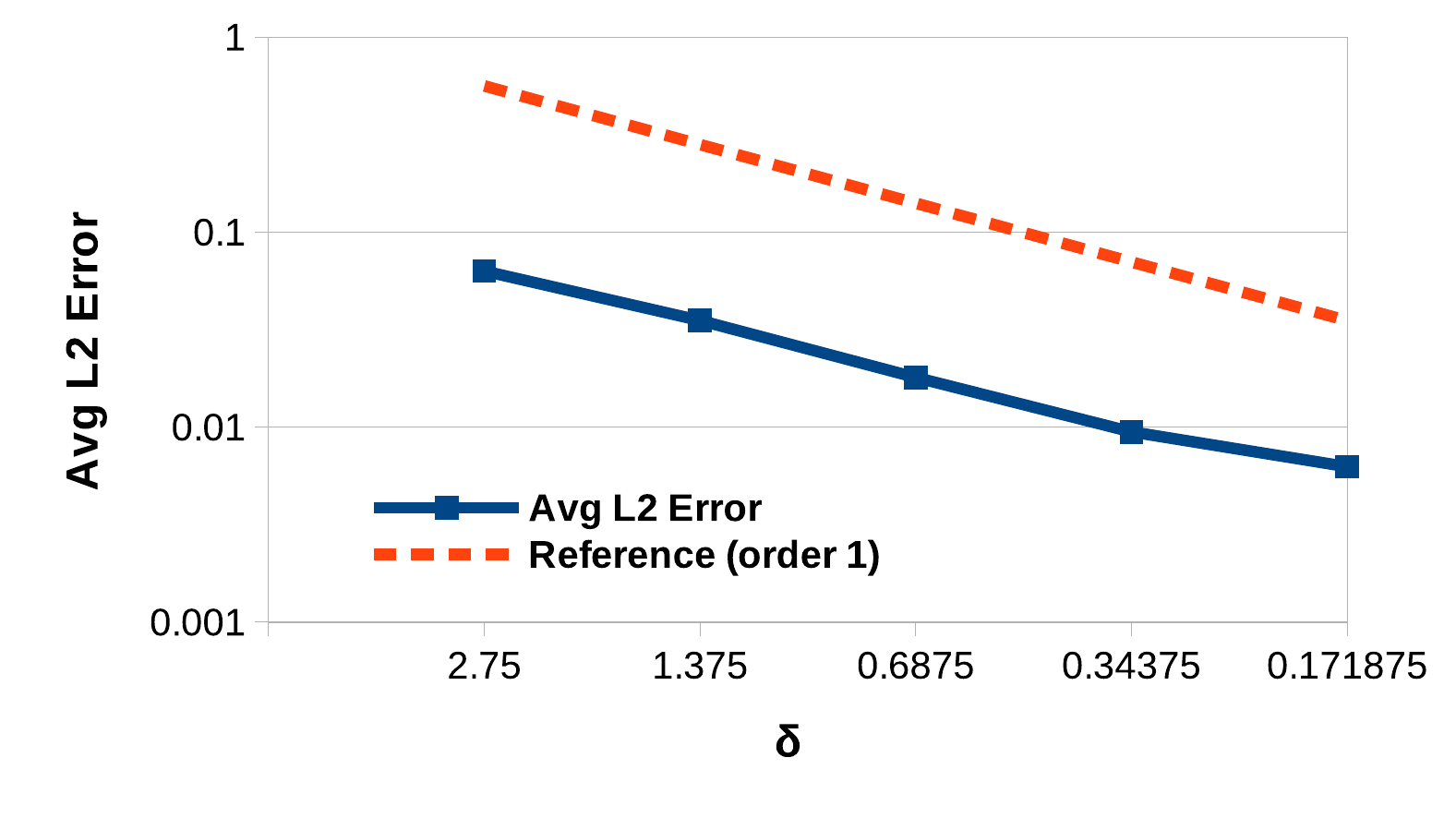}\quad 
\includegraphics[scale=0.37]{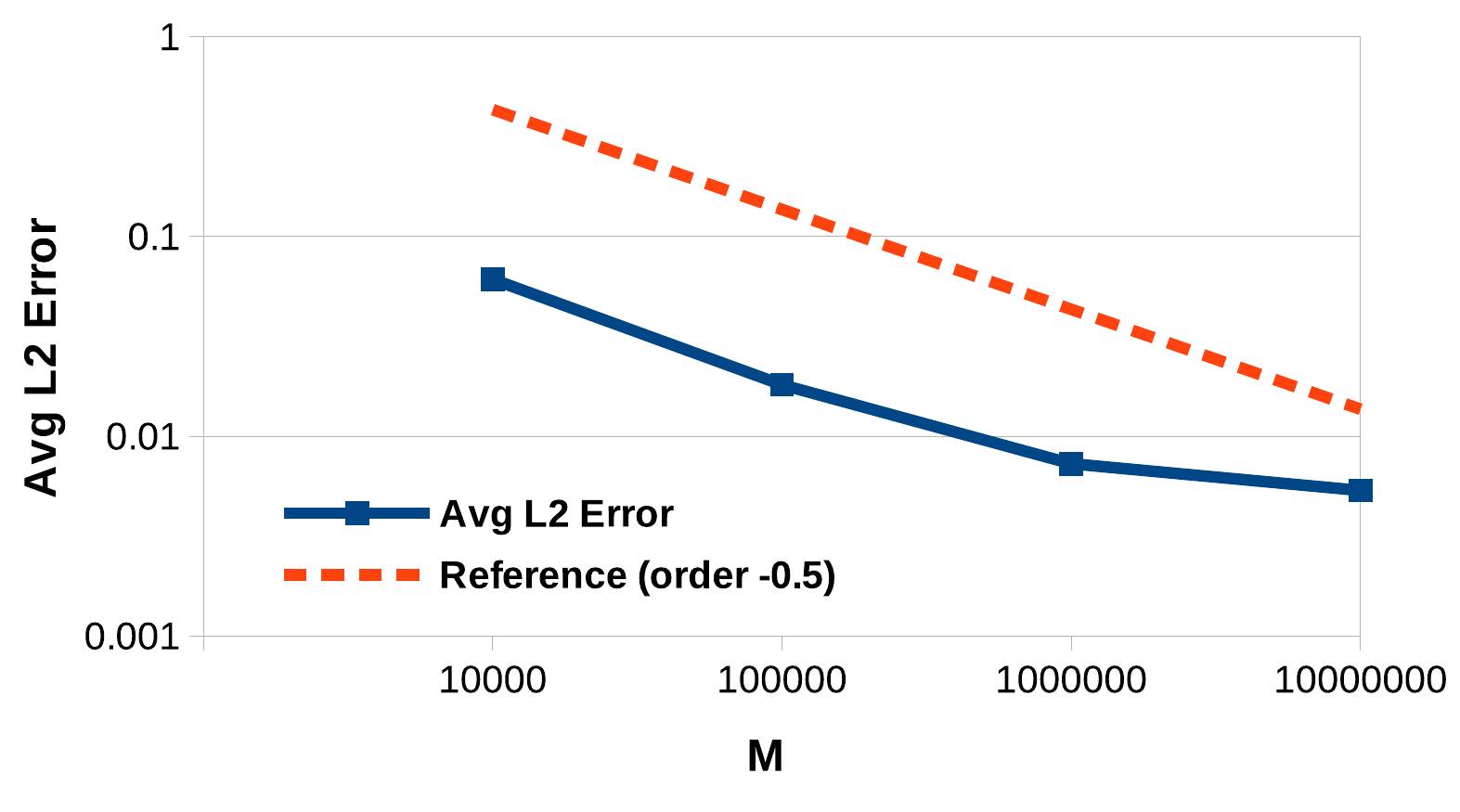}
}
\caption{\em Errors and convergence rates for the approximation \eqref{proposed_f} of the Laplace distribution \eqref{LaplacePDF}. Left: convergence with respect to $\delta$ with $M=10^7$ is fixed. Right: convergence with respect to $M$ with $h=11/2^{12}$ fixed.}
\label{Laplace}
\end{figure}

Whenever there is any information about the smoothness of the PDF, one can choose an appropriate value of $r$ in \eqref{ns1}. Alternately, possibly through a preliminary investigation, one can estimate the convergence rate of the approximation \eqref{proposed_f}. In the case of the Laplace distribution which is continuous but not continuously differentiable, one cannot expect a convergence rate greater than one. Selecting $r=1$ in \eqref{ns1} to relate $M$ and $\delta$, we obtain the results given in Figure \ref{Laplace2} which depicts rates somewhat worse that we should perhaps expect.

\begin{figure}[h!]
 \begin{center}
\includegraphics[scale=0.37]{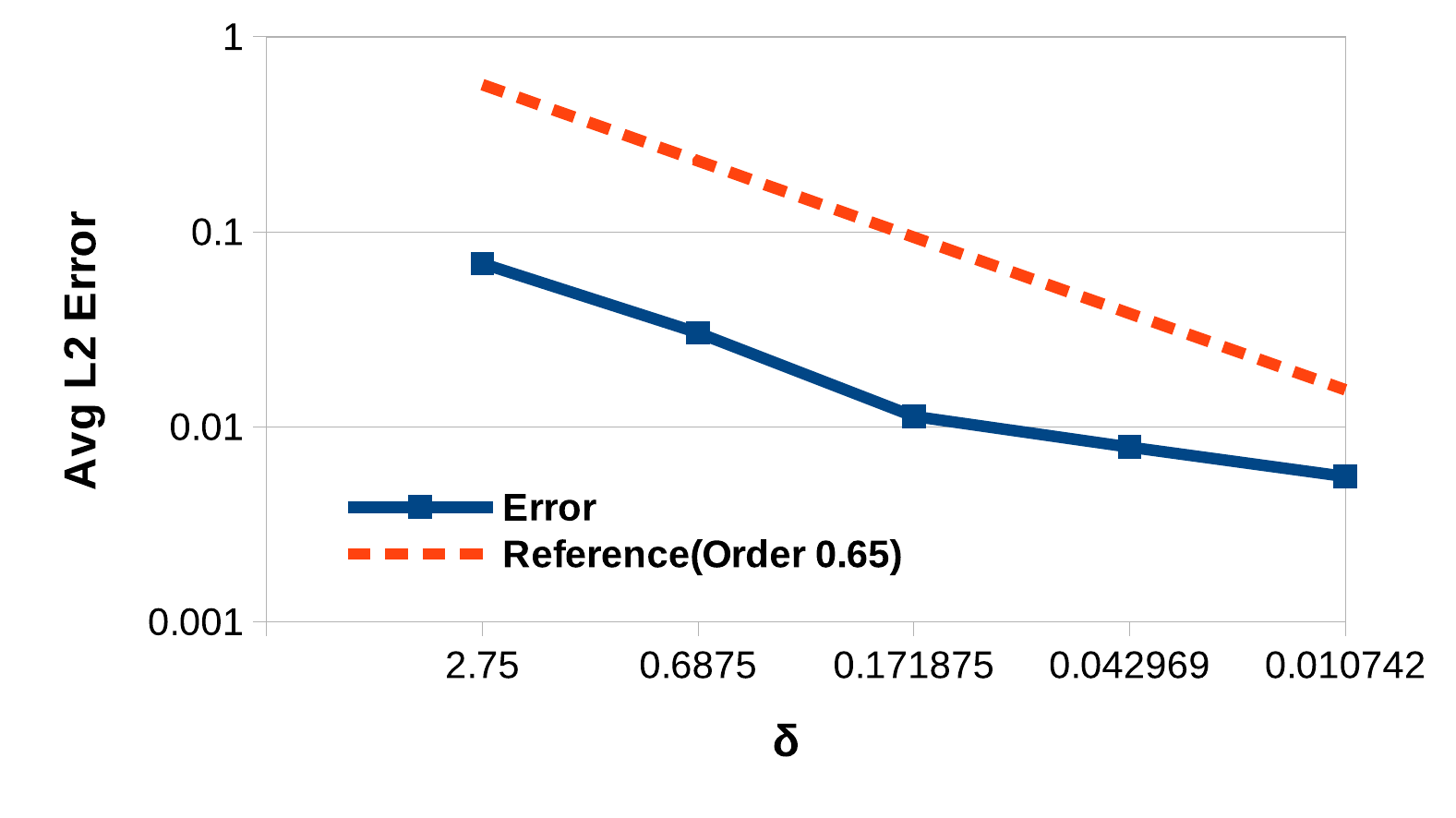}\quad 
\includegraphics[scale=0.37]{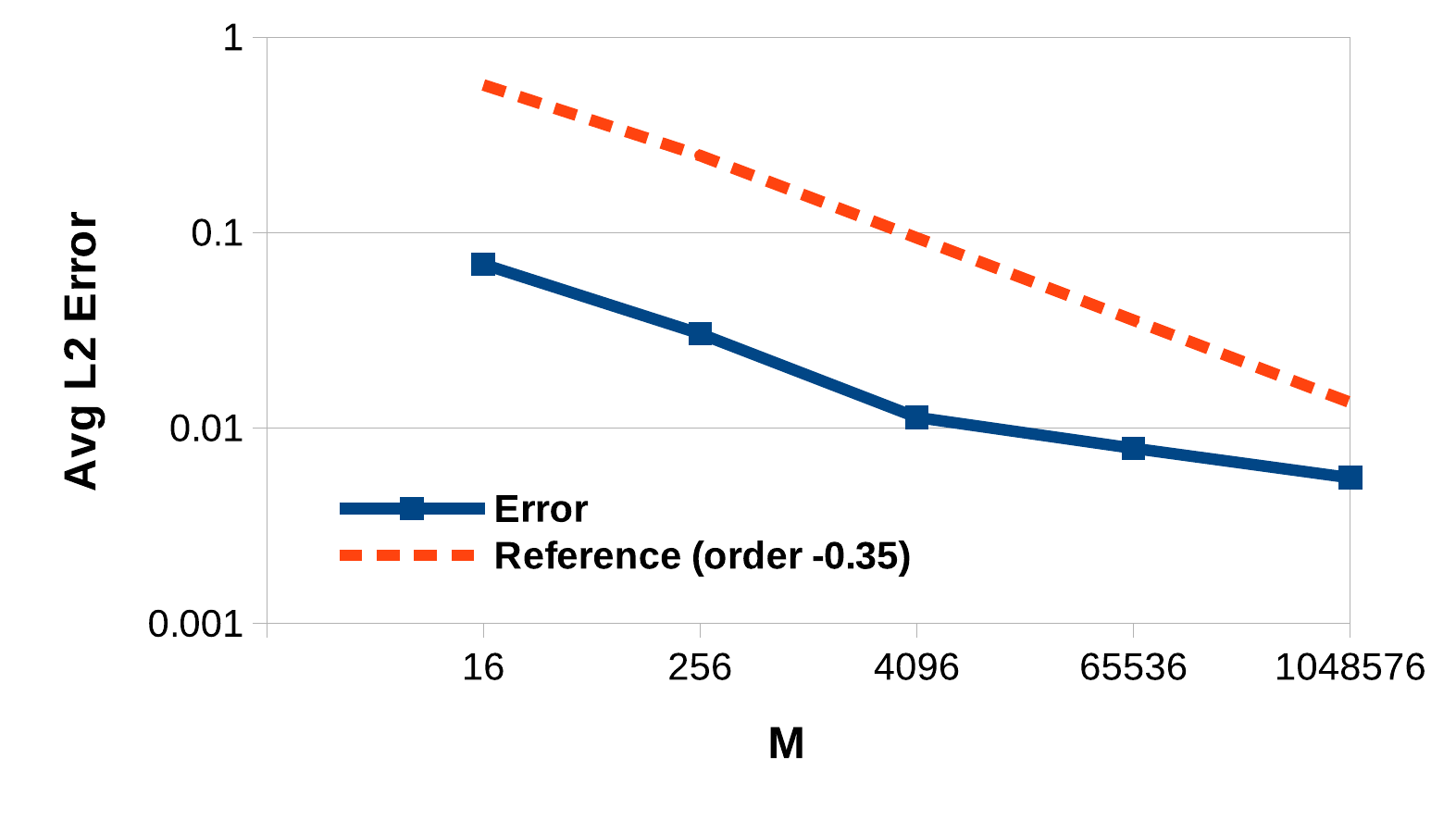}
\caption{\em For $M$ and $\delta$ related through \eqref{ns1} with $r=1$, convergence rates with respect to $\delta$ (left) and $M$ (right) for the Laplace distribution with approximations built with respect to the approximate support interval $\Gamma=[-5.5,5.5]$.}
\label{Laplace2}
\end{center}
\end{figure}

The Laplace distribution, although not globally $C^2$, is piecewise smooth, with failure of smoothness only occurring at the symmetry point of the distribution. For example, for the particular case of the centered distribution \eqref{LaplacePDF}, the distribution is smooth for $Y>0$ and $Y<0$. Thus, in general, one could build two separate, optimally accurate approximations, one for the right of the symmetry point and the other for the left of that point. Of course, doing so requires knowledge of where that point is located. If this information is not available, then one can estimate the location of that point by a process analogous to what we described in Section \ref{unpdf} for distributions whose support is not known a priori. Such a process can be extended to distributions with multiple points at which smoothness is compromised.

\subsection{Bivariate mixed PDF}\label{bivpdf}

We now consider a bivariate PDF in which the random variables $Y_1$ and $Y_2$ are independently distributed according to different PDFs. Specifically, we have that $Y_1$ is distributed according to a truncated Gaussian distribution with zero mean and standard deviation $2$, whereas $Y_2$ is distributed according to a truncated standard Gaussian. We choose $\Gamma=[-5.5,5.5]^2$ so that the joint PDF is given by 
\begin{equation}\label{jointpdf1}
 f({\bm Y}) = \dfrac{1}{\sqrt{8 \pi} C'_{G}} \exp{\Big(-\frac{Y_1^2}{8}\Big)}  \dfrac{1}{\sqrt{2 \pi} C_{G}} \exp{\Big(-\frac{Y_2^2}{2}\Big)},
\end{equation}
where $C_G$ is as in \eqref{stdGaussiantrunc} and $C'_{G} = 1/2(\mbox{erf}(2.75/\sqrt{2}) - \mbox{erf}(-2.75/\sqrt{2}))$. Results for this case are shown in Figure \ref{mixed_rate}, where we observe optimal convergence rates with respect to both $\delta$ and $M$. Visual evidence of the accuracy of our approach is given in Figure \ref{2D_mixed} that shows the approximation of the exact PDF \eqref{jointpdf1} and zoom-ins of the approximate and approximate PDFs. Computational times are very similar to those for $N_\Gamma=2$ in Table \ref{timeGauss} so that they are not provided here. 

\begin{figure}[h!]
\centerline{
\includegraphics[scale=0.37]{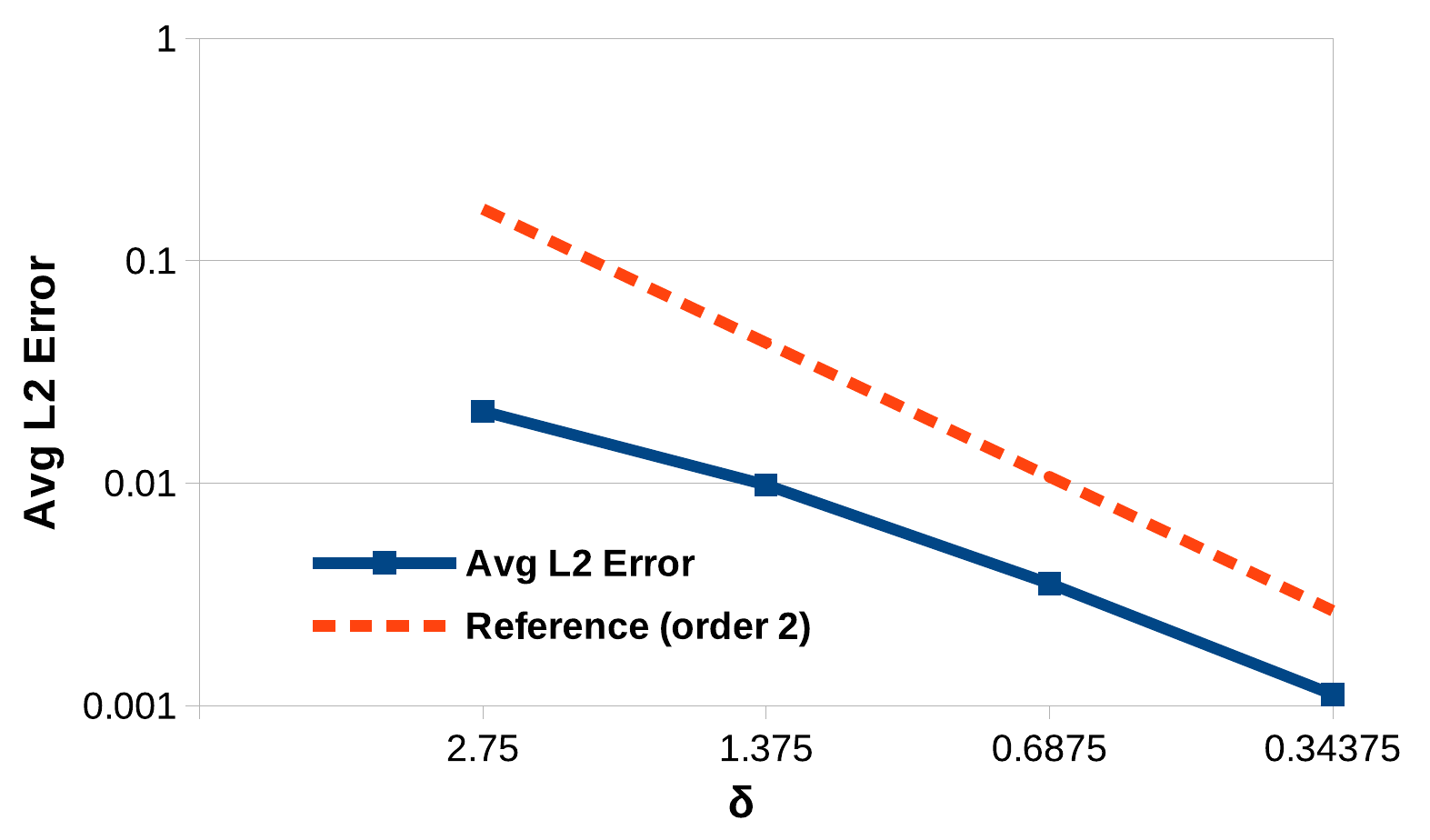}\quad 
\includegraphics[scale=0.37]{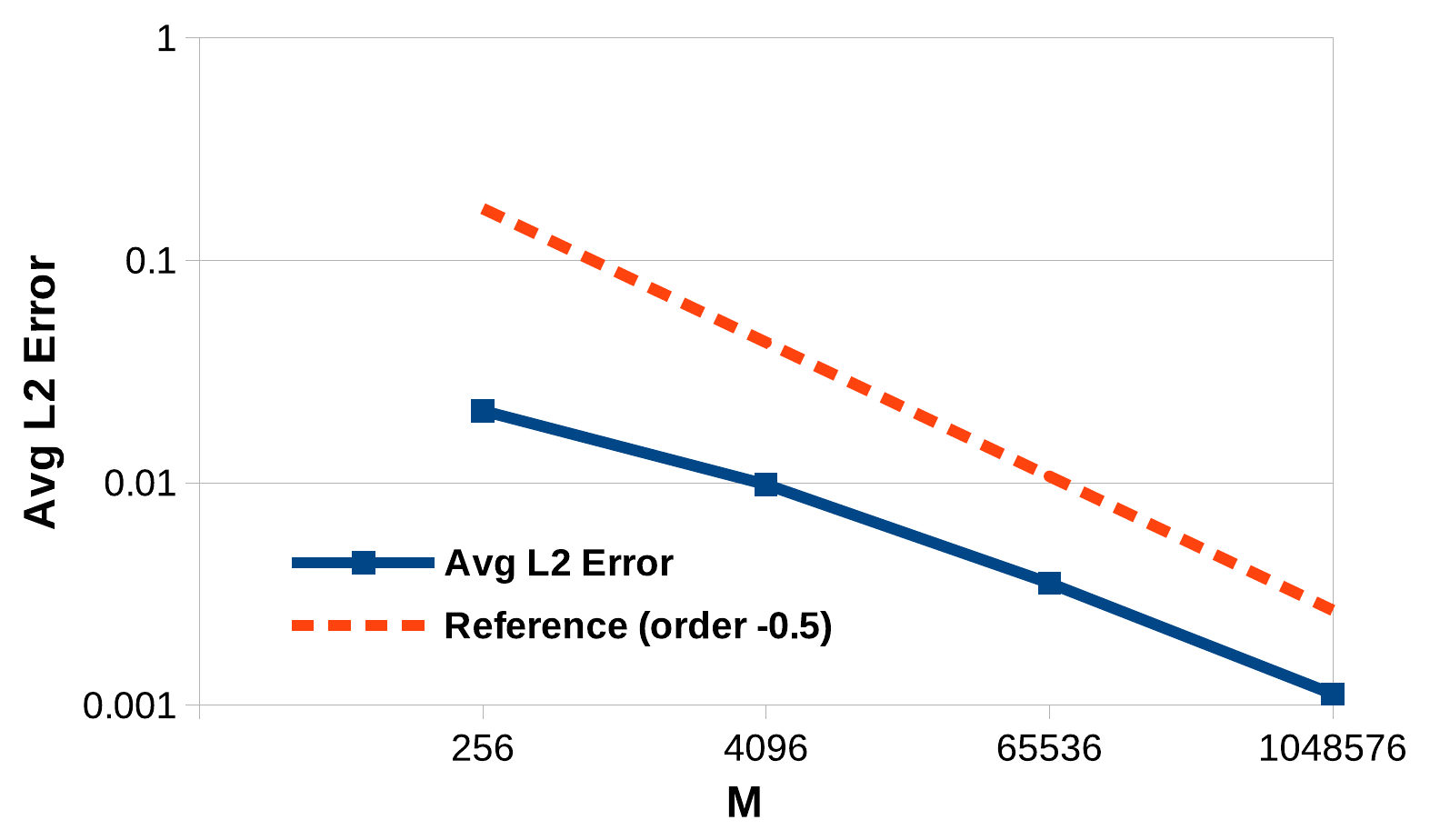}
} 
\caption{\em Errors and convergence rates for the approximation \eqref{proposed_f} of the bivariate mixed-distribution PDF \eqref{jointpdf1}. Left: convergence with respect to $\delta$ with $M=10^7$ is fixed. Right: convergence with respect to $M$ with $h=11/2^8$ fixed.}
\label{mixed_rate}   
\end{figure}

\begin{figure}[h!]
\centerline{
\includegraphics[height=1.7in,width=1.93in]{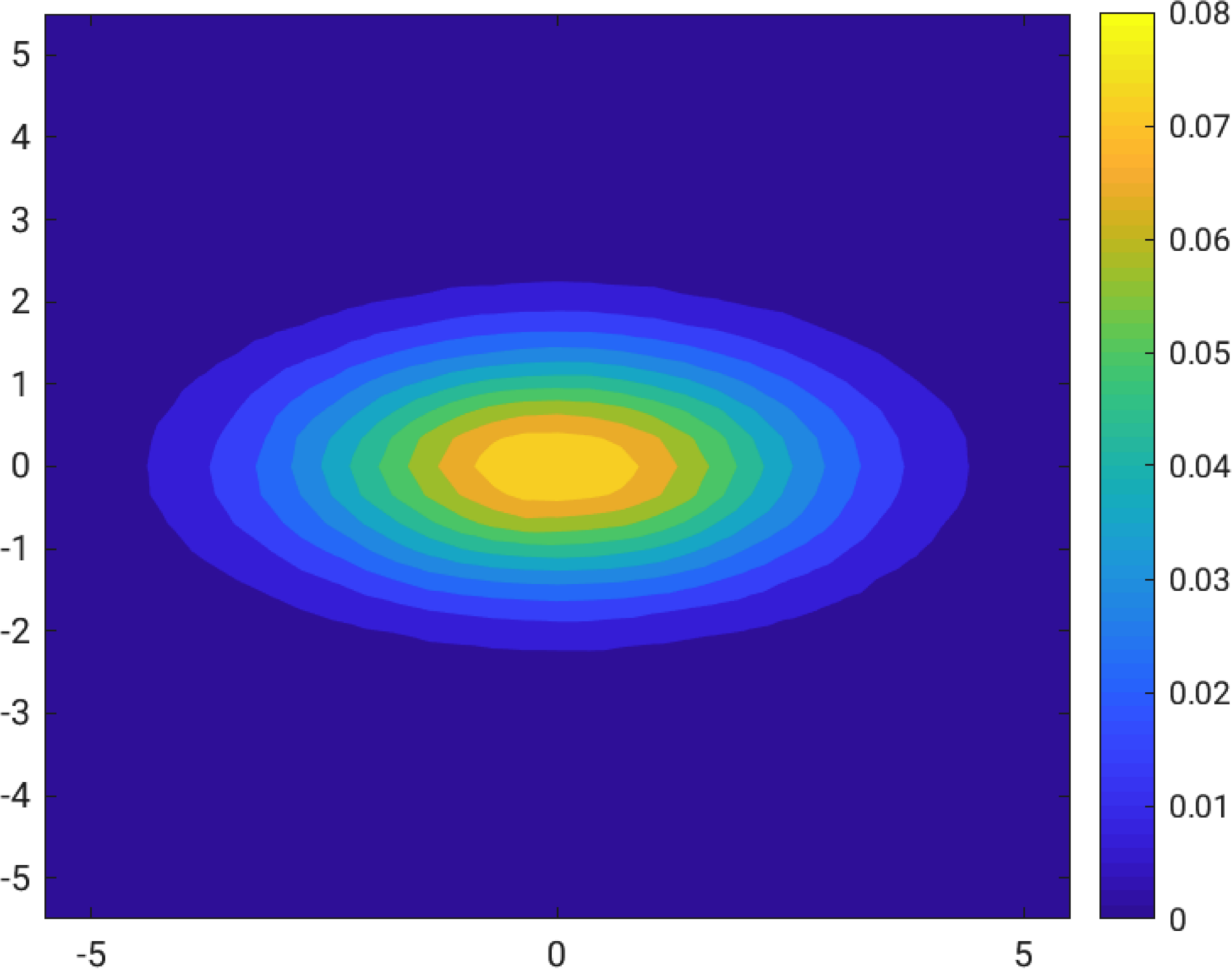}
\qquad\includegraphics[height=1.7in]{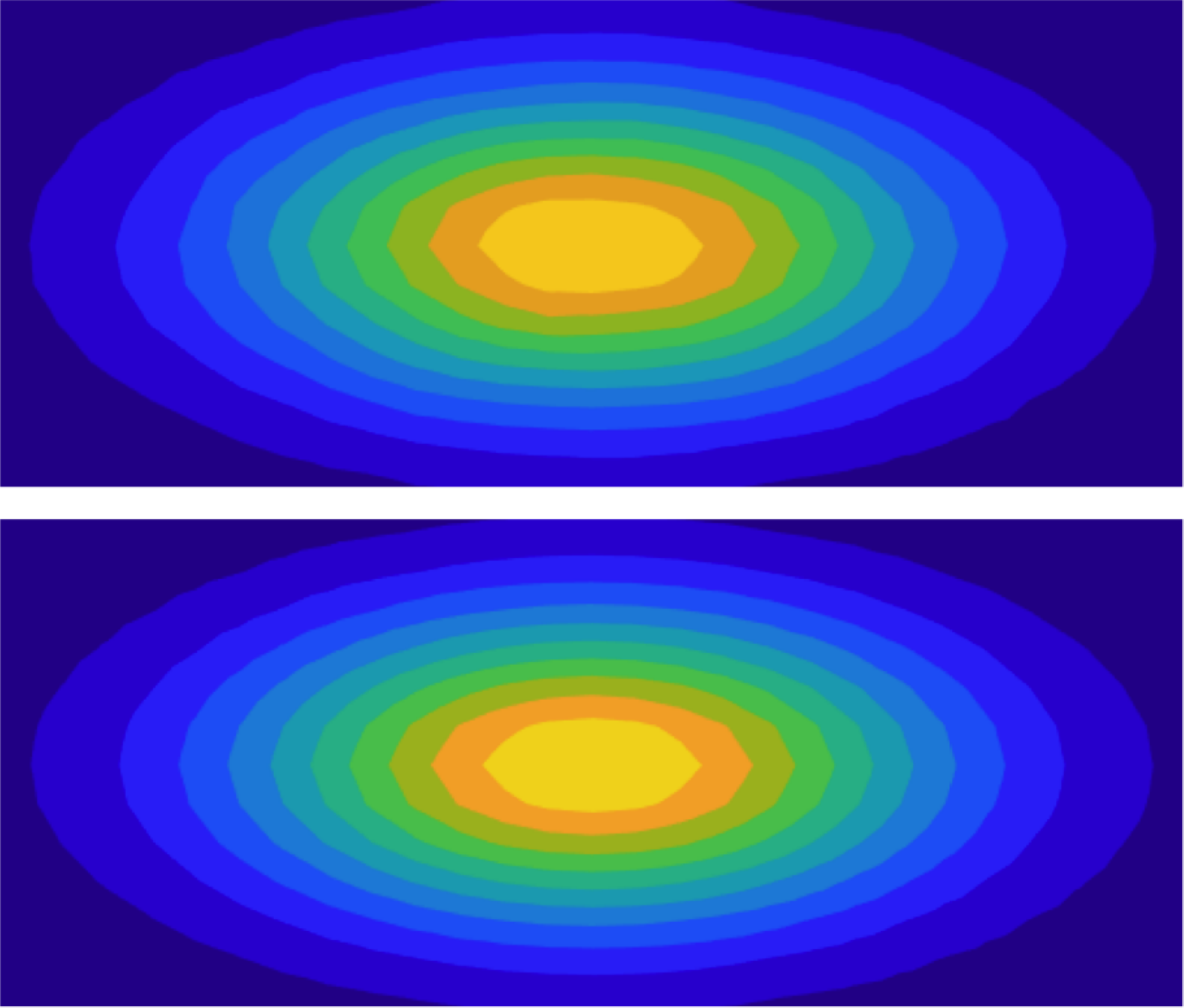}
}
\caption{\em Left: the approximation \eqref{proposed_f} of the bivariate mixed-distribution PDF \eqref{jointpdf1}. Right-top: a zoom in of the exact PDF. Right-bottom: a zoom in of the approximate PDF. For these plots, $\delta=0.34375=11/2^5$ and $M=1048576=16^5$.}
\label{2D_mixed} 
\end{figure}

\section{Application to an unknown PDFs associated with a stochastic PDE}\label{updf}

In this section, we consider the construction of approximations of the PDF of outputs of interest that depend on the solution of a stochastic PDE. In general, such PDFs are unknown a priori.

The boundary value problem considered is the stochastic Poisson problem
\begin{align}\label{poissonSys}
\begin{cases}
- \nabla \cdot \big(\,\kappa(\textbf{x}, {\bm Z}) \,\, \nabla u(\textbf{x}, {\bm Z}) \,\big)  = 1 & \text{for} \quad \textbf{x}\in D ,\,\,{\bm Z}\in \Gamma_{input} \\
 u(\textbf{x}, {\bm Z})=0 & \text{for} \quad \textbf{x}\in\partial D ,\,\,{\bm Z}\in \Gamma_{input},
\end{cases} 
\end{align}
where $D \subset \mathbb{R}^d$ denotes a spatial domain with boundary $\partial D$ and $\Gamma_{input}\subset \mathbb{R}^{N_{\Gamma_{input}}}$ is the sample space for the {\em input} random vector variable ${\bm Z}$ which we assume is distributed according to a {\em known input joint PDF} $f_{input}({\bm Z})$. 

For the coefficient function $\kappa(\textbf{x}, {\bm Z})$, we assume that there exists a positive lower bound $\kappa_{\min}>0$ almost surely on $\Gamma_{input}$ for all $\textbf{x} \in D$. We also assume that $\kappa(\textbf{x}, {\bm Z})$ is measurable with respect to ${\bm Z}$.  It is then known that the system \eqref{poissonSys} is well posed almost surely for ${\bm Z}\in\Gamma_{input}$; see, e.g.,
\cite{babuvska2007stochastic, gunzburger2014stochastic, nobile2008anisotropic, nobile2008sparse} for details. 

\vskip10pt
\paragraph{\textbf{Stochastic Galerkin approximation of the solution of the PDE}}

We assume that we have in hand an approximation $u_{approx}(\textbf{x},{\bm Z})$ of the solution $u(\textbf{x}, {\bm Z})$ of the system \eqref{poissonSys}. Specifically, spatial approximation is effected via a piecewise-quadratic finite element method \cite{brenner2007mathematical,ciarlet}. Because this aspect of our algorithm is completely standard, we do not give further details about how we effect spatial approximation. For stochastic approximation, i.e., for approximation with respect to the parameter domain $\Gamma_{input}$, we employ a spectral method. Specifically, we approximate using global  orthogonal polynomials, where orthogonality is with respect to $\Gamma_{input}$ and the known PDF $f_{input}({\bm Z})$. Thus, if $\{\Phi_j({\textbf x})\}$ denotes a basis for the finite element space used for spatial approximation and $\{\Psi_i({\bm Z})\}$ denotes a basis for the spectral space used for approximation with respect to ${\bm Z}$, with the stochastic Galerkin method (SGM) we obtain an approximation of the form 
\begin{equation}\label{SGMSGM}
\begin{aligned}
 u_{approx}(\mathbf{x}, {\bm Z}) &= \sum_i \sum_{j} U_{i,j} \, \Phi_j(\mathbf{x}) \Psi_i({\bm Z})
\\&= \sum_i u_i(\mathbf{x}) \Psi_i({\bm Z}),
\qquad\mbox{where}\qquad
 u_i(\mathbf{x}) = \sum_{j} U_{i,j} 
\Phi_j(\mathbf{x}) .
\end{aligned}
\end{equation}
Note that once the SGM approximation \eqref{SGMSGM} is constructed, it may be evaluated at any point ${\textbf x}\in D$ and for any parameter vector ${\mathbf Z}\in\Gamma_{input}$. Having chosen the types of spatial and stochastic approximations we use, the approximate solution $u_{approx}(\textbf{x},{\bm Z})$, i.e., the set of coefficients $\{U_{i,j}\}$, is determined by a stochastic Galerkin projection, i.e., we determine an approximation to the Galerkin projection of the exact solution with respect to both the spatial domain $D$ and the parameter domain $\Gamma_{input}$. This approach is well documented so that we do not dwell on it any further; one may consult, e.g., \cite{babuska2004galerkin, babuvska2005solving, capodaglio2018approximation, ghanem1991stochastic, gunzburger2014stochastic}, for details. Note that once the surrogate \eqref{SGMSGM} for the solution of the PDE is built, it can be used through direct evaluation to cheaply determine an approximation of the solution of the PDE for any ${\textbf x}\in\ D$ and any ${\bm Z}\in\Gamma_{input}$ instead of having to do a new approximate PDE solve for any new  choice of ${\textbf x}$ and ${\bm Z}$.

The error in $u_{approx}(\textbf{x},{\bm Z})$ depends on the grid-size parameter $h$ used for spatial approximation and the degree of the orthogonal polynomials used for approximation with respect to the input parameter vector ${\bm Z}$. In practice, these parameters should be chosen so that the two errors introduced are commensurate. However, here, because our focus is on stochastic approximation and because throughout we use the same finite element method for spatial approximation, we use a small enough spatial grid size so that the error due to spatial approximation is, for all practical purposes, negligible compared to the errors due to stochastic approximation.

\vskip10pt
\paragraph{\textbf{Outputs of interest depending on the solution of the PDE}}

In our context, outputs of interest are spatially-independent functionals of the solution $u(\textbf{x},{\bm Z})$ of the system \eqref{poissonSys}. Here, we consider the two specific functionals
\begin{align}\label{QoI}
 Y({\bm Z}) 
 =\sum_i \Big(\dfrac{1}{|D|}\int\limits_{D} u_i(\mathbf{x}) d \mathbf{x}\Big) \Psi_i({\bm Z})
\end{align}
or
\begin{align}\label{QoI2}
 Y({\bm Z}) =\sum_i \Big(\int\limits_{D} u_i^2(\mathbf{x}) d \mathbf{x}\Big) \Psi_i({\bm Z}),
\end{align}
i.e., the average of the approximate solution and one involving the integral of the square of $u_i$, respectively, over the spatial domain $D$. The output of interest $Y$ is a random variable that depends on the random input vector ${\bm Z}$. Note that although we consider scalar outputs of interest, the extension to vector-valued outputs of interest is straightforward. Note that throughout, all outputs of interest are standardized, namely, they are translated by their mean and scaled by their standard deviation. Hence, we seek approximations of the PDFs of standardized outputs of interest.

It is important to keep in mind that we are dealing with two random variables. First, we have the {\em input} random variable ${\mathbf Z}$ having a known PDF $f_{input}({\mathbf Z})$ supported over the known parameter domain $\Gamma_{input}$. Second, we have the {\em output} random variable $Y$ having an unknown PDF $f_{output}(Y)$ supported over an unknown parameter domain $\Gamma_{output}$. Although we do not know the output PDF, in fact that is what we want to construct so that further samples of the output of interest $Y$ can be obtained by simple direct sampling of the PDF $f_{output}(Y)$ for $Y$.

Thus, the task at hand is 
\begin{quote}\normalsize
{\em given the known PDF $f_{input}({\bm Z})$ of the random input ${\bm Z}\in\Gamma_{input}$, determine an approximation of the {\em unknown} PDF $f_{output}(Y)$ of an output of interest $Y({\bm Z})$.}
\end{quote}
To deal with this task, one simply follows the recipe:
\begin{enumerate}
\item construct the stochastic Galerkin approximation $u_{approx}(\textbf{x},{\bm Z})$ given in \eqref{SGMSGM} of the solution $u(\textbf{x}, {\bm Z})$ of the PDE \eqref{poissonSys};
\item choose $M$ samples $\{ {\bm Z}_{m} \}_{m=1}^M$ of the input random vector ${\bm Z}$ according to the known given input PDF $f_{input}({\bm Z})$;
\item determines $M$ samples of the approximate solution $\{u_{approx}(\textbf{x},{\bm Z}_m)\}_{m=1}^M$ of the PDE \eqref{poissonSys} by evaluating \eqref{SGMSGM} at each of the samples ${\bm Z}_m$ chosen in step 2;
\item use the approximate solution samples obtained in step 3 to determine $M$ samples $\{Y_m=Y({\bm Z}_m)\}_{m=1}^M$ of an output of interest from, e.g., \eqref{QoI} or \eqref{QoI2};
\item use the output of interest samples $\{Y_m\}_{m=1}^M$ obtained in step 4 to determine, using \eqref{proposed_f}, an approximation to the output PDF $f_{output}(Y)=f_{output}\big(Y({\bm Z})\big)$.
\end{enumerate}

Of course, because the exact PDF $f_{output}(Y)$ is not known, we cannot use \eqref{al2e} to compute errors. Thus, as a surrogate for the exact PDF, we use a histogram approximation $f^{histo}_{{\widehat M},{\widehat\delta}}(Y)$ obtained with a large number of bins (and therefore a very small ${\widehat\delta}$ and a large number of samples ${\widehat M}$), where ``large'' is relative to what is used in obtaining approximations using, e.g., \eqref{proposed_f}. Thus, we now use
\begin{equation}\label{al2e_histo}
 {\mathcal E}_{f_{approx}} = \Big(\dfrac{1}{M}\sum_{m=1}^M \big(f^{histo}_{{\widehat M},{\widehat\delta}}(Y_m) - {f}_{approx}(Y_m)\big)^2 \Big)^{1/2}
\end{equation}
as a measure of the error in any approximation ${f}_{approx}(Y)$ of ${f}_{output}(Y)$ that involves $M\ll{\widehat M}$ samples and a bin width $\delta\gg{\widehat\delta}$.

\vskip10pt
\paragraph{\textbf{Illustrative results for a specific choice for the coefficient of the PDE}}

For the coefficient function in the PDE \eqref{poissonSys}, we choose 
\begin{eqnarray}\label{coeff}
\kappa(\textbf{x}, {\bm Z}) = \kappa_{\min} \mbox{+} \exp \big(\gamma(\textbf{x}, {\bm Z})\big)
\quad\mbox{with}\quad \gamma(\textbf{x}, {\bm Z})=
\mu\,\, \mbox{+} \sum_{n=1}^{{N_{input}}} \sqrt{\lambda_n} \psi_n({\textbf{x}}) Z_{n},
\end{eqnarray}
where $\{\lambda_n,\psi_n({\textbf{x}})\}$ are the eigenpairs of a given covariance function, with the eigenvalues arranged in non-increasing order and the random variables $\{ Z_n\}_{n=1}^{N_{input}}$ are independent and identically distributed standard Gaussian variables. One recognizes that $\gamma(\textbf{x}, {\bm Z})$ is a truncated Karhunen-Lo\`{e}ve (KL) expansion corresponding to a correlated Gaussian random field with mean $\mu$ \cite{frauenfelder2005finite,li2008fourier, schevenels2004application}.

Here, we consider the specific covariance function
\begin{align}\label{covfunc}
C_{\gamma}(\textbf{x},{\textbf{x}}') = \sigma_{\gamma}^2 \exp\Big[-\frac{1}{L}\Big(\sum_{i=1}^d | x_i - {x}_i'|\Big)\Big],
\end{align}
where $\sigma_{\gamma}^2$ denotes a variance and $0<L \leq \mbox{diam}(D)$ a correlation length. The eigenpairs satisfy the generalized eigenvalue problem
\begin{align}\label{continuousEVP}
 \int_D C_{\gamma}(\textbf{x},{\textbf{x}'})\psi_n(\textbf{x}') d\textbf{x}' = \lambda_n \psi_n({\textbf{x}}).
\end{align}
The eigenpairs $\{\lambda_n,\psi_n\}$ are approximately determined by means of a Galerkin projection as in \cite{capodaglio2018approximation,gunzburger2014stochastic}.

The $N_{input}$ components of the input stochastic variable ${\bm Z}$ are independent and are all distributed according to a standard Gaussian PDF. Thus, in the spectral method discretization with respect to the input stochastic variable ${\bm Z}$, we use tensor products of Hermite polynomials as the orthonormal basis $\{\Psi_j({\mathbf Z})\}$. Due to the orthogonality of these polynomials with respect to the Gaussian PDF, this choice results in very substantial savings in both the assembly and solution aspects of the discretized SGM system. Details can be found in, e.g.,  \cite{capodaglio2018approximation,gunzburger2014stochastic}.

For the numerical tests, we choose the spatial domain $D$ to be the unit square $[0,1]^2$, the correlation length $L=0.1$, $\kappa_{\min}=0.01$, and $\mu = 0$.

\vskip5pt
\underline{\em Output of interest \eqref{QoI}.} Here, we consider the approximation \eqref{proposed_f} of the PDF $f_{output}(Y)$ of the standardized output of interest $Y$ given by \eqref{QoI}. 

We consider $\Gamma_{output}=[-5,3]$, and in \eqref{coeff} and \eqref{covfunc} we set $N_{input}=2$, $\sigma_{\gamma}=1.4$. Because ${\bm Z}=(Z_1,Z_2)$ with $Z_1$ and $Z_2$ being standard Gaussian variables, $\Gamma_{input}=(-\infty,\infty)^{N_{input}}$. A plot of the output of interest \eqref{QoI} as a function of the input variables ${\bm Z}=(Z_1,Z_2)$ is given, for ${\bm Z}\in[-3.5.3.5]^2$, in Figure \ref{qoi1} (left). Plots of the approximate output PDF $f_{\delta,M}(Y)$ determined using \eqref{proposed_f} is given in Figure \ref{trial}. For comparison purposes, a plot of the histogram approximation $f^{histo}_{\widehat{\delta},\widehat{M}}(Y)$ determined using \eqref{histo_def} is also provided in that figure, but with larger sample size $M$ and smaller bin size $\delta$ compared to those used for $f_{\delta,M}(Y)$.
 
\begin{figure}[h!]
\centerline{
\includegraphics[height=1.4in,width=1.6in]{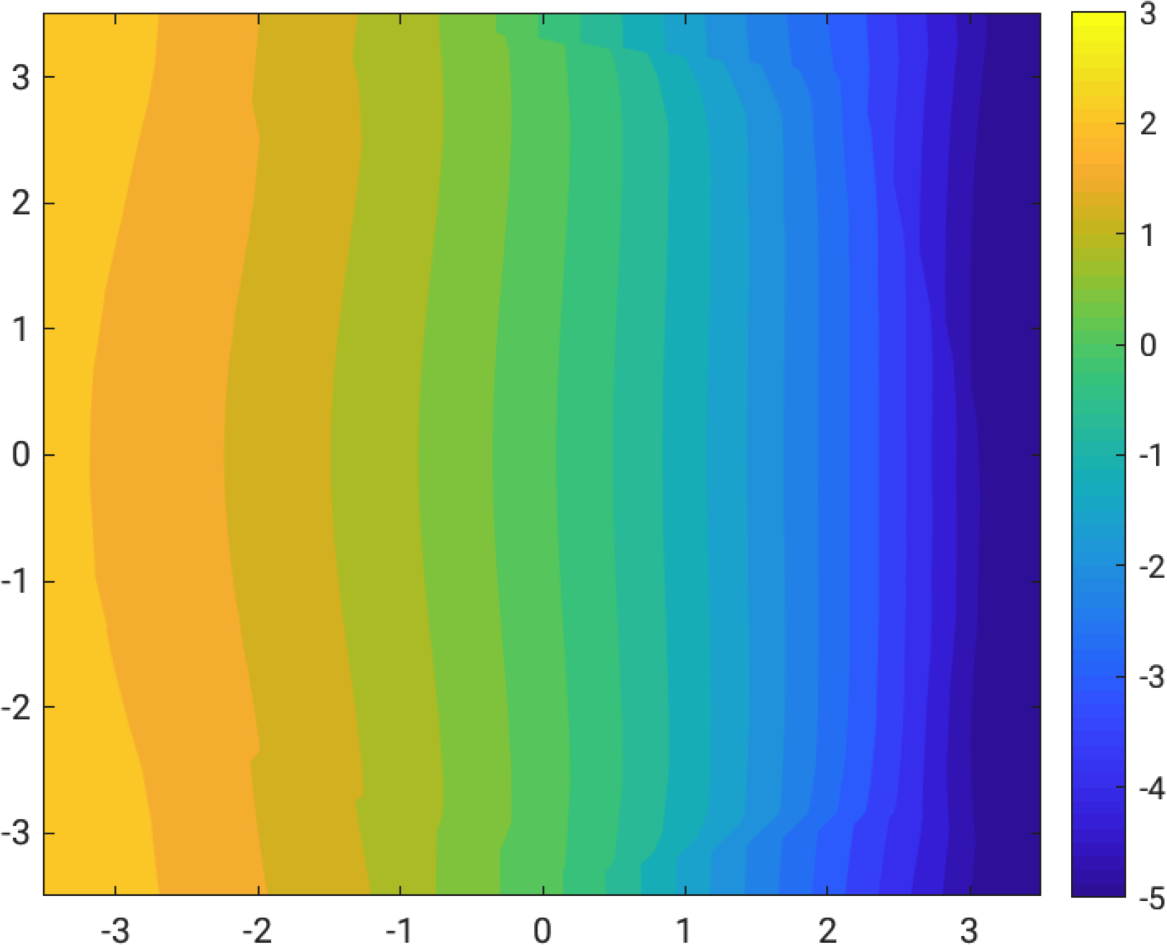}
\quad
\includegraphics[height=1.4in,width=1.6in]{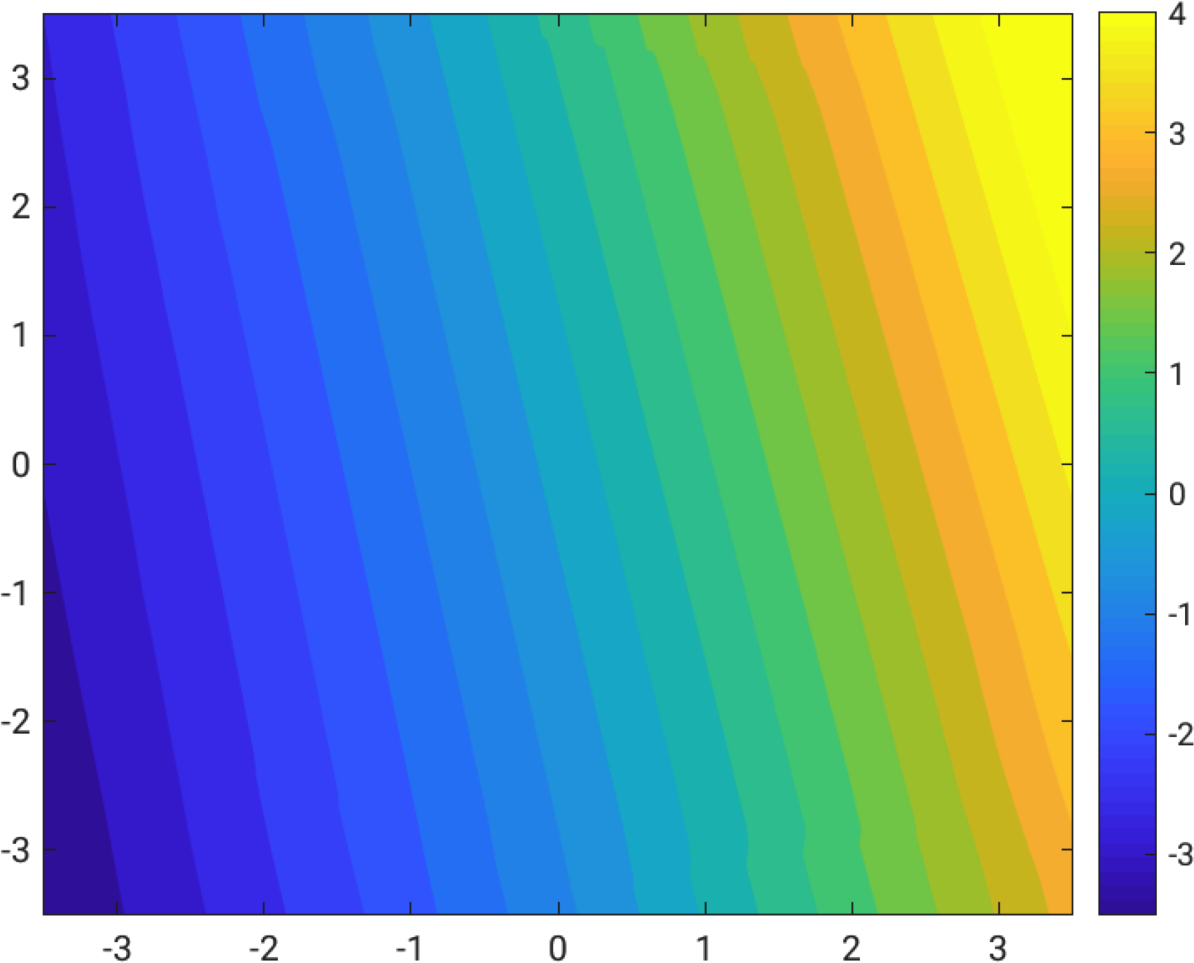}
} 
\caption{\em For the coefficient \eqref{coeff} with $N_{input}=2$, the outputs of interest \eqref{QoI} (left) and \eqref{QoI2} (right) as a function of the input variable ${\bm Z}=(Z_1,Z_2)$.}
\label{qoi1}
\end{figure}

\begin{figure}[h!]
\centerline{\includegraphics[scale=0.5]{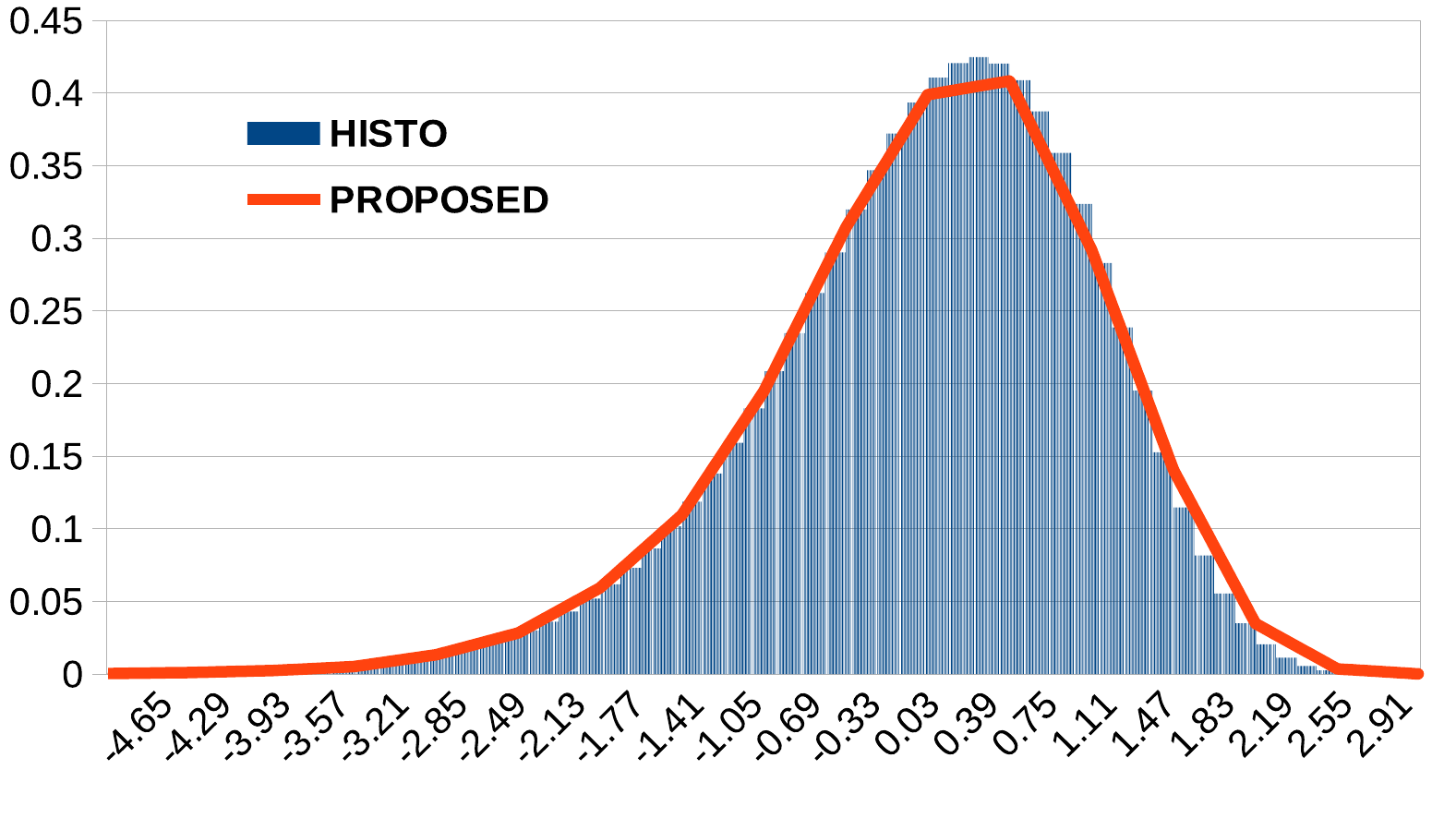}}

\centerline{\includegraphics[scale=0.5]{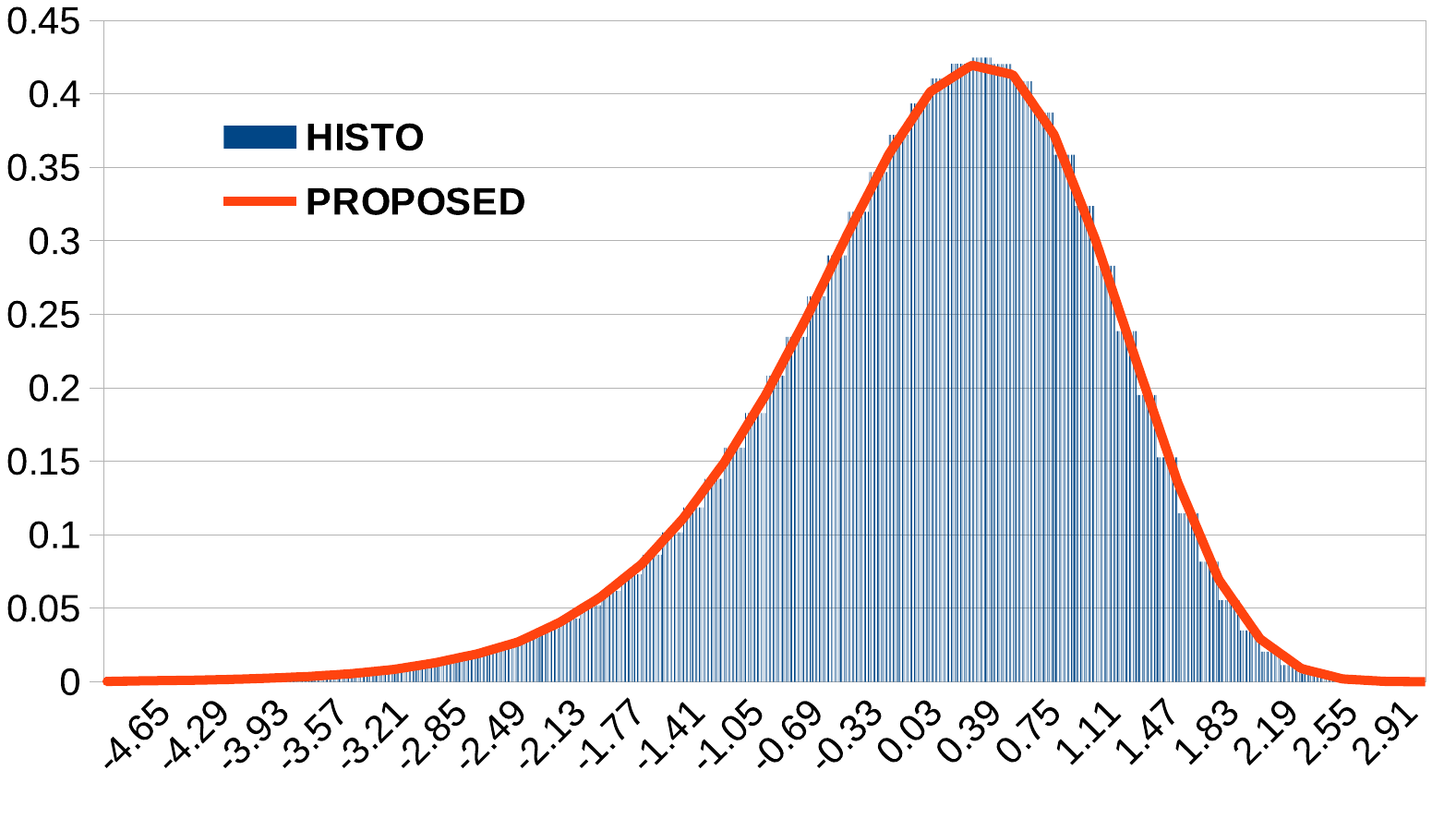}} 
\caption{\em For the coefficient \eqref{coeff} with $N_{input}=2$ and the output of interest \eqref{QoI}, a comparison between the histogram approximation \eqref{histo_def} of the PDF $f_{output}(Y)$ with $\widehat{\delta} = 0.125$ and  $\widehat{M}=16^6$ and the approximation \eqref{proposed_f} with $h=0.5$ and $M=16^4$ (top) and with $h=0.25$ and $M=16^5$ (bottom).}
\label{trial}
\end{figure}

We next examine the convergence behavior of the approximation \eqref{proposed_f} of the output PDF. Because the exact PDF is unknown, we measure the error using \eqref{al2e_histo} with the histogram surrogate $f^{histo}_{\widehat\delta,\widehat M}$ obtained with a bin size of $\widehat \delta = 8/2^8$ and $\widehat M=10^8$ samples. To study the order of convergence with respect to $\delta$, we choose
\begin{align}\label{h_order_SGM}
M = 10^7 \qquad \mbox{and} \qquad \delta=8/2^k, \qquad \mbox{for} \qquad  k=2,3,4,5,
\end{align}
whereas for the convergence with respect to $M$, we choose
\begin{align}\label{M_order_SGM}
\delta = 8/2^7 \qquad \mbox{and} \qquad M=10^k \qquad \mbox{for} \qquad  k=3,4,5,6. 
\end{align}
Note that for these values, we have $M<{\widehat M}$ and $\delta>{\widehat \delta}$. In Figure \ref{order_sgm}, we observe that  the convergence rates with respect to $\delta$ and $M$ are approximately $1.75$, and $0.45$, respectively. Given that, from examining Figure \ref{trial}, the output PDF seems to be $C^2(\Gamma_{output})$, these rates are lower than the values $2$ and $0.5$, respectively, that one might expect. Likely causes of these lower rates are that errors are determined by comparing to a histogram approximation and not to the an exact PDF and also because the values of $\widehat M$ and $\widehat \delta$ used for the histogram surrogate are ``close'' to the corresponding values used to determine the approximation \eqref{proposed_f}.

\begin{figure}[h!]
\centerline{
\includegraphics[scale=0.35]{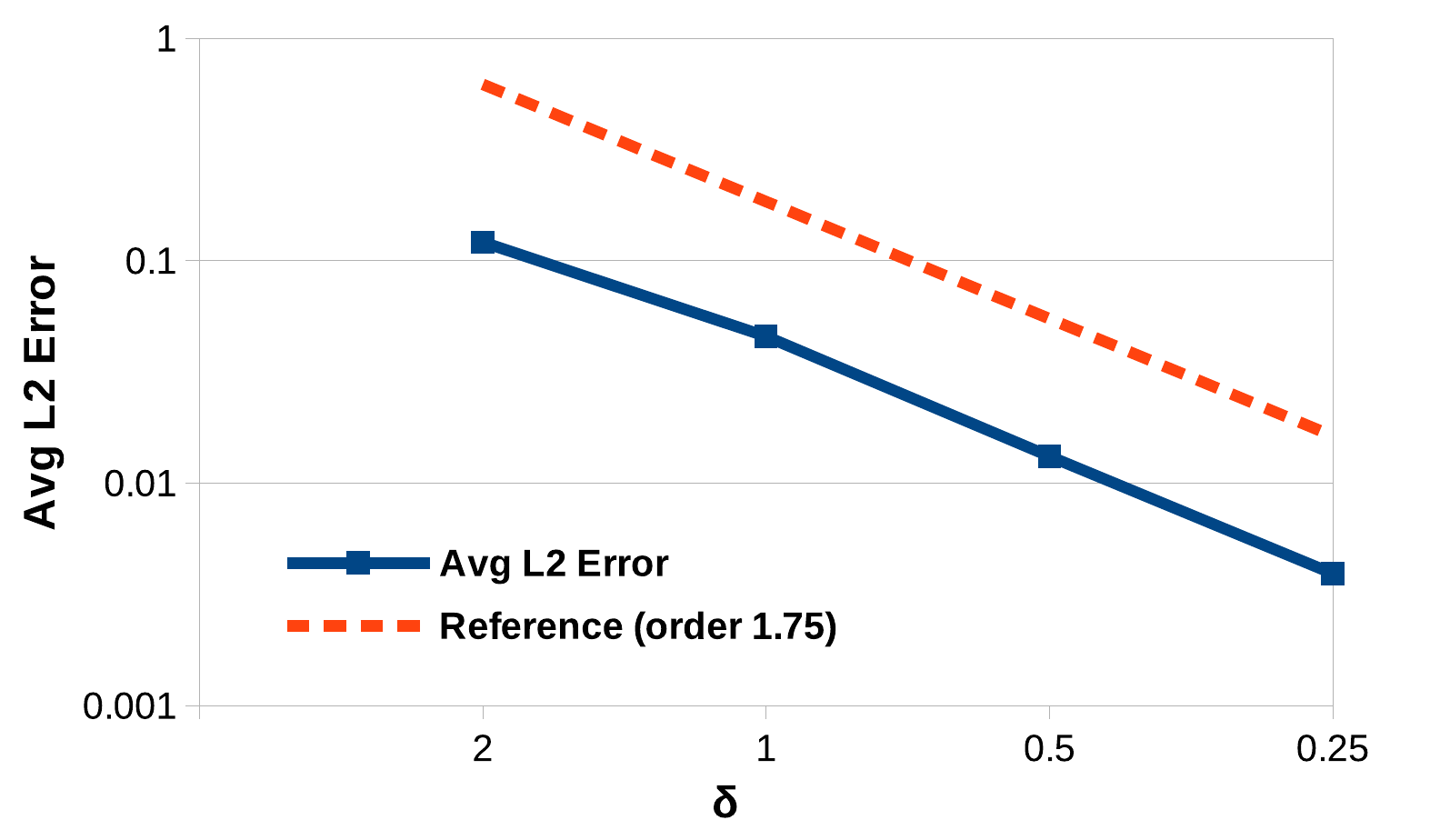} 
\quad
\includegraphics[scale=0.35]{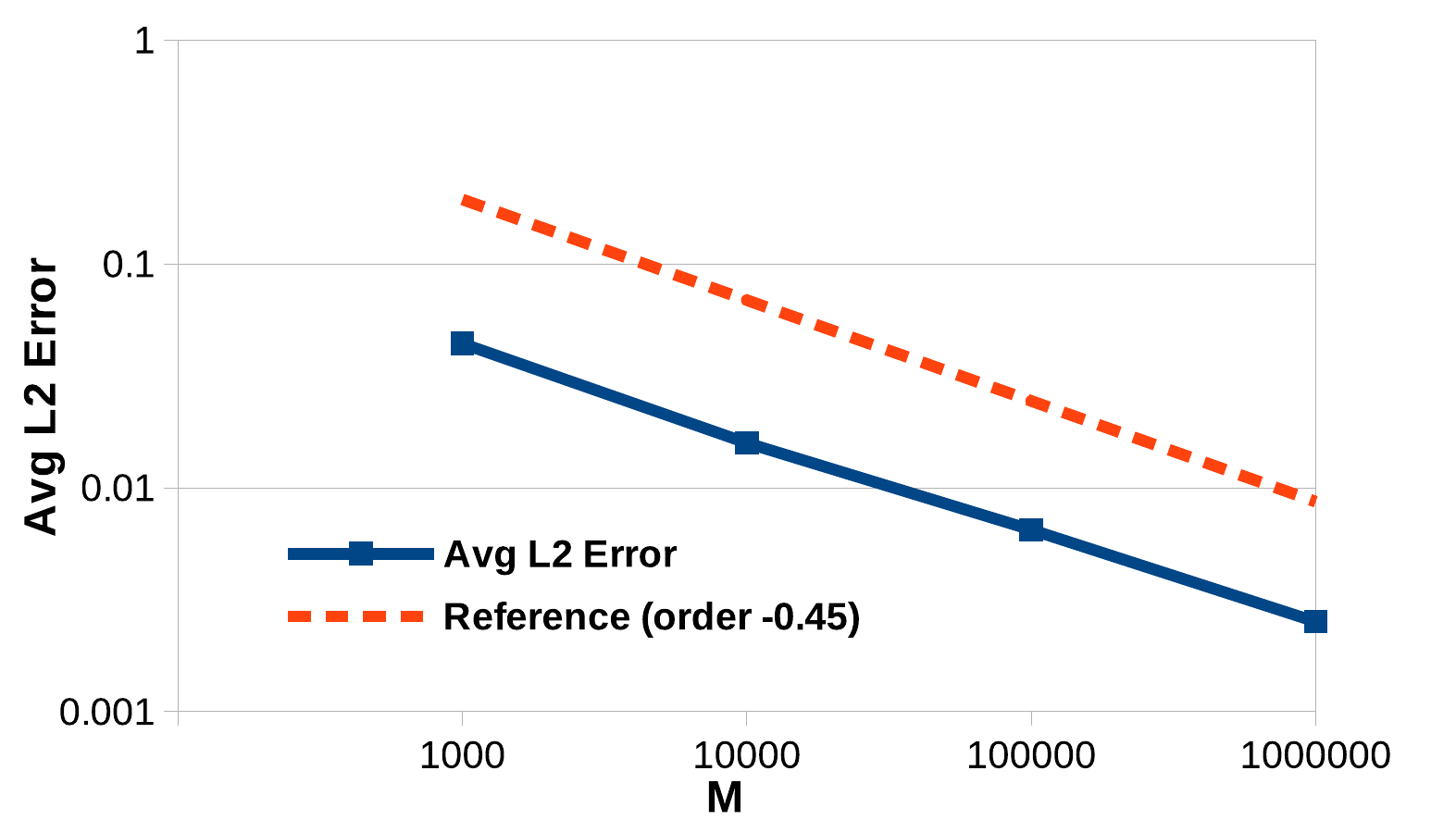}
}
 \caption{\em Errors and convergence rates for the for the approximation \eqref{proposed_f} to the output PDF $f_{output}(Y)$ for the output of interest \eqref{QoI}. Left: convergence with respect to $\delta$ with $M =10^7$ fixed. Right: convergence with respect to $M$ with $\delta = 0.0625$ fixed.}
\label{order_sgm}
\end{figure}

Further results about errors and convergence rates are given in Figure \ref{errSGM} for which \eqref{ns1} with $r=2$ is used to relate $\delta$ and $M$. The histogram used for comparison to estimate errors is obtained with $\widehat\delta = 0.125$ and  $\widehat M=16^6$ samples. The results in this figure are consistent with those of Figure \ref{order_sgm}. 

\begin{figure}[h!]
\centerline{
\includegraphics[scale=0.35]{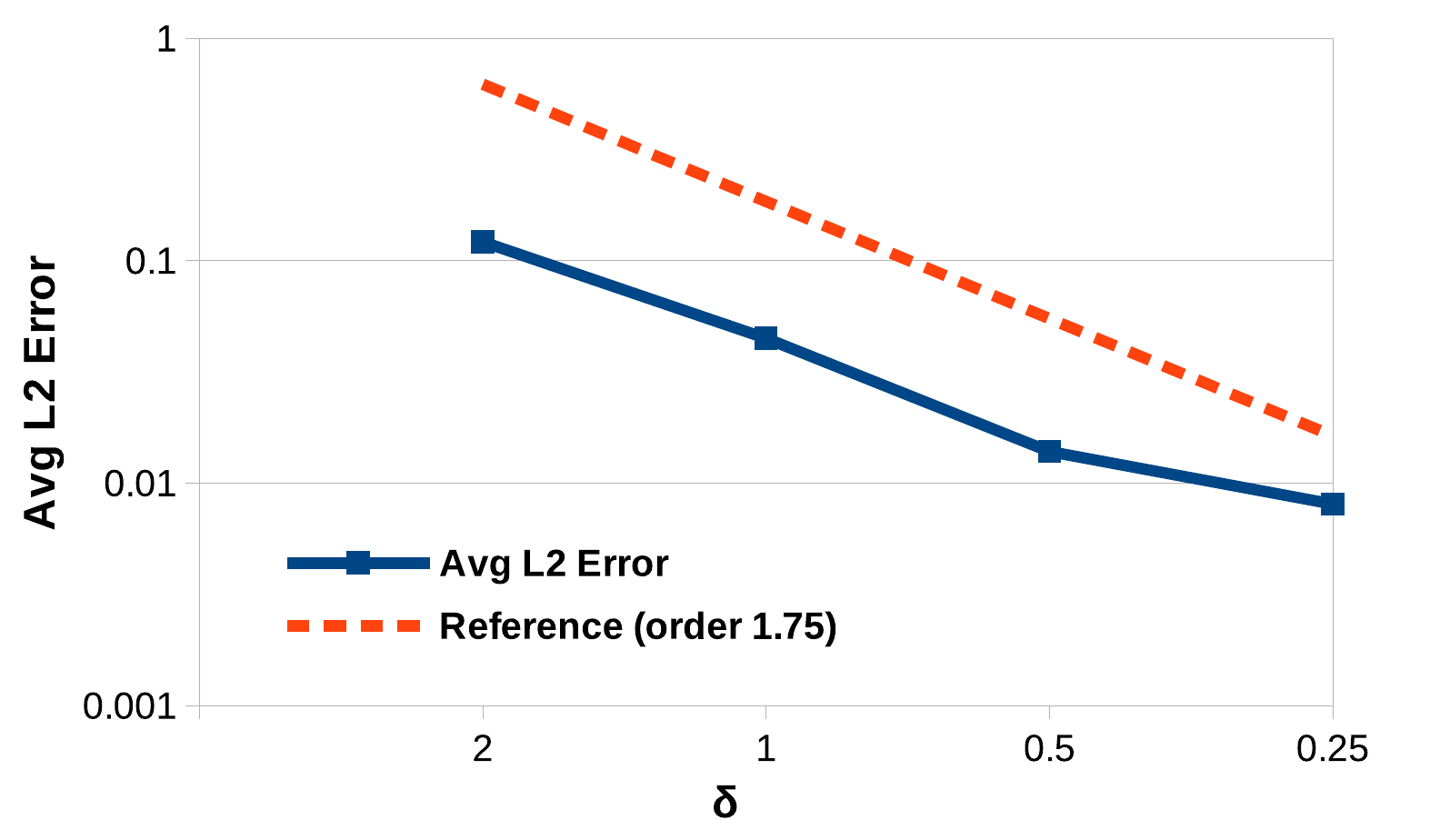}
\quad 
\includegraphics[scale=0.35]{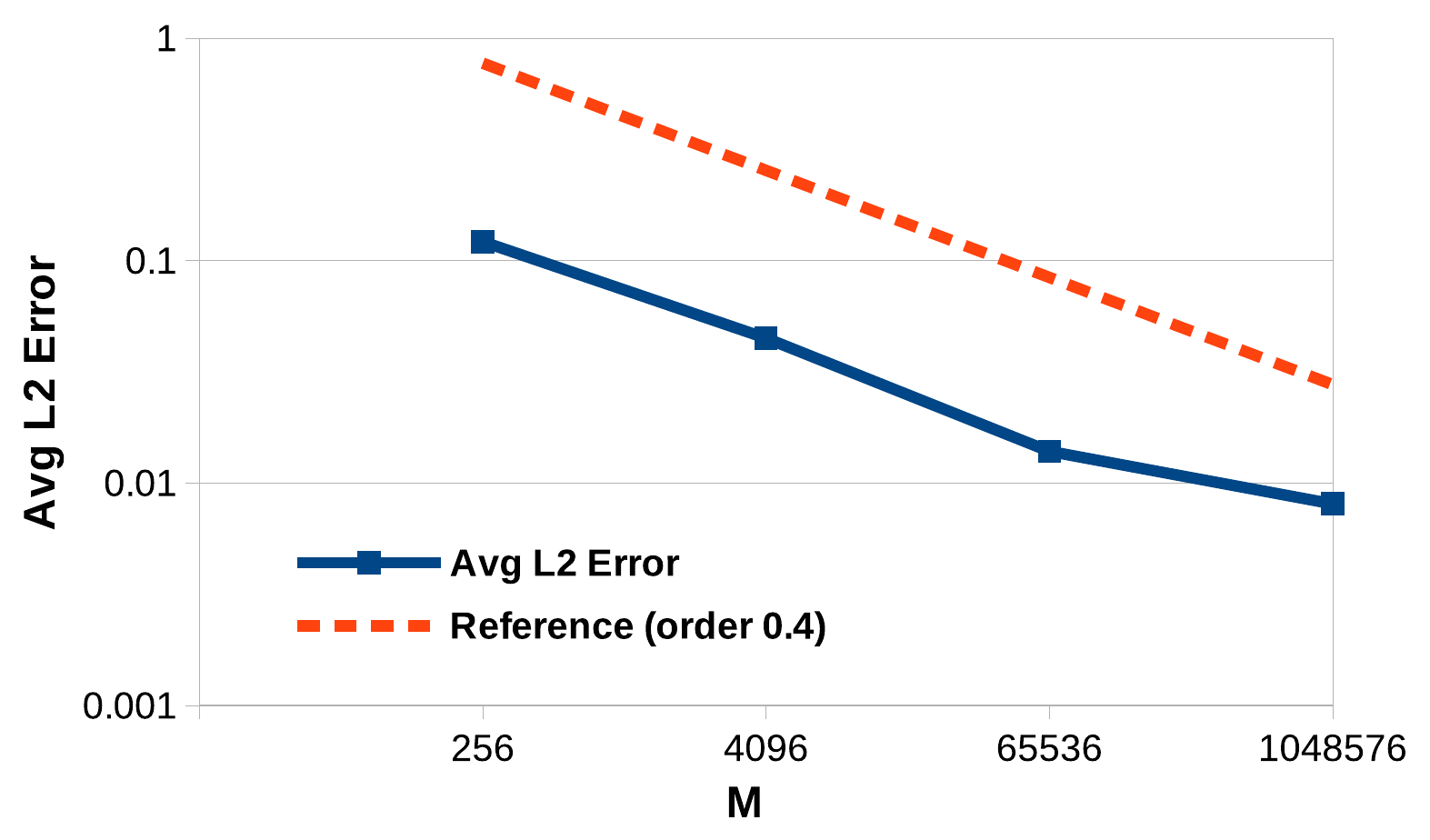}
} 
\caption{\em Errors and convergence rates for the for the approximation \eqref{proposed_f} to the output PDF $f_{output}(Y)$ for the output of interest \eqref{QoI} with $\delta$ and $M$ related through \eqref{ns1} with $r=2$.}
\label{errSGM}
\end{figure}

\vskip5pt
\underline{\em Output of interest \eqref{QoI2}.}
We next consider the approximation \eqref{proposed_f} of the PDF $f_{output}(Y)$ of the standardized output of interest $Y$ given by \eqref{QoI2}. The various inputs are the same as those used for the output of interest \eqref{QoI} except that now $\Gamma_{output} = [-3.5,4]$ and $\sigma_{\gamma}=2$. A plot of this output of interest as a function of the input random variables ${\bm Z}\in[-3.5.3.5]^2$ is given in the right plot of Figure \ref{qoi1}.
Plots of the approximate output PDF $f_{\delta,M}(Y)$ determined using \eqref{proposed_f} is given in Figure \ref{trial2}. For comparison purposes, plots of the histogram approximation $f^{histo}_{\widehat{\delta},\widehat{M}}(Y)$ determined using \eqref{histo_def} are also provided in that figure, but with larger sample size $M$ and smaller bin size $\delta$ compared to those used for $f_{\delta,M}(Y)$. Figure \eqref{errSGM2} provides plots of the errors in the approximation \eqref{proposed_f} determined through comparisons with histogram approximations determined with $\widehat\delta=0.1171875$ and $\widehat M=16^6$. Convergence rates of $1.75$ and $0.4$ are observed with respect to $\delta$ and $M$, respectively.  

\begin{figure}[h!]
\centerline{\includegraphics[scale=0.5]{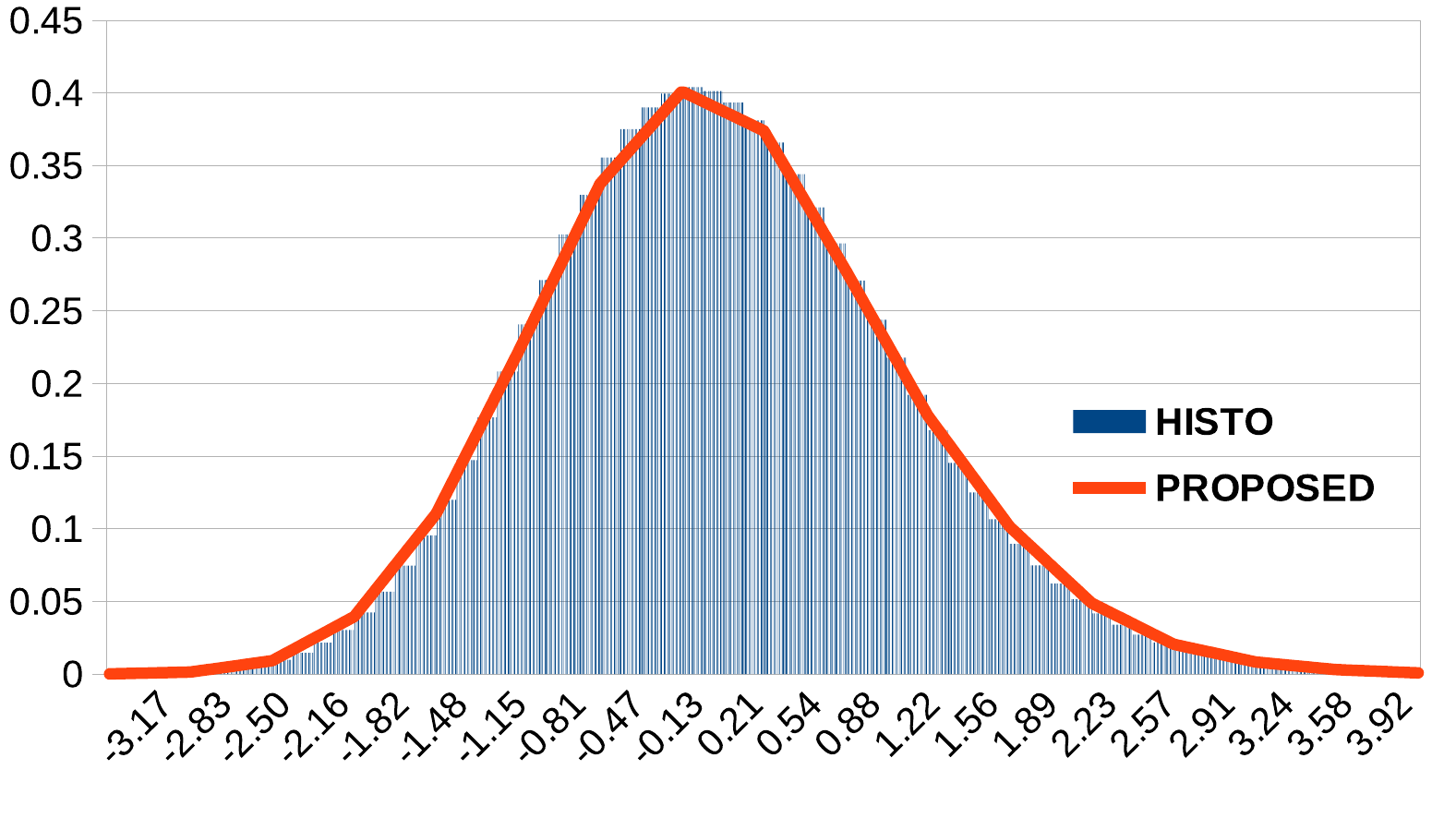}}

\centerline{\includegraphics[scale=0.5]{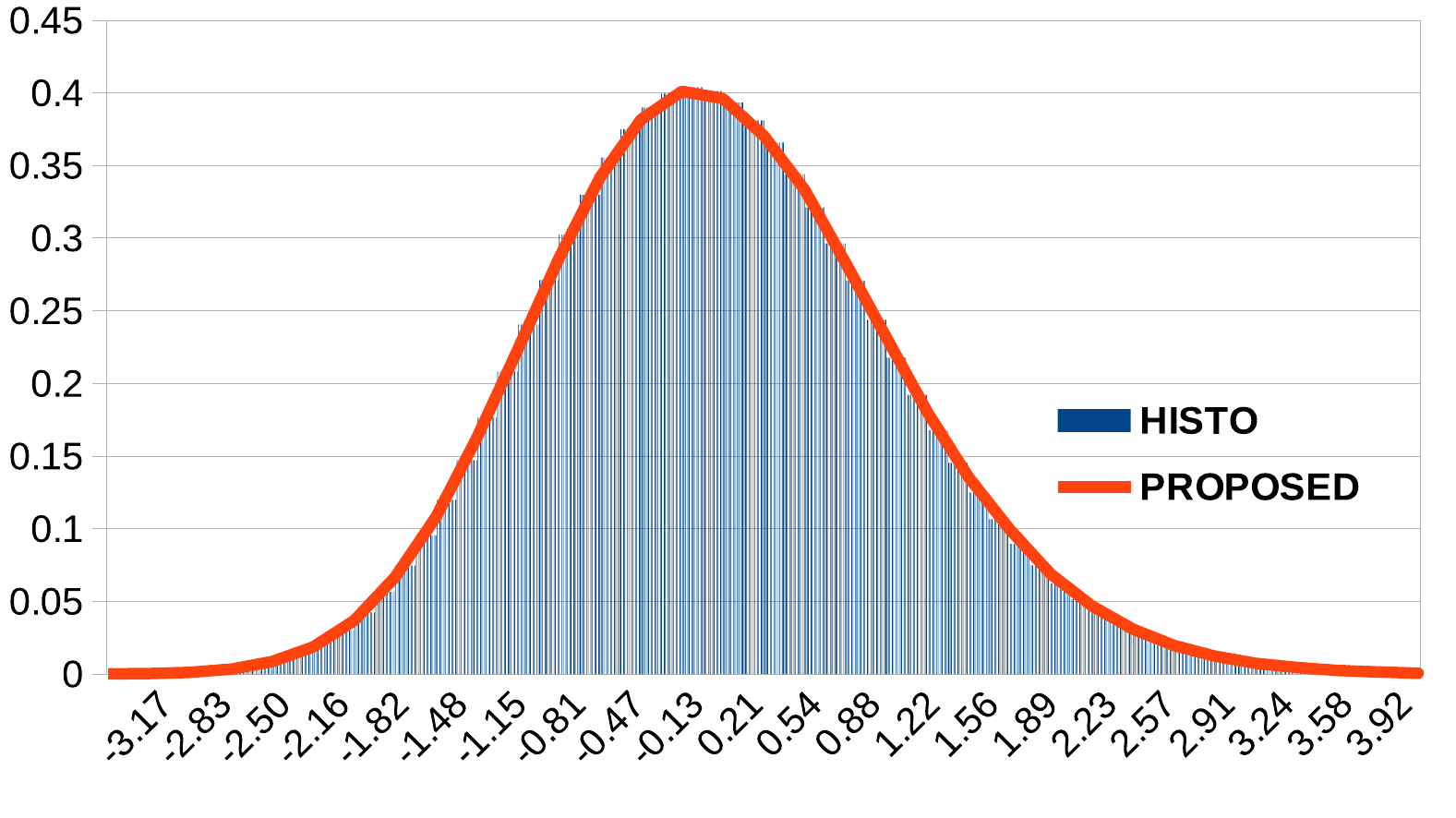}} 
\caption{\em For the coefficient \eqref{coeff} with $N_{input}=2$ and the output of interest \eqref{QoI2}, a comparison between the histogram approximation \eqref{histo_def} of the PDF $f_{output}(Y)$ with $\widehat\delta = 0.1171875$ and  $\widehat M=16^6$ and the approximation \eqref{proposed_f} with $\delta=0.46875$ and $M=16^4$ (top) and with $\delta=0.234375$ and $M=16^5$ (bottom).}
\label{trial2}
\end{figure}

\begin{figure}[h!]
\centerline{
\includegraphics[scale=0.35]{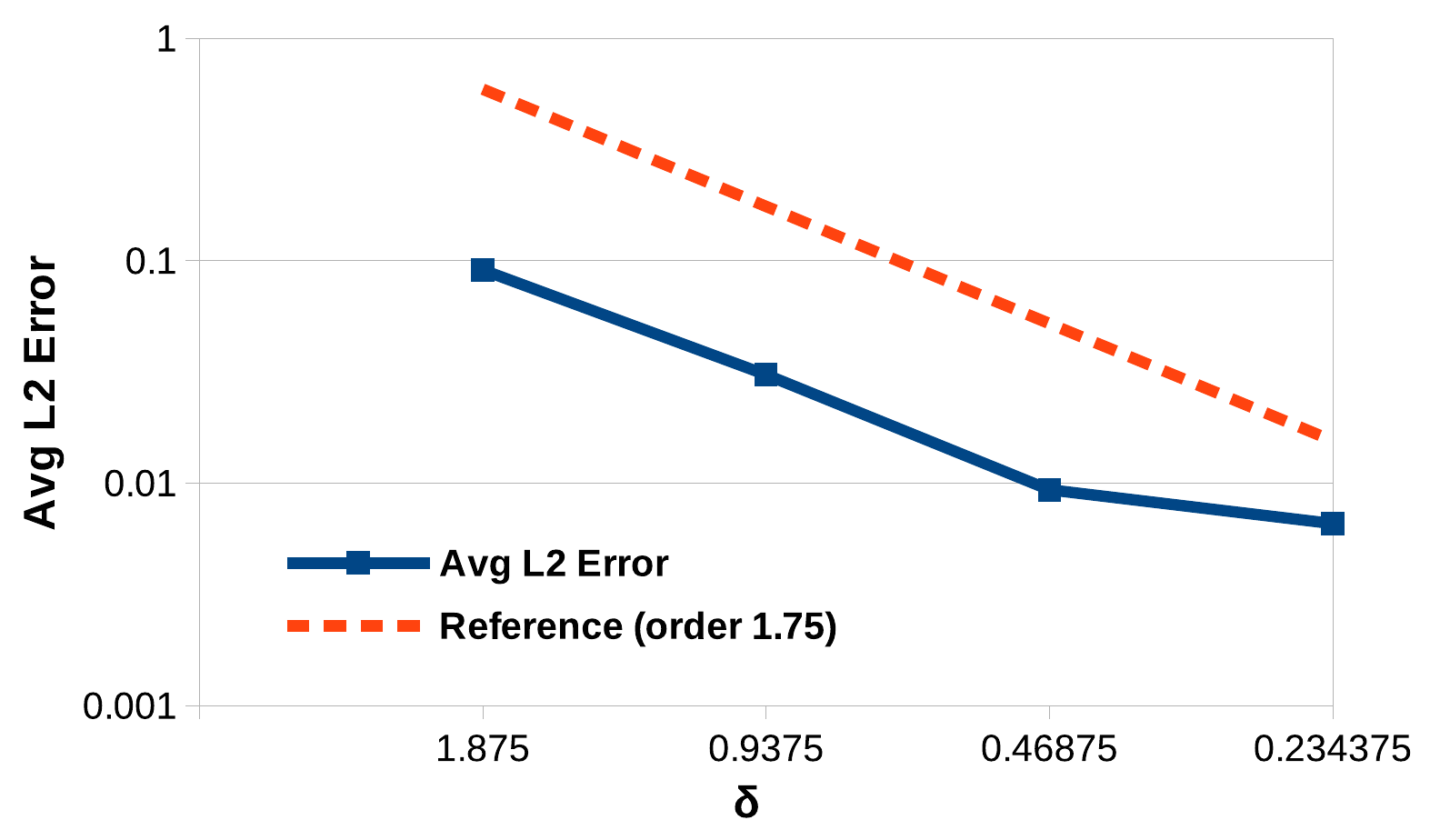}
\quad 
\includegraphics[scale=0.35]{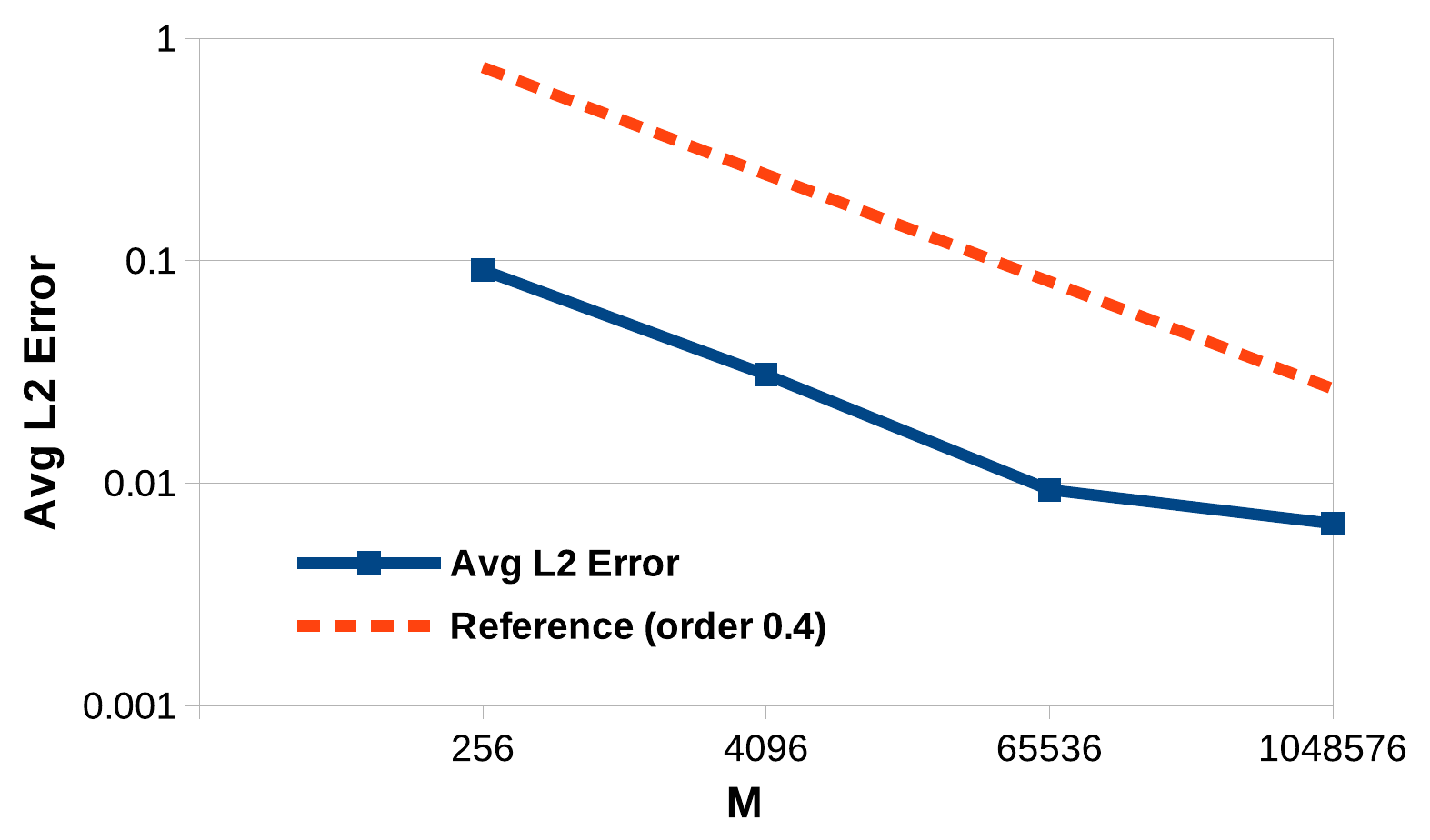}
} 
\caption{\em Errors and convergence rates for the for the approximation \eqref{proposed_f} to the output PDF $f_{output}(Y)$ for the output of interest \eqref{QoI2} with $\delta$ and $M$ related through \eqref{ns1} with $r=2$.}
\label{errSGM2}
\end{figure}

\section{Concluding remarks}\label{conclusions}

A {piecewise-linear density estimation method} \eqref{proposed_f} for the approximation of the PDF associated with a given set of samples is presented. The approximation is {naturally} scalable with respect to the sample size, is intrinsically positive, and has a unitary integral. It is also consistent, meaning that it converges to the exact PDF in the $L^2$ norm if the sample size goes to infinity and the bin size goes to zero. The construction of the approximation does not require the solution of a linear system and is fast even for a large number of samples. The computational time has been shown to scale linearly with the sample size. {Moreover, the binning size is the only smoothing parameter involved, and it can be related to the sample size by a simple rule.}

Future work would involve {strategies to extend the method to higher dimensions and decrease the computational time. For instance, the finite element basis functions may be replaced with sparse-grid basis functions in the way done in \cite{peherstorfer2014density}, in order to achieve scalability also with respect to the sample dimension.}  Computational speed-ups could be obtained by considering parallelization. Another means to speed up the computation would be to use adaptive refinement to coarsen the binning subdivision near the tails of the distribution. With this technique, fewer bins would have to be checked by the point locating algorithm of \cite{capodaglio2017particle} that we employ to locate a given point in the binning subdivision shared by several processors. In addition, the Monte Carlo sampling used in our method can be replaced by, e.g., quasi-Monte Carlo or sparse-grid sampling, resulting in efficiency gains.

\bibliographystyle{plain}

\end{document}